\documentclass[smallextended]{svjour3}       
\usepackage{graphicx}
\usepackage{amsmath,amssymb,graphicx,subfigure,float,url,mathrsfs}
\usepackage{pdfsync}
\usepackage[usenames]{color}
\usepackage{fullpage}
\usepackage[english]{babel}
\usepackage{csquotes}
\usepackage{oldgerm}
\usepackage{algorithmic}
\usepackage{algorithm}
\usepackage{hyperref}
\usepackage{enumitem}

\newfont{\rsfsten}{rsfs10 scaled 1200}
\DeclareMathAlphabet{\mathcal}{OMS}{cmsy}{m}{n} 

\newcommand{\N} {\mathbb{N}}

\newcommand{\C} {\mathbb{C}}
\newcommand{\R} {\mathbb{R}}

\renewcommand{\geq}{\geqslant}
\renewcommand{\leq}{\leqslant}

\newcommand{\lone}{\ell_1}
\newcommand{\ltwo}{\ell_2}
\newcommand{\ltau}{\ell_{\tau}}

\newcommand{\sJ}{\mathcal{J}_{\tau}}
\newcommand{\sF}[1]{\mathcal{F}_{#1}(y)}
\newcommand{\defleft}{\mathrel{\mathop:}=}

\newcommand{\xtil}{\tilde{x}}
\newcommand{\wtil}{{w}}
\newcommand{\etil}{{\varepsilon}}
\newcommand{\vectornorm}[1]{\left\|#1\right\|}
\newcommand{\tol}{\textnormal{tol}}

\newcommand{\HR}[1]{{#1}}
\newcommand{\HRomit}{}

\newcommand{\Rev}[1]{{\color{black}{#1}}}

\newcommand{\figurespath}[1]{./#1-eps-converted-to}

\DeclareMathOperator*{\argmin}{arg\,min\,}
\DeclareMathOperator*{\supp}{supp}
\DeclareMathOperator*{\sign}{sign}
\DeclareMathOperator*{\diag}{diag}
\DeclareMathOperator*{\rank}{rank}
\DeclareMathOperator*{\subspace}{span}

\journalname{Computational Optimization and Applications}
\begin{document}

\title{Conjugate gradient acceleration of iteratively re-weighted least squares methods}


\author{Massimo Fornasier \and
Steffen Peter \and
Holger Rauhut \and
Stephan Worm
}


\institute{M. Fornasier \at 
Technische Universit\"at M\"unchen, Fakult\"at f{\"u}r Mathematik, Boltzmannstrasse 3, D-85748, Garching bei M\"unchen, Germany \\ 
\email{massimo.fornasier@ma.tum.de}
\and 
S. Peter \at
Technische Universit\"at M\"unchen, Fakult\"at f{\"u}r Mathematik, Boltzmannstrasse 3, D-85748, Garching bei M\"unchen, Germany \\
Tel.: +49-89-289-17482\\
\email{steffen.peter@ma.tum.de}
\and 
H. Rauhut \at 
RWTH Aachen University, Lehrstuhl C f\"ur Mathematik (Analysis), Pontdriesch 10,
D-52062, Aachen, Germany \\
\email{rauhut@mathc.rwth-aachen.de}
\and
S. Worm \at
Schlo{\ss}str. 34, D-53115 Bonn, Germany \\
\email{stephanworm@gmx.de}
}

\date{Received: date / Accepted: date}

\maketitle

\begin{abstract}
Iteratively Re-weighted Least Squares (IRLS) is a method for solving minimization problems involving non-quadratic cost functions, perhaps non-convex and non-smooth, which however can be described as the infimum over a family of quadratic functions. 
This transformation suggests an algorithmic scheme that solves a sequence of quadratic problems to be tackled efficiently by tools of numerical linear algebra. 
Its general scope and its usually simple implementation, transforming the initial non-convex and non-smooth minimization problem  into a more familiar and easily solvable quadratic optimization problem, make it a
versatile algorithm. 
However, despite its simplicity, versatility, and elegant analysis,  the complexity of IRLS 
 strongly depends on the way the solution of the successive quadratic optimizations is addressed. For the important special case of {\it compressed sensing} and sparse recovery problems in signal processing, we investigate theoretically and numerically how accurately one needs to solve the quadratic problems by means of the {\it conjugate gradient} (CG) method in each iteration in order to
 guarantee convergence. The use of the CG method may significantly speed-up the numerical solution of the quadratic subproblems, in particular,
 when fast matrix-vector multiplication (exploiting for instance the FFT) is available for the matrix involved.
In addition, we study convergence rates.
 Our modified IRLS method outperforms state of the art first order methods such as Iterative Hard Thresholding (IHT) or Fast Iterative Soft-Thresholding Algorithm (FISTA) in many situations, 
 especially in large dimensions. Moreover, IRLS is often able to recover sparse vectors from fewer measurements than required for IHT and FISTA. 
\keywords{Iteratively re-weighted least squares \and conjugate gradient method \and $\ltau$-norm minimization \and compressed sensing \and sparse recovery.}
\end{abstract}

\newpage
\section{Introduction}
\subsection{Iteratively Re-weighted Least Squares}

Iteratively Re-weighted Least Squares (IRLS) is a method for solving minimization problems
by transforming them into a sequence 
of easier quadratic problems which are then solved with efficient tools of numerical linear algebra. 
Contrary to classical Newton methods \HRomit smoothness of the objective function is not required in general. 
We refer to the recent paper \cite{OBD14} for an updated and rather general view about these methods.

\HR{In the context of constructive approximation}, 
an IRLS algorithm appeared for the first time in the doctoral thesis of Lawson in 1961 \cite{LW61} in the form of an algorithm for solving uniform approximation problems.
\HR{It computes a sequence of polynomials that minimize a sequence of weighted $L_\tau$--norms.}
This iterative algorithm is now well-known in classical approximation theory as Lawson's algorithm. In \cite{Cline72} it is proved that 
this algorithm \HR{essentially obeys} a linear convergence rate. 

In the 1970s extensions of Lawson's algorithm for $\ell_\tau$-norm minimization, and in particular $\ell_1$-norm minimization, were proposed. Since 
\HR{then} IRLS has become a rather popular method \HR{also} in mathematical statistics for robust linear regression \cite{HoWe77}. Perhaps the most comprehensive mathematical analysis of the performance of IRLS for $\ell_\tau$-norm minimization was given in the work of Osborne \cite{Osborne85}. 

The increased popularity of total variation minimization in image processing \HR{starting} with the pioneering work \cite{faosru92}, significantly revitalized the interest \HRomit 
in these algorithms, because of their \HRomit 
simple and intuitive implementation, contrary to more general optimization algorithms such as interior point methods. 
In particular, in \cite{ChambolleLions97,VO98} an IRLS \HR{for} 
total variation minimization has been proposed. 
\HRomit At the same time, IRLS appear\HR{ed} as well under the name of {\it {K}a\v{c}anov method} in \cite{hajesh97} as a fixed point iteration for the solution of certain quasi-linear elliptic partial differential equations.
In signal processing, IRLS was used as a technique to build algorithms for sparse signal reconstruction in \cite{GoRa}. 
\HR{A}fter the pioneering work \cite{chdosa99} and the starting of the development of {\it compressed sensing} with the seminal papers \cite{carota06,do06-2}, several works \cite{Chartrand07,Chartrand08,ChartrandYin08,dadefogu10} addressed systematically the analysis of IRLS  for $\ell_\tau$-norm minimization in the form 
 \begin{equation}
 \label{eq:ltauproblem}
\min\limits_{\Phi x= y} \|x\|_{\ltau},
\end{equation}
where $0 < \tau \leq 1$, $\Phi \in \mathbb C^{m\times N}$ is a given matrix, and $y\in \C^m$ a given measurement vector. 
In these papers, the asymptotic super-linear convergence of IRLS towards $\ell_\tau$-norm minimization for $\tau<1$ has been shown. As an extension of the analysis of the aforementioned papers, IRLS have been also generalized towards low-rank matrix recovery from minimal linear measurements \cite{FornasierRauhutWard11}.

In recent years, there has been an explosion of papers on applications and variations on the theme \HR{of} IRLS, especially in the engineering community of signal processing, and it is by now \HR{almost} impossible to give a complete account of the developments. (Presently Scholar Google reports more than 3180 papers since 2010 containing the \HR{phrase} ``Iteratively Re-weighted Least Squares'' and more than 100 with it in the title since 1970, half of which appeared after 2003.)

\subsection{Contribution of this paper}

Since it is based on a relatively simple reformulation of the initial \HR{potentially} non-convex and non-smooth minimization problem (for instance  of the type \eqref{eq:ltauproblem}) into a more familiar and easily solvable quadratic optimization, IRLS is one of the most immediate and intuitive approaches towards such non-standard optimizations and perhaps one of the first and popular algorithms beginner practitioners consider for their first experiments.
However, despite its simplicity, versatility, and elegant analysis, IRLS does not outperform in general well-established first order methods, which have been proposed recently for similar problems, such as Iterative Hard Thresholding (IHT)~\cite{Blumensath09} or Fast Iterative Soft-Thresholding Algorithm (FISTA)~\cite{beck09}, as we also show in our numerical experiments in Section \ref{sec:numerics}. In fact, its complexity very strongly depends on the way the solution of the successive quadratic optimizations is addressed, whether one uses preconditioned iterative methods and exploits fast matrix-vector multiplications or just considers simple direct linear solvers.
If the dimensions of the problem are not too large or the involved matrices have no special structure allowing for fast matrix-vector multiplications, \HRomit
then the use of a direct method such as Gaussian elimination can be appropriate. When instead the dimension of the problem is large and one can take advantage of the structure of the matrix to perform fast matrix-vector multiplications \HR{(e.g., for partial Fourier or partial circulant matrices)}, 
then it \HR{is appropriate} to use iterative solvers such as the Conjugate Gradient method (CG).
The use of CG in the implementation of IRLS is appearing, for instance, in \cite{VO98} towards total variation minimization and in \cite{sergei,davo15} towards $\ell_1$-norm minimization.  However, the price to pay is that such solvers will return only an approximate solution whose precision depends on the number of iterations\HR{. A} proper analysis of the convergence of the perturbed method in this case has not been reported in the literature. Without such an analysis it is impossible to give any estimate of the actual complexity of IRLS. Thus, the scope of this work is to clarify, specifically for compressed sensing problems (i.e., for matrices $\Phi$ with certain spectral properties such as the Null Space Property), how accurately one needs to solve the quadratic problems by means of CG \HR{in order} to guarantee convergence and 
possibly also asymptotic \HR{(super-)}linear \HR{convergence} rates.

Besides analyzing the effect of CG in an IRLS for problems of the type \eqref{eq:ltauproblem}, we further extend it in Section~\ref{sec:IRLSgenltau} to a class of problems of the type
\begin{equation}
\label{eq:l1lagrangian}
\min\limits_x\vectornorm{\Phi x - y }_{\ltwo}^2 + 2\alpha \vectornorm{x}_{\ltau},
\end{equation}
for $0 < \tau \leq 1$, used for sparse recovery in signal processing. In the work \cite{xulaiyin,sergei,davo15} a convergence analysis of IRLS towards the solution of \eqref{eq:l1lagrangian} has been carried out with two limitations:
\begin{enumerate}
\item In \cite{xulaiyin} the authors do not consider the use of an iterative algorithm to solve the appearing system of linear equations and they do not show the behavior of the algorithm when the measurements $y$ are given with additional noise; 
\item Also in \cite{sergei,davo15}  a precise analysis of convergence is \HR{missing} when iterative methods are used to solve the intermediate sequence of systems of linear equations. Also the non-convex case of $\tau<1$ is not specifically addressed.
\end{enumerate}
Regarding these gaps, we contribute in this work by
\begin{itemize}
\item giving a proper analysis of the convergence when inaccurate CG solutions are used;
\item extending the results of convergence in~\cite{sergei,davo15} to the case of $0 < \tau<1$ by combining our analysis with findings in~\cite{RamlauZarzer12,Zarzer09};
\item performing \HRomit 
numerical tests which 
\HR{evaluate} possible speedups 
\HR{via the CG method}, also taking problems into consideration 
where measurements may be affected by noise.
\end{itemize}
\HR{Our work on CG accelerated IRLS for \eqref{eq:l1lagrangian}}
does not \HR{analytically address rates of convergence} 
because \HR{this} turned out to be a very technical task. \\

We illustrate the theoretical results of this paper described above by several numerical experiments. 
We first show that 
\HR{our versions} of IRLS 
yield significant  improvements in terms
of computational time 
\HR{and may outperform} state of the art first order methods such as Iterative Hard Thresholding (IHT)~\cite{Blumensath09} and Fast Iterative Soft-Thresholding Algorithm (FISTA)~\cite{beck09}, especially in high dimensional problems ($N \geq 10^5$). These results are somehow both surprising and counterintuitive as it is well-known that first order methods should be preferred in higher dimension. However, they can be easily explained by observing that in certain regimes preconditioning in the conjugate gradient method (as we show at the end of Subsection \ref{CGIRLSl}) turns out to be extremely efficient. This is perhaps not a completely new discovery, as benefits of preconditioning in IRLS have been reported already in minimization problems involving total variation terms \cite{VO98}.
The second significant outcome of our experiments is that  CG-IRLS not only is faster than state of the art first order methods, but 
\HR{also shows higher recovery rates, i.e., requires less measurements for successful sparse recovery}. This will be demonstrated 
\HR{with} 
corresponding phase transition diagrams of empirical success rates (Figure \ref{fig:PT_diagram}).

\subsection{Outline of the paper}
The paper is organized as follows: In Section~\ref{sec:definitions}, we introduce definitions and notation and give a short review on the CG method. Although this brief introduction on CG retraces very well-known facts of the numerical linear algebra literature, it is necessary for us for the sake of a consistent presentation also in terms of notation. We hope that this small detour will help readers to access more easily the \HRomit 
technical parts of the paper. In Section \ref{sec:IRLSltau}, we present the IRLS method tailored to problems of the type~\eqref{eq:ltauproblem} and its modification including CG for the solution of the quadratic optimizations. We present a detailed analysis of the convergence and rate of convergence. The approach is further extended to problems of \HRomit 
type~\eqref{eq:l1lagrangian} in Section~\ref{sec:IRLSgenltau}, where we also analyze the convergence of the method. 
\HR{We conclude with} numerical experiments in Section~\ref{sec:numerics} showing that the modifications to IRLS inspired by our theoretical results make the algorithm extremely efficient, also compared to state of the art first order methods, especially in high dimension.


\section{Definitions, Notation, and Conjugate Gradient method}
\label{sec:definitions}
In this section, we introduce the main terms and notation used in this paper. In addition to this, we shortly review the basics around the Conjugate Gradient method. \Rev{For a  more detailed introduction to conjugate gradient methods, we refer to respective text books, e.g.,~\cite{nowr2006,sarisa00}}. In order to simplify cross-reading, we use the same notation as in \cite{dadefogu10}.

For matrices $\Phi\in\C^{m\times N}$ and $y\in\C^{m}$, we define
\begin{eqnarray}
 \sF{\Phi} &\defleft& \left\{ z \in \C^{N} \hspace{1mm} | \hspace{1mm} \Phi z = y\right\}, \label{def: sF} \\
 \mathcal{N}_{\Phi} &\defleft& \ker \Phi = \left\{ z \in \C^{N} \hspace{1mm} | \hspace{1mm} \Phi z = 0\right\}. \label{def: N}
\end{eqnarray}
\HR{Unless noted otherwise}, we denote with $\Phi^*$ the adjoint (conjugate transpose) matrix of a matrix $\Phi$. Thus, in the particular case of a scalar, $x^*$ denotes the complex conjugate of $x\in\C$.
\begin{definition}[Weighted $\ell_p$-spaces]
We define the quasi-Banach space $\ell_p^N(w)\defleft(\C^N,||\cdot||_{\ell_p(w)})$ endowed with the weighted quasi-norm
\[\vectornorm{x}_{\ell_p(w)}\defleft \left(\sum\limits_{i=1}^{N}|x_i|^{p}w_i\right)^{\frac1p},\]
for a weight vector $w\in\R^N$ with positive entries and $0 < p < \infty$. Furthermore, we define the $\ell_p^N$-spaces by setting $\ell_p^N\defleft\ell_p^N(\mathbf{1})$, where $\mathbf{1}$ 
\HR{denotes} the weight with entries identically set to $1$. Below we may ignore the superscript indicating the dimension $N$, when it is clear from the context, so that we write $\ell_p=\ell_p^N$ or $\ell_p(w) = \ell_p^N(w)$. \HR{The space $\ell_2^N(w)$ is a Hilbert space endowed with the weighted scalar product
\[
\langle x, y \rangle_{\ell_2(w)} = \sum_{i=1}^N x_i y_i^* w_i.
\]
In the unweighted case $w= \mathbf{1}$ it reduces to the standard  
complex scalar product $\langle\cdot,\cdot\rangle_{\ltwo}$.} 
\\
For $\Phi\in\C^{m\times N}$, we define the norm 
\[\vectornorm{\Phi}_{\ell_p^N \rightarrow\ell_q^m}\defleft\sup\limits_{\vectornorm{x}_{\ell_p^N} = 1}\vectornorm{\Phi x}_{\ell_q^m},\]
and for the particular case of $p= q = 2$, $\vectornorm{\Phi}:= \vectornorm{\Phi}_{\ell_2^N \rightarrow\ell_2^m}$ is the standard operator norm and can be given explicitly by 
\[\vectornorm{\Phi} = \sqrt{\lambda_{\max}(\Phi^*\Phi)},\]
where $\lambda_{\max}(\cdot)$ denotes the largest eigenvalue of a \HR{square} 
matrix (compare Definition~\ref{def:eig}).
\end{definition}

\begin{definition}[K-sparse vector]
A vector $x\in\C^N$ is called \emph{$K$-sparse} for $K \in \N$, $K \leq N$, if the number \HR{$\#\{\HR{i} |x_i\neq0\}$} of its non-zero entries does not exceed $K$. 
\end{definition}
\begin{definition}[Nonincreasing rearrangement]
\label{def:nir}
The \emph{nonincreasing rearrangement} $r(x)$ of the vector $x\in\C^N$ is defined by $r(x) := \left(|x_{i_1}|,\ldots,|x_{i_N}|\right)$ with $|x_{i_j}|\geq |x_{i_{j+1}}|$ for $j=1,\ldots,N-1$ \HR{and where $j \mapsto i_j$ is a permutation of $\{1,\hdots,N\}$}. Furthermore, the \emph{\HR{best $K$-term} approximation error} $\sigma_K(x)_{\ell_\tau}$ 
\HR{in $\ell_\tau$} is given by
\[\sigma_K(x)_{\ell_\HR{\tau}} \defleft \inf\limits_{z\in\C^N,\;K\text{-sparse}}\vectornorm{x-z}_{\ell_\HR{\tau}}^\HR{\tau}  = \sum\limits_{j=K+1}^{N}|r_j(x)|^\HR{\tau},\quad 0 <  \HR{\tau} < \infty.\] 
\end{definition}

In this paper we restrict our attention to optimization problems of the type \eqref{eq:ltauproblem} for matrices $\Phi \in \C^{m\times N}$ for $m \leq N$ \Rev{having full rank, i.e., $\rank(\Phi)=m$, and} certain spectral properties. Such matrices are used in the practice of {\it compressed sensing} and we refer to \cite{fora13} for more details. The following notion has been introduced in \HR{\cite{Chartrand07,Chartrand08,ChartrandYin08,grni03,codade09,dadefogu10}}.

\begin{definition}[Null Space Property (NSP)]
\label{def: NSP}
A matrix $\Phi \in\C^{m\times N}$ satisfies the \textit{Null Space Property} of order \HR{$K$} for $\gamma_\HR{K}>0$ and fixed $0 < \tau \leq 1 $ if
\begin{equation}\label{eq: NSP}
\left\|\eta_{T}\right\|_{\ltau}^{\tau} \le \gamma_\HR{K}\left\|\eta_{T^{c}}\right\|_{\ltau}^{\tau},
\end{equation}
for all sets $T \subseteq \{1,\ldots,N\}$ with $\#T \le \HR{K}$ and all $\eta\in\ker \Phi \backslash\{0\}$. We say in short that $\Phi$ has the $(\HR{K},\gamma_\HR{K})$-NSP.
\end{definition}

\HR{We give an important consequence of the NSP \cite{codade09,fora13}, \cite[Lemma 7.6]{dadefogu10}.}
\HR{
\begin{lemma}\label{lem:NSP:invtriangle} Assume that $\Phi \in\C^{m\times N}$ satisfies the $(K,\gamma_K)$-NSP 
for $0 < \tau \leq 1$. Then for
any vectors $z, z' \in \C^N$ it holds
\[
\|z' - z\|_{\ltau}^{\tau} \leq \frac{1+\gamma_K}{1-\gamma_K}\left( \|z'\|_{\ltau}^{\tau} - \|z\|_{\ltau}^{\tau} + 2 \sigma_K(z)_{\ell_\tau}\right).
\]
\end{lemma}
It follows immediately from this lemma that the solution $x^\sharp$ of $\ell_\tau$-minimization \eqref{eq:ltauproblem} run on $y = \Phi x$ satisfies 
$\| x^\sharp - x\|_{\ltau}^{\tau} \leq \frac{2(1+\gamma_K)}{1-\gamma_K}\sigma_K(z)_{\ell_\tau}$. Another consequence is the following statement, 
see \cite[Lemma 4.3]{dadefogu10} for the case $\tau = 1$.

\begin{lemma}\label{lem: NSP}
Assume that $\Phi$ has the $(K,\gamma_K)$-NSP \eqref{eq: NSP}. Suppose that $\sF{\Phi}$ contains a $K$-sparse vector $x^*$. 
Then this vector is the unique $\ltau$-minimizer in $\sF{\Phi}$. Moreover we have for all $v\in\sF{\Phi}$
\begin{equation}\label{eq: NSP_distance_bound}
	\|v-x^{*}\|_{\ltau}^{\tau} \le 2\frac{1+\gamma_K}{1-\gamma_K}\sigma_{K}(v)_{\ltau}.
\end{equation}
\end{lemma}
}
\HR{It is well-known that the NSP for $0< \tau \leq 1$ can be shown via the restricted isometry property \cite{Chartrand08,fora13}, but also direct proofs of 
the NSP are available for certain random matrices giving often better constants and working under weaker assumptions \cite{chgulepa12,dilera15,fora13,kara13,leme15}. 
In particular, Gaussian random matrices satisfy the NSP of order $K$ with high probability if $m \geq C K \log(K/N)$. Structured random matrices
including random partial Fourier and discrete cosine matrices, and partial random circulant matrices -- both important in applications -- satisfy the RIP and hence, the NSP with high probability provided that $m \geq C K \log^4(N)$ \cite{cata06,fora13,krmera14,ra10,ruve08}. Note that for these types of structured matrices, fast matrix vector multiplication routines are available.
}

\begin{definition}[Set of eigenvalues and singular values]
\label{def:eig}
We denote with $\Lambda(A)$ the \emph{set of eigenvalues} of a square matrix A. Respectively, $\lambda_{\min}(A)$ and $\lambda_{\max}(A)$ are the smallest and largest eigenvalues. We define by $\sigma_{\min}(A)$ and $\sigma_{\max}(A)$ the smallest and largest singular value \HR{of a rectangular matrix $A$}.
\end{definition}

\subsection{Conjugate gradient method (CG)}
The CG method was originally proposed by Stiefel and Hestenes in \cite{Stiefel1952} and generalized to complex systems in~\cite{jacobs1986}. For an Hermitian and positive \Rev{definite} matrix $A\in\C^{N\times N}$ 
the CG method solves the linear equation $A x = y$ or equivalently the minimization problem\Rev{
\[\argmin\limits_{x\in\C^N} \left(F(x)\defleft\frac{1}{2}x^*Ax - x^*y\right).\]}%
\Rev{The algorithm is designed to iteratively compute the minimizer $x^i$ of $F$ on the affine subspace $\tilde{V}_i\defleft x^0 + V_i$ with $V_i$ being the Krylov subspace $V_i \defleft\subspace\{r^0, Ar^0, \ldots, A^{i-1}r^0\}\subset\C^N$, $x^0\in\C^N$ a starting vector, and $r^0\defleft y-Ax^0$ (\emph{minimality property of CG}).}

\begin{algorithm}[H]
\caption{Conjugate Gradient (CG) method}
\label{CG}
Input: initial vector $x^{0} \in \C^{N}$, matrix $A \in \C^{N \times N}$, given vector $y \in \C^{N}$ and optionally a desired accuracy $\delta$.
\begin{algorithmic}[1]
\STATE Set $r^{0} = p^{0} = y-Ax^0$ and $i=0$
\WHILE{$r^{i}\neq0$ (or $\vectornorm{r^i}_{\ltwo} > \delta$)}
\STATE $a_{i} = \langle r^{i}, p^{i} \rangle_{\ell_2} / \langle A p^{i},p^i\rangle_{\ell_2}$
\STATE $x^{i+1} = x^{i} + a_{i}p^{i}$
\STATE $r^{i+1} = y - Ax^{i+1}$ 
\STATE $b_{i+1} = \langle Ap^{i}, r^{i+1}\rangle_{\ell_2} / \langle A p^{i},p^i\rangle_{\ell_2}$
\STATE $p^{i+1} = r^{i+1} - b_{i+1}p^{i}$
\STATE $i = i+1$
\ENDWHILE
\end{algorithmic}
\end{algorithm}
Roughly speaking, CG \HR{iteratively searches} for a minimum of the functional $F$ along conjugate directions $p^i$ with respect to $A$, i.e., $(p^i)^{*}Ap^j = 0$, $j < i$.
Thus, in step $i+1$ of CG the new iterate $x^{i+1}$ is found by minimizing $F(x^i + a_i p^i)$ with respect to the scalar $a_i\in \R$ along the 
search direction $p^i$. Since we perform a minimization in each iteration, \HR{this implies} monotonicity of the iterates, 
\HRomit $F(x^{i+1}) \leq F(x^i)$. \Rev{If the algorithm produces at some iteration a residual $r^{i}=0$, then a solution of the linear system is found. Otherwise it produces a new conjugate direction $p^{i}$. One can show that the conjugate directions $p^0,\ldots,p^{i-1}$ also span $V_i$. Since the conjugate directions are linear independent, we have $V_N=\C^N$ (assumed that $r^i\neq 0$, $i=0,\ldots,N-1$). Then, according to the above mentioned minimality property, the iterate $x^N$ is the minimizer of $F$ on $\C^N$, which means that CG terminates after at most $N$ iterations. Nevertheless,  the algorithm can be stopped after a significantly smaller number of steps as soon as the machine precision is very high and theoretically convergence already occurred. In view of propagation of errors in practice the algorithm may be run longer than just $N$ iterations though.}

The following theorem establishes the convergence and the convergence rate of CG.

\begin{theorem}[{\cite[Theorem 4.12]{sarisa00}}]
\label{thm:quarteroni}
Let the matrix $A$ be Hermitian and positive definite. The Algorithm~CG converges to the solution of the system $Ax=y$ after at most $N$ steps. Moreover, the error $x^i - x$ is such that 
\[\vectornorm{A^{\frac12}(x^i-x)}_{\ltwo}\leq \frac{2c_A^i}{1+c_A^{2i}}\vectornorm{A^{\frac12}(x^0-x)}_{\ltwo},\quad \text{with } \Rev{c_A  = \frac{\sqrt{\kappa_A}-1}{\sqrt{\kappa_A} +1}<1}, \]
where $\kappa_A = \frac{\sigma_{\max}(A)}{\sigma_{\min}(A)}$ is the condition number of the matrix $A$ and $\sigma_{\max}(A)$ (resp. $\sigma_{\min}(A)$) is the largest (resp. smallest) singular value of $A$.
\end{theorem}
\begin{remark}
Theorem~\ref{thm:quarteroni} is slightly modified with respect to the formulation in~\cite{sarisa00}. There, the matrix $A$ is considered to be symmetric instead of being Hermitian. However, in the complex case, the proof can be performed similarly by replacing the transpose by the conjugate transpose.
\end{remark}
\Rev{\begin{remark}
\label{rem:coeffless1}
Since $\kappa_A \geq 1$, it follows that $0 \leq c_A  < 1$, and also $0 \leq c_A^i  < 1$, for positive iteration numbers $i$. From
$ 0 < (1- c_A^i)^2 = 1 + c_A^{2i} - 2c_A^i$, we immediately see that $2c_A^i/(1+c_A^{2i}) < 1$ for all $i \in \N$, and obviously $2c_A^i/(1+c_A^{2i}) \to 0$ for $i \to +\infty$.
\end{remark} }

\subsection{Modified conjugate gradient method (MCG)}
\label{sec:MCG}
In Section~\ref{sec:IRLSltau}, we are interested in a vector which solves the weighted least-squares problem
\[\hat{x}=\argmin\limits_{x\in\sF{\Phi}}\|x\|_{\ell_2(w)},\]
given $\Phi\in\C^{m\times N}$ with $m\leq N$.
As we show 
below in Section~\ref{sec: original IRLS}, the minimizer $\hat{x}$ is given explicitly by the (weighted) Moore-Penrose pseudo-inverse
\[\hat{x} = D\Phi^{*}(\Phi D \Phi^{*})^{-1}y,\]
where $D \defleft \diag{[w_i^{-1}]}_{i=1}^N$. \HR{Hence, in order to determine $\hat x$, we first solve the system} 
\begin{equation}\label{eq:firstsys}
\Phi D \Phi^{*} \theta = y,
\end{equation} 
and then we compute  $\hat x = D \Phi^{*} \theta$. Notice that the system \eqref{eq:firstsys} has the general form
\begin{equation}
\label{eq:modsystem}
TT^*\theta = y,
\end{equation}
with $T\defleft\Phi D^{\frac{1}{2}}$. We consider the application of CG to this system for the matrix $A=TT^*$. This approach leads to the modified conjugate gradient (MCG) method, presented in Algorithm~\ref{MCG} and proposed by J.T.~King in \cite{king89}. It provides a sequence $(\theta^i)_{i\in\N}$ with $\theta^i \in U_i \defleft \subspace \{y, TT^{*}y, \ldots, (TT^{*})^{i-1}y\}$, the Krylov subspace associated to \eqref{eq:modsystem}, with the property that $\bar{x}^i\defleft T^*\theta^i$ 
minimizes $\vectornorm{\bar{x}^i - \bar{x}}_{\ell_2}$, where $\bar{x} = \argmin\limits_{x\in\sF{T}}\vectornorm{x}_{\ell_2}$. \HR{Finally}, we compute $\hat{x} = D^{\frac12}\bar{x}$. 

\begin{algorithm}[H]
\caption{Modified conjugate gradient (MCG) method}
\label{MCG}
Input: initial vector $\theta^{0} \in \C^{m}$, \HR{$T \in \C^{m \times N}$, $y \in \C^{m}$, desired accuracy $\delta$ (optional)}.

\begin{algorithmic}[1]
\STATE Set $\rho^{0} = p^{0} = y$ and $i=0$
\WHILE{$\rho^{i}\neq0$ (or $\vectornorm{\rho^i}_{\ltwo} > \delta$)}
\STATE $\alpha_{i} = \langle \rho^{i}, p^{i} \rangle_{\ell_2} / \| T^{*} p^{i}\|_{\ell_2}^{2}$
\STATE $\theta^{i+1} = \theta^{i} + \alpha_{i}p^{i}$
\STATE $\rho^{i+1} = y - TT^{*}\theta^{i+1}$ \label{def: rho}
\STATE $\beta_{i+1} = \langle T^{*}p^{i}, T^{*}\rho^{i+1}\rangle_{\ell_2} / \| T^{*} p^{i}\|_{\ell_2}^{2}$
\STATE $p^{i+1} = \rho^{i+1} - \beta_{i+1}p^{i}$
\STATE $i = i+1$
\ENDWHILE
\STATE Set $\bar{x}^{i+1} = T^{*} Ñ\theta^{i+1}$
\end{algorithmic}
\end{algorithm}

The following theorem provides a precise rate of convergence of MCG. Additionally, we emphasize the monotonic decrease of the error $\vectornorm{\hat{x}^i-\hat{x}}_{\ltwo(w)}$, which we use below in Lemma \ref{lemma:ineqfunctional}.
\begin{theorem}
Suppose the matrix $T$ to be surjective. Then the sequence $(\bar{x}^i)_{i\in\N}$ generated by the Algorithm~MCG converges to $\bar{x} = T^*(TT^*)^{-1}y$ in at most $N$ steps, and
\begin{equation}\label{eq:estimmcg} \vectornorm{\bar{x}^i-\bar{x}}_{\ltwo}\leq \frac{2c_{TT^*}^i}{1+c_{TT^*}^{2i}}\vectornorm{\bar{x}^0-\bar{x}}_{\ltwo} \text{, with }\Rev{c_{TT^*}<1},
\end{equation}
for all $i\geq 0$, where $c_{TT^*} \HR{= \frac{\sqrt{\kappa(T T^*)} - 1}{\sqrt{\kappa(TT^*)}+1} = \frac{\sigma_{\max}(T)-\sigma_{\min}(T)}{\sigma_{\max}(T)+\sigma_{\min}(T)}}$ 
is defined as in Theorem~\ref{thm:quarteroni}, and $\bar{x}^0 = T^*\theta^0$ is the initial vector. Moreover, by setting $D \defleft \diag{[w_i^{-1}]}_{i=1}^N$, and $\hat{x}^i=D^{\frac12}\bar{x}^i$ as well as $\hat{x}=D^{\frac12}\bar{x}$, we obtain
\begin{equation}
\label{eq:estimateMCGcond}
\vectornorm{\hat{x}^i-\hat{x}}_{\ltwo(w)}\leq \frac{2c_{TT^*}^i}{1+c_{TT^*}^{2i}}\vectornorm{\hat{x}^0-\hat{x}}_{\ltwo(w)}.
\end{equation}
\end{theorem}
\begin{proof}
By Theorem~\ref{thm:quarteroni}, we have 
\[ \vectornorm{(TT^*)^{\frac12}(\theta^i-\theta)}_{\ltwo}\leq \frac{2c_{TT^*}^i}{1+c_{TT^*}^{2i}}\vectornorm{(TT^*)^{\frac12}(\theta^0-\theta)}_{\ltwo},\]
for $\theta$ as given in~\eqref{eq:modsystem}. By the identity
\begin{align*}
\vectornorm{(TT^*)^{\frac12}(\theta^i-\theta)}_{\ltwo}^2 &= \langle (TT^*)^{\frac12}(\theta^i-\theta), (TT^*)^{\frac12}(\theta^i-\theta)\rangle_{\ltwo}
= \langle (TT^*)(\theta^i-\theta), \theta^i-\theta\rangle_{\ltwo}\\
&= \langle T^*(\theta^i-\theta), T^*(\theta^i-\theta)\rangle_{\ltwo}
= \langle \bar{x}^i-\bar{x}, \bar{x}^i-\bar{x}\rangle_{\ltwo}
= \vectornorm{\bar{x}^i-\bar{x}}_{\ltwo}^2,
\end{align*}
we obtain the assertion~\eqref{eq:estimmcg}. \HR{Inequality \eqref{eq:estimateMCGcond} follows then from the definition of the diagonal matrix $D$ and the weighted norm $\ell_2(w)$.} \Rev{The fact that the coefficient  $2c_{TT^*}^i/(1+c_{TT^*}^{2i})<1$ for all $i \in \N$, and $2c_{TT^*}^i/(1+c_{TT^*}^{2i}) \to 0$  for $i \to \infty$ follows as in Remark~\ref{rem:coeffless1}.}
\end{proof}

\section{Conjugate gradient acceleration of the IRLS method for $\ltau$-minimization}
\label{sec:IRLSltau}
In this section, we start with a detailed introduction of the IRLS algorithm and its modified version 
\HR{that uses CG} for the solution of the successive quadratic optimization \HR{problems}. Afterwards, we present two results providing \HRomit 
the convergence and the rate of convergence of the modified algorithm.
\HR{As crucial feature, we give bounds on the accuracies of the (inexact) CG solutions of the intermediate least squares problems which ensure convergence of the overall IRLS methods. In particular,
these tolerances must depend on the current iteration and should tend to zero with increasing iteration count. In fact, without this condition, one may observe
divergence of the method.}
The proofs of the theorems are developed into several lemmas.

From now on, we consider a fixed parameter $\tau$ such that $0 < \tau \leq 1$. At some points of the presentation, we  explicitly switch to the case of $\tau = 1$ to prove additional properties of the algorithm which 
\HR{are due} to the convexity of the $\lone$-norm minimization problem.

\subsection{Iteratively Re-weighted Least Squares (IRLS) algorithm for $\ell_{\tau}$-minimization} \label{sec: original IRLS}


\HR{The following functional turns out to be a crucial tool for the analysis of the IRLS algorithm and its modified variant.}

\begin{definition}\label{def: sJ}
Given a real number $\etil > 0$, $x \in \C^N$, and a weight vector $w\in\mathbb{R}^{N}$ with positive entries $w_j>0$, $j=1,\ldots,N$, we define
\begin{equation} \label{eq: sJ}
\sJ \left( x,w,\etil \right) \defleft \frac{\tau}{2}\left[ \sum\limits_{j=1}^{N}{|x_{j}|^{2}w_{j}}  + \sum\limits_{j=1}^{N}{\left( \etil^{2}w_{j} + \frac{2-\tau}{\tau}w_{j}^{-\frac{\tau}{2-\tau}}\right)} \right].
\end{equation}
\end{definition}

\Rev{The standard IRLS algorithm for $\ell_{\tau}$-minimiziation is intuitively motivated in~\cite{dadefogu10} by means of a weighted least squares approximation of the $\ell_{\tau}$-minimization problem. However, for the sake of a concise presentation, we introduce below the algorithm directly as an alternating minimization procedure of the functional $\mathcal J_{\tau}$ with respect to the three variables $x$, $w$, and $\varepsilon$.} \Rev{ Algorithm~\ref{old_algorithm}, recalls the formulation of IRLS, as appearing in \cite[Section 7.2]{dadefogu10}, or \cite[Chapter 15.3]{fora13}. For the sake of notational clarity,  we use the nonincreasing rearrangement $r(\cdot)$, as introduced in Definition~\ref{def:nir},  at step 3. }

\begin{algorithm}[H]
\caption{Iteratively Re-weighted Least Squares (IRLS)}
\label{old_algorithm}
Set $w^{0} \defleft (1,\ldots,1),$ $\etil^{0} \defleft 1$
\begin{algorithmic}[1]
\WHILE{$\etil^{n}\neq0$}
\STATE $x^{n+1} \defleft \argmin\limits_{x\in\sF{\Phi}}\sJ(x,w^{n},\etil^{n}) = \argmin\limits_{x\in\sF{\Phi}}\left\|x\right\|_{\ell_{2}(w^{n})}$ \label{def: x_opt}
\STATE $\etil^{n+1} \defleft \min(\etil^{n}, \frac{r(x^{n+1})_{K+1}}{N})$ \label{def: epsilon}
\STATE $w^{n+1} \defleft \argmin\limits_{w>0} \sJ(x^{n+1}, w, \etil^{n+1})$, i.e., $w_j^{n+1} = [|x_j^{n+1}|^2 + (\etil^{n+1})^2]^{-\frac{2-\tau}{2}}, \hspace{0.5cm} j = 1,\ldots,N$\label{def: w_opt}
\ENDWHILE
\end{algorithmic}
\end{algorithm}


\Rev{The convergence of IRLS is by now well-established and we refer to  \cite{dadefogu10}  and \cite[Section 15.3]{fora13} for details, which we in part extend in our analysis in  Section~\ref{sec:irlsconv}.}

\HRomit

In this section we propose a practical method to solve approximatively the least squares problems appearing in \Rev{step~\ref{def: x_opt}} of Algorithm~\ref{old_algorithm}. The following characterization of their solution turns out to be very useful. \HR{Note that the $\ell_{2}(w)$-norm is strictly convex, therefore its minimizer subject to an affine constraint
is unique.}

\begin{lemma}[{\cite[{(2.6)}]{dadefogu10}}, {\HR{\cite[Proposition A.23]{fora13}}}]
\label{lem: orthogonality}
	We have
	$\hat{x} = \argmin\limits_{x\in\sF{\Phi}}\|x\|_{\ell_{2}(w)}$ if and only if \HR{$\hat{x} \in \sF{\Phi}$} and
	\begin{equation}\label{eq: orthogonality}
		\langle \hat{x}, \eta \rangle_{w} = 0 \hspace{1cm}\text{ for all }\hspace{0.5cm} \eta \in \mathcal{N}_{\Phi}.
	\end{equation}
\end{lemma}

\HRomit
By means of Lemma~\ref{lem: orthogonality}, we are able to derive an explicit representation of the weighted $\ltwo$-minimizer $\hat{x}:=\argmin\limits_{x\in\sF{\Phi}}\vectornorm{x}_{\ell_2(w)}$. \Rev{Define $D \defleft \diag\left[(w_j)^{-1}\right]_{j=1}^{N}$.} From~\eqref{eq: orthogonality}, we have the equivalent formulation
\[D^{-1}\hat{x}\in\mathcal{R} ( \Phi^*),\]
where $\mathcal{R}(\cdot)$ denotes the range of a linear map. Therefore, there is a $\xi\in\R^m$ such that $\hat{x}=D\Phi^*\xi$. To compute $\xi$, we observe that
\[y = \Phi \hat{x} = (\Phi D \Phi^*)\xi,\]
and thus, since $\Phi$ has full rank and $\Phi D \Phi^*$ is invertible, we conclude
\[\hat{x} = D\Phi^*\xi = D\Phi^*(\Phi D \Phi^*)^{-1}y.\]

As a consequence, we see that at step \ref{def: x_opt} of Algorithm~IRLS the minimizer of the least squares problem is explicitly given by the equation
\begin{equation}\label{eq: def_least_squares_problem}
	x^{n+1} = D_n\Phi^{*}(\Phi D_n \Phi^{*})^{-1}y,
\end{equation}
where we introduced the $N \times N$ diagonal matrix $$D_n \defleft \diag\left[(w_j^n)^{-1}\right]_{j=1}^{N}.$$

Furthermore, the new weight vector in step~\ref{def: w_opt} of Algorithm~IRLS is explicitly given by
\begin{equation}\label{eq: def_new_weights}
	w_j^{n+1} = [|x_j^{n+1}|^2 + (\etil^{n+1})^2]^{-\frac{2-\tau}{2}}, \hspace{1cm} j = 1,\ldots,N.
\end{equation} 
Taking into consideration that $w_j>0$, this \HR{formula} can be derived \HR{from} the first order optimality condition $\HR{\partial} \sJ(x^{n+1}, w, \etil^{n+1}) / \HR{\partial} w = 0$. \\

\subsection{The algorithm CG-IRLS}
\label{sec:IRLSCGalg}

Instead of solving {\it exactly} the system of linear equations \HR{\eqref{eq: def_least_squares_problem} 
occurring in step \ref{def: x_opt}} of algorithm~IRLS, we substitute the exact solution by \HR{the} approximate solution provided by the iterative algorithm MCG described in Section~\ref{sec:MCG}. We shall set a tolerance $\tol_{n+1}$, which gives us an upper threshold for the error between the optimal and the approximate solution in the weighted $\ell_2$-norm. In this section, we give a precise and implementable condition on the sequence $(\tol_n)_{n\in\N}$ of the tolerances \HR{that guarantees convergence} of \HR{the modified IRLS presented as Algorithm~\ref{cg-IRLS} below}. 

\begin{algorithm}[h!]
\caption{Iteratively Re-weighted Least Squares combined with CG (CG-IRLS)}
\label{cg-IRLS}
Set $\wtil^{0} \defleft (1,\ldots,1)$, $\etil^{0} \defleft 1$, $\beta \in (0,1]$
\begin{algorithmic}[1]
\WHILE{$\etil^{n}\neq0$}
\STATE Compute $\xtil^{n+1}$ by means of MCG s.t. $\|\hat{x}^{n+1} - \xtil^{n+1}\|_{\ell_{2}(\wtil^{n})}^{2} \le \textnormal{tol}_{n+1}$,  where $\hat{x}^{n+1} \defleft \argmin\limits_{x\in\sF{\Phi}}\sJ(x,w^{n},\etil^{n}) = \argmin\limits_{z\in\sF{\Phi}}\left\|z\right\|_{\ell_{2}(\wtil^{n})}$\label{def: xtil}. \HR{Use the last iterate $\theta^{n,i}$ corresponding to
$\tilde{x}^{n} = T^* \theta^{n,i}$ from MCG of the previous IRLS iteration as initial vector $\theta^0 = \theta^{n+1,0}$ for the present run of MCG.}
\STATE $\etil^{n+1} \defleft \min(\etil^{n}, \beta r(\xtil^{n+1})_{K+1})$ \label{def: etil}
\STATE $\wtil^{n+1} \defleft \argmin\limits_{w>0} \sJ(\xtil^{n+1}, w, \etil^{n+1})$, i.e., $w_j^{n+1} = [|\xtil_j^{n+1}|^2 + (\etil^{n+1})^2]^{-\frac{2-\tau}{2}}, \hspace{0.5cm} j = 1,\ldots,N$ \label{def: wtil}
\ENDWHILE
\end{algorithmic}
\end{algorithm}

In contrast to Algorithm~IRLS, the value $\beta$ in step \ref{def: etil} is introduced to obtain flexibility in tuning the performance of the algorithm. While we prove in Theorem~\ref{thm: convergence} convergence for any positive value of $\beta$, \HR{Theorem~\ref{thm: convergence}(iii) below guarantees instance optimality only} for $\beta<\left(\frac{1-\gamma}{1+\gamma}\frac{K+1-k}{N}\right)^{\frac{1}{\tau}}$ in the case that $\lim\limits_{n\rightarrow\infty}\etil^n\neq0$. 
\HR{Nevertheless in practice, choices of $\beta$ which do not necessarily fulfill this condition may work very well. Section~\ref{sec:numerics}, 
investigates good choices of $\beta$ numerically.} 

From now on, we fix the notation $\hat{x}^{n+1}$ for the exact solution in step \ref{def: xtil} of Algorithm~\ref{cg-IRLS}, and $\xtil^{n+1,i}$ for its approximate solution in the $i$-th iteration of Algorithm MCG. We have to make sure that $\|\hat{x}^{n+1} - \xtil^{n+1,i}\|_{\ell_{2}(\wtil^{n})}^{2}$ is sufficiently small to fall below the given tolerance. To this end, we could use the bound on the error provided by~\eqref{eq:estimateMCGcond}, but \HR{this} has the following two unpractical drawbacks:
\begin{enumerate}
\item The vector $\hat{x} = \hat{x}^{n+1}$ is not known a priori;
\item The computation of the condition number $c_{TT^*}$ is possible, but it requires the computation of 
eigenvalues with additional computational cost which we prefer to avoid.
\end{enumerate}

Hence, we propose an alternative estimate of the error in order to guarantee $\|\hat{x}^{n+1} - \xtil^{n+1}\|_{\ell_{2}(\wtil^{n})}^{2} \le \textnormal{tol}_{n+1}$. 
We use the notation of Algorithm MCG, but add an additional upper index for the outer IRLS iteration, e.g., $\theta^{n+1,i}$ is the $\theta^i$ in the $n+1$-th IRLS iteration. After $i$ steps of MCG, we have 
\begin{equation*}
	\|\hat{x}^{n+1} - \xtil^{n+1,i}\|_{\ell_{2}(\wtil^{n})}^{2} = \|D_{n}\Phi^{*}(\Phi D_{n} \Phi^{*})^{-1}y - D_{n}\Phi^{*}\theta^{n+1,i}\|_{\ell_{2}(\wtil^{n})}^{2}.
\end{equation*}
We use $\theta^{n+1,i} = (\Phi D_n \Phi^{*})^{-1}(y-\rho^{n+1,i})$ from step \ref{def: rho} of MCG to obtain
\begin{align*}
	\|\hat{x}^{n+1} &- \xtil^{n+1,i}\|_{\ell_{2}(\wtil^{n})}^{2} = \|D_{n}^{\frac12}\Phi^{*}(\Phi D_{n} \Phi^{*})^{-1}\rho^{n+1,i}\|_{\ell_{2}}^{2} \leq \| D_{n}\|\|\Phi\|^{2} \|(\Phi D_{n} \Phi^{*})^{-1}\|^{2} \|\rho^{n+1,i}\|_{\ell_{2}}^{2} \\
	&= \frac{\max\limits_{\HR{1 \leq \ell \leq N}}\left(|\tilde{x}_{\HR{\ell}} |^2+\left(\etil^n\right)^2\right)^{\frac{2-\tau}{2}}\|\Phi\|^{2}}{\lambda_{\min}\left( \Phi D_{n} \Phi^{*} \right)} \|\rho^{n+1,i}\|_{\ell_{2}}^{2} 
	\le \left(1+\max\limits_{1\leq \HR{\ell} \leq N}\left(\frac{|\tilde{x}_\HR{\ell} ^n|}{\etil^n}\right)^2\right)^{\frac{2-\tau}{2}}\frac{\|\Phi\|^{2}}{\sigma_{\min}\left( \Phi  \right)} \|\rho^{n+1,i}\|_{\ell_{2}}^{2}.
\end{align*}

The last inequality above 
\HR{results} from $\lambda_{\min}\left( \Phi D_{n} \Phi^{*} \right) = \sigma_{\min}^{2}\left( \Phi D_{n}^{\frac12}  \right)$ and 
\begin{equation*}
\sigma_{\min}\left(\Phi D_{n}^{\frac12} \right) 
\HR{\geq} \sigma_{\min}\left( \Phi \right)\sigma_{\min}\left( D_{n}^{\frac12}  \right)\geq (\etil^n)^{2-\tau}\sigma_{\min}\left( \Phi \right).
\end{equation*}

Since $\etil^n$ and $\tilde{x}^n$ are known from the previous iteration, and $\|\rho^{n+1,i}\|_{\ell_2}$ is explicitly calculated within the MCG algorithm, $\|\hat{x}^{n+1} - \xtil^{n+1,i}\|_{\ell_{2}(\wtil^{n})}^{2} \le \textnormal{tol}_{n+1}$ can be achieved by iterating until 
\begin{equation}
\label{eq:stopcritmcg}
\|\rho^{n+1,i}\|^2_{\ell_2} \le \frac{\sigma_{\min}\left( \Phi  \right)}{\left(1+\max\limits_{1\leq \HR{\ell} \leq N}\left(\frac{|\tilde{x}_{\HR{\ell}}^n|}{\etil^n}\right)^2\right)^{\frac{2-\tau}{2}}\|\Phi\|^{2}} \textnormal{tol}_{n+1}.
\end{equation}
Consequently, we shall use the minimal $i\in\N$ such that the above inequality is valid and set $\xtil^{n+1}\defleft\xtil^{n+1,i}$, which will be the standard notation for the approximate solution.\\

In inequality~\eqref{eq:stopcritmcg}, the computation of $\sigma_{\min}\left( \Phi  \right)$ and $\|\Phi\|$ is necessary. The computation of these constants might be demanding, but has to be performed only once before the algorithm starts. Furthermore, in practice it is sufficient to compute approximations of these values \HR{and therefore} these operations are not critical for the computation time of the algorithm.


\subsection{Convergence results}
\label{sec:irlsconv}
After introducing Algorithm~CG-IRLS, we state below the two main results of this section. Theorem~\ref{thm: convergence} shows the convergence of the algorithm to a limit point \HR{that obeys certain} error guarantees with respect to the solution
of  \eqref{eq:ltauproblem}. Below $K$ denotes the index used in the $\etil$-update rule, i.e., step \ref{def: etil}) of Algorithm~CG-IRLS.

\begin{theorem} \label{thm: convergence}
\HR{Let $0 < \tau \leq 1$. Assume $K$ 
is such} that $\Phi$ satisfies the \textit{Null Space Property} \eqref{eq: NSP} of order $K$, with $\gamma<1$. If $\tol_{n+1}$ in Algorithm~CG-IRLS is chosen such that
\begin{equation}
\label{eq:deftol}
\sqrt{\tol_{n+1}} \le \sqrt{\left(\frac{c_n}{2}\right)^2+\frac{2a_{n+1}}{\tau\bar{W}^2_{n+1}}}-\frac{c_n}{2},
\end{equation}
where
\begin{align}
\label{eq:defcnwn}
c_n &\defleft 2 W_n\left(\vectornorm{\tilde{x}^{n}}_{\ltwo(w^{n-1})} + \sqrt{\tol_n}\right), \text{ with   } \\
\bar{W}_{n} &\defleft  \sqrt{\frac{\max\limits_i |\xtil_i^{n-1}|^{2-\tau} + (\etil^{n-1})^{2-\tau}}{(\etil^{n})^{2-\tau}}},\textrm{ and } W_n\defleft\vectornorm{D_n^{-\frac12}D_{n-1}^{\frac12}}, \label{def:Wn}
\end{align}
for a sequence $(a_n)_{n\in\N}$, which fulfills $a_n\geq 0$ for all $n\in\N$, and $\sum\limits_{i=0}^{\infty}a_n < \infty$,
then, for each $y\in\C^{m}$, Algorithm CG-IRLS produces a non-empty set of accumulation points $\mathcal{Z}_{\tau}(y)$. Define $\etil\defleft\lim\limits_{n\rightarrow\infty}\etil^n$, then the following holds: 
\begin{enumerate}
	\item If $\etil = 0$, then $\mathcal{Z}_{\tau}(y)$ consists of a single $K$-sparse vector $\bar{x}$, which is the unique $\ltau$-minimizer in $\sF{\Phi}$. Moreover, we have
	for \HRomit any $x\in\sF{\Phi}$:
	\begin{equation}\label{eq: theorem_statement_1}
		\|x-\bar{x}\|_{\ltau}^{\tau} \le c_1\sigma_{\HR{K}}(x)_{\ltau}, \hspace{3mm} \text{with } c_1 \defleft 2\frac{1+\gamma}{1-\gamma} .
	\end{equation}
	\item If $\etil > 0$, then for each $\bar{x} \in\mathcal{Z}_{\tau}(y)\neq\emptyset$, we have $\left\langle \bar{x},\eta\right\rangle_{\hat{w}(\bar{x},\etil,\tau)} = 0$ for all $\eta \in\mathcal{N}_{\Phi}$, \HR{where $\hat{w}(\bar{x},\etil,\tau) = \left[\left||\bar{x}_{i}|^{2} + \etil^{2}\right|^{-\frac{2-\tau}{2}}\right]_{i=1}^N$}. Moreover, in the case of $\tau = 1$, $\bar{x}$ is the single element of $\mathcal{Z}_{\tau}(y)$ and $\bar{x}=x^{\etil,1} \defleft \argmin\limits_{x\in\sF{\Phi}}\sum\limits_{j=1}^{N}|x_{j}^{2} + \etil^{2}|^{\frac{1}{2}}$ (compare~\eqref{eq: x_epsilon}).
	\item Denote by $\mathcal{X}_{\etil,\tau}(y)$ the set of global minimizers of $f_{\etil,\tau}(x) \defleft \sum\limits_{j=1}^{N}|x_{j}^{2} + \etil^{2}|^{\frac{\tau}{2}}$ on $\sF{\Phi}$. If $\etil >0$ and $\bar{x} \in \mathcal{Z}_{\tau}(y) \cap \mathcal{X}_{\etil,\tau}(y)$, then for each $x \in \sF{\Phi}$ and any $ \beta < \left(\frac{1-\gamma}{1+\gamma}\frac{K+1-k}{N}\right)^{\frac{1}{\tau}}$, we have \[\vectornorm{x - \bar{x}}_{\ltau}^{\tau}\leq c_2 \sigma_k(x)_{\ltau}, \hspace{3mm} \text{with } c_2\defleft \frac{1+\gamma}{1-\gamma}\left(\frac{2+\frac{N\beta^{\tau}}{K+1-k}}{1-\frac{N\beta^{\tau}}{K+1-k}\frac{1+\gamma}{1-\gamma}}\right).\]
\end{enumerate}
\end{theorem}
\Rev{
\begin{remark}
\label{rem:loose_tol}
Notice that \HR{\eqref{eq:deftol} is an implicit bound on $\mathrm{tol}_{n+1}$} since it depends on $\etil^{\HR{n+1}}$, which means that in practice this value has to be updated in the MCG loop of the algorithm. To be precise, after the update of $\theta^{n+1,i+1}$ in step 4 of Algorithm~\ref{MCG} we compute $\tilde{x}^{n+1,i+1}= T^*\theta^{n+1,i+1}$ in each iteration $i$ of the MCG loop. \HR{If $\tilde{x}^{n+1,i+1}$ is $K$-sparse for some iteration $i$, then 
$\varepsilon^{n+1} = \varepsilon^{n+1,i+1} = \min\left\{\varepsilon^n,\beta r\left(\tilde{x}^{n+1,i+1}\right)_{K+1}\right\} = 0$ and $\tol_{n+1} = 0$ by 
\eqref{eq:defcnwn} and \eqref{def:Wn}. In this case, MCG and IRLS are stopped by definition.} The usage of this implicit bound is not  efficient in practice since the computation of $r(\tilde{x}^{n+1,i+1})_{K+1}$ requires a sorting of $N$ elements in each iteration of the MCG loop. While the implicit rule is required for the convergence analysis of the algorithm, we demonstrate in Section~\ref{sec:numAlgCGIRLS} that an explicit rule is sufficient for convergence in practice, and more efficient in terms of computational time.
\end{remark}
}

Knowing that the algorithm converges and leads to an adequate solution, \HR{one is also interested in how fast one approaches this solution}. Theorem~\ref{thm: rate of convergence} states that a linear rate of convergence can be established in the case of $\tau =1$. In the case of $0 < \tau < 1$ this rate is even asymptotically super-linear.

\begin{theorem} \label{thm: rate of convergence}
	Assume $\Phi$ satisfies the \textit{NSP} of order $K$ with constant $\gamma$ such that $0 < \gamma < 1- \frac{2}{K+2}$, and that $\sF{\Phi}$ contains a $k$-sparse vector $x^*$. Define $\Lambda \defleft \supp(x^*)$. Suppose that $k < K - \frac{2\gamma}{1-\gamma}$ and $0 < \nu < 1$ are such that 
	\begin{align}
		\nonumber\mu &\defleft \frac{\gamma (1+\gamma)}{(1-\nu)^{\tau(2-\tau)}\left(\min\limits_{j\in\Lambda}|x_j^*|\right)^{\tau(1-\tau)}}  \left(1 +\frac{(N-k)\beta^{\tau}}{K+1-k}\right)^{2-\tau}< 1, \\
		\nonumber R^* &\defleft \left(\nu \min\limits_{j\in \Lambda} |x_j^*|\right)^{\tau},\\
		\label{eq:tildemu} \tilde{\mu}(R^*)^{1-\tau} &\leq 1,
	\end{align}
	for \HR{some $\tilde{\mu}$ satisfying} $\mu < \tilde{\mu} < 1$. Define the error 
	\begin{equation} \label{def: error}
E_n \defleft \| \tilde{x}^n - x^*\|_{\ltau}^{\tau}.
\end{equation}
	Assume there exists $n_0$ such that
	\begin{equation}\label{eq: error induktion}
		E_{n_0} \le R^*. 
	\end{equation}
If $a_{n+1}$ and $tol_{n+1}$ are chosen as in Theorem~\ref{thm: convergence} with the additional bound
\begin{equation}
	\label{tolrec}
	\tol_{n+1} \leq \left(\frac{(\tilde{\mu}-\mu)E_n^{2-\tau}}{(NC)^{\frac{2-\tau}{2}}}\right)^{\frac{2}{\tau}},
	\end{equation}	
	then for all $n\geq n_0$, we have
	\begin{equation}
	\label{eq:errordynamicsltau}
		E_{n+1} \le \mu E_n^{2-\tau} + (NC)^{1-\frac{\tau}{2}}(\tol_{n+1})^{\frac{\tau}{2}},
	\end{equation}
	and
	\begin{equation}
	\label{eq:errordynamicsltaugeneral}
		E_{n+1} \le \tilde{\mu} E_n^{2-\tau} ,
	\end{equation}
	where ${C\defleft 3\sum\limits_{n=1}^{\infty}a_n} + \sJ \left( \tilde{x}^{1}, \wtil^{0}, \etil^{0} \right)$.
	Consequently, $\tilde{x}^n$ converges \Rev{linearly to  $x^*$} in the case of $\tau=1$. \Rev{The convergence is super-linear  in the case of $0 < \tau < 1$.} 
\end{theorem}	
\begin{remark}
Note that the second bound in~\eqref{tolrec}, which implies~\eqref{eq:errordynamicsltaugeneral}, is only of theoretical nature. Since the value of $E_n$ is unknown it cannot be computed in an implementation. However, heuristic choices of $\tol_{n+1}$ may fulfill this bound. Thus, in practice one can only guarantee the \enquote{asymptotic} (super-)linear convergence~\eqref{eq:errordynamicsltau}. 
\end{remark}

In the remainder of this section we aim to prove both results by means of some technical lemmas which are reported in Section~\ref{sec:resfunc} and Section~\ref{sec:fetil}.

\subsubsection{Preliminary results concerning the functional $\sJ(x,w,\etil)$}
\label{sec:resfunc}

\HR{One important issue in the investigation of the dynamics of Algorithm~CG-IRLS} is the relationship between the weighted norm of an iterate and the weighted norm of its predecessor. In the following lemma, we present some helpful estimates.
\begin{lemma}\label{lem:iteratestol}
Let $\hat{x}^n$, $\hat{x}^{n+1}$, $\tilde{x}^n$, $\tilde{x}^{n+1}$ and the respective tolerances $\tol_n$ and $\tol_{n+1}$ as defined in Algorithm~CG-IRLS. Then the inequalities
\begin{align}\label{eq:iteratestol1}
\left|\vectornorm{\hat{x}^{n+1}}_{\ltwo(w^n)} - \vectornorm{\tilde{x}^{n+1}}_{\ltwo(w^n)} \right| &\leq \sqrt{\tol_{n+1}},\text{ and}\\
\label{eq:iteratestol2}
\vectornorm{\hat{x}^{n+1}}_{\ltwo(w^n)}&\leq W_n\left(\vectornorm{\tilde{x}^{n}}_{\ltwo(w^{n-1})} + \sqrt{\tol_n}\right),
\end{align}
hold for all $n\geq 1 $, where $W_n\defleft\vectornorm{D_n^{-\frac12}D_{n-1}^{\frac12}}$.
\end{lemma}
\begin{proof}
Inequality~\eqref{eq:iteratestol1} is a direct consequence of the triangle inequality for norms and the property that $\vectornorm{\hat{x}^{n+1} - \xtil^{n+1}}_{\ltwo(w^{n})}\leq\sqrt{\tol_{n+1}}$ of step 2 in Algorithm~CG-IRLS. \\
\HR{In order to} prove inequality~\eqref{eq:iteratestol2}, we first notice that $\hat{x}^n,\hat{x}^{n+1}\in\sF{\Phi}$. Using that $\hat{x}^{n+1}$ is the minimizer of $\vectornorm{\cdot}_{\ltwo(w^{n})}$ on $\sF{\Phi}$, 
we obtain 
\begin{align*}
\vectornorm{\hat{x}^{n+1}}_{\ltwo(w^n)} &\leq \vectornorm{\hat{x}^{n}}_{\ltwo(w^n)} = \vectornorm{D_n^{-\frac12}\hat{x}^{n}}_{\ltwo} = \vectornorm{D_n^{-\frac12}D_{n-1}^{\frac12}D_{n-1}^{-\frac12}\hat{x}^{n}}_{\ltwo}\\
&\leq \vectornorm{D_n^{-\frac12}D_{n-1}^{\frac12}}\vectornorm{D_{n-1}^{-\frac12}\hat{x}^{n}}_{\ltwo} = W_n \vectornorm{\hat{x}^{n}}_{\ltwo(w^{n-1})}
\leq W_n\left(\vectornorm{\tilde{x}^{n}}_{\ltwo(w^{n-1})} + \sqrt{\tol_n}\right),
\end{align*}
where the last inequality is due to~\eqref{eq:iteratestol1}. \\
\end{proof}

The functional $\sJ(x,w,\etil)$ \HR{obeys} the following monotonicity property.

\begin{lemma}\label{lem: J_ineq}
The inequalities 
\begin{equation}\label{sec:eq: J_ineq}
\sJ \left( \tilde{x}^{n+1}, \wtil^{n+1}, \etil^{\HR{n+1}} \right)
\le
\sJ \left( \tilde{x}^{n+1}, \wtil^{n}, \etil^{\HR{n+1}} \right)
\le
\sJ \left( \tilde{x}^{n+1}, \wtil^{n}, \etil^{\HR{n}} \right).
\end{equation}
hold for all $n\ge0$.
\end{lemma}
\begin{proof}
	The first inequality follows from the minimization property of $\wtil^{n+1}$. The second inequality follows from $\etil^{\HR{n+1}}\le\etil^{\HR{n}}$.
\end{proof}

The following lemma describes how the difference of the functional, evaluated in the exact and the approximated solution can be controlled by a positive scalar $a_{n+1}$ and an appropriately chosen tolerance $\tol_{n+1}$.

\begin{lemma}\label{lem: toleranz}
Let $a_{n+1}$ be a positive scalar, $\xtil^{n+1}$, $\wtil^{n+1}$, and $\etil^{n+1}$ as described in Algorithm CG-IRLS, and $\hat{x}^{n+1} = \argmin\limits_{x \in \sF{\Phi}} \sJ\left( x, \wtil^{n}, \etil^{n} \right)$ . If we choose $\tol_{n}$ as in~\eqref{eq:deftol}, then
\begin{align}
\label{eq: Toleranz}\left|\sJ \left( \hat{x}^{n+1}, \wtil^{n+1}, \etil^{n+1} \right) - \sJ \left( \tilde{x}^{n+1}, \wtil^{n+1}, \etil^{n+1} \right)\right| &\le a_{n+1},\\
\label{eq: Toleranz2}\left|\sJ \left( \hat{x}^{n+1}, \wtil^{n}, \etil^{n} \right) - \sJ \left( \tilde{x}^{n+1}, \wtil^{n}, \etil^{n} \right) \right| &\le a_{n+1}, \text{ and}\\
\label{eq: J_xopt_delta}\sJ \left( \hat{x}^{n+1}, \wtil^{n+1}, \etil^{n+1} \right) &\le \sJ \left( \hat{x}^{n+1}, \wtil^{n}, \etil^{n} \right) +  2 a_{n+1}.
\end{align}
\end{lemma}

\begin{proof}
The core of this proof is to find a bound on the quotient of the weights from one iteration step to the next and then to use the bound of the difference between $\hat{x}^{n+1}$ and $\xtil^{n+1}$ in the $\ell_2(w^{n})$-norm by $\tol_{n+1}$. Starting with the definition of $W_{n+1}$ in Lemma~\ref{lem:iteratestol}, the quotient of two successive weights can be estimated by
\begin{align}
\nonumber W_{n+1} &= \vectornorm{D_{n+1}^{-\frac12}D_{n}^{\frac12}} = \sqrt{\max\limits_{\HR{\ell}=1,\ldots,N}\frac{w_\HR{\ell}^{n+1}}{w_\HR{\ell}^{n}}} = \sqrt{\max\limits_{\HR{\ell}=1,\ldots,N}\frac{\left(|\xtil_{\HR{\ell}}^{n}|^2 +(\etil^n)^2\right)^{\frac{2-\tau}{2}}}{\left(|\xtil_{\HR{\ell}}^{n+1}|^2 +(\etil^{n+1})^2\right)^{\frac{2-\tau}{2}}}} \\
\label{eq: weight_approx}&\le \sqrt{\frac{\max\limits_{\HR{\ell=1,\ldots,N}} |\xtil_{\HR{\ell}}^{n}|^{2-\tau} + (\etil^n)^{2-\tau}}{(\etil^{n+1})^{2-\tau}}} = \bar{W}_{n+1},
\end{align}
where $\bar{W}_{n+1}$ was defined \HR{in~\eqref{def:Wn}}.
By choosing $\textnormal{tol}_{n+1}$ as in~\eqref{eq:deftol}, we obtain
\begingroup
\allowdisplaybreaks
\begin{align*}
&\left|\sJ \left( \hat{x}^{n+1}, \wtil^{n+1}, \etil^{n+1} \right) - \sJ \left( \tilde{x}^{n+1}, \wtil^{n+1}, \etil^{n+1} \right) \right|  \\
 =&\left|\frac{\tau}{2} \sum\limits_{j=1}^{N}{\left(|\hat{x}_{j}^{n+1}|^{2} - |\tilde{x}_{j}^{n+1}|^{2}\right)\wtil_{j}^{n+1}} \right|\\
=&\left|\frac{\tau}{2} \sum\limits_{j=1}^{N}{\left(|\hat{x}_{j}^{n+1}| - |\tilde{x}_{j}^{n+1}|\right)\left(|\hat{x}_{j}^{n+1}| + |\tilde{x}_{j}^{n+1}|\right)\wtil_{j}^{n+1}} \right|\\
\le& \frac{\tau}{2}  \left(\sum\limits_{j=1}^{N}{\left| \hat{x}_{j}^{n+1} - \tilde{x}_{j}^{n+1}\right|^2\wtil_{j}^{n+1}}\right)^{\frac12}\left(\sum\limits_{j=1}^{N}{\left|Ê|\hat{x}_{j}^{n+1}| + |\tilde{x}_{j}^{n+1}|\right|^2 \wtil_{j}^{n+1}}\right)^{\frac12} \\
 \leq &\frac{\tau}{2}\max\limits_{\HR{\ell}=1,\ldots,N}\frac{w_{\HR{\ell}}^{n+1}}{w_{\HR{\ell}}^{n}}\left(\sum\limits_{j=1}^{N}{\left| \hat{x}_{j}^{n+1} - \tilde{x}_{j}^{n+1}\right|^2\wtil_{j}^{n}}\right)^{\frac12}\left(\sum\limits_{j=1}^{N}{\left|Ê|\hat{x}_{j}^{n+1}| + |\tilde{x}_{j}^{n+1}|\right|^2 \wtil_{j}^{n}}\right)^{\frac12}\\
  \leq &\frac{\tau}{2}\bar{W}_{n+1}^2\vectornorm{\hat{x}^{n+1}-\tilde{x}^{n+1}}_{\ltwo(w^n)}\vectornorm{|\hat{x}^{n+1}|+|\tilde{x}^{n+1}|}_{\ltwo(w^n)}\\
   \leq &\frac{\tau}{2}\bar{W}_{n+1}^2\sqrt{\tol_{n+1}}\left(\vectornorm{\hat{x}^{n+1}}_{\ltwo(w^n)}+\vectornorm{\tilde{x}^{n+1}}_{\ltwo(w^n)}\right)\\
   \leq &\frac{\tau}{2}\bar{W}_{n+1}^2\sqrt{\tol_{n+1}}\left[ 2 W_n\left(\vectornorm{\tilde{x}^{n}}_{\ltwo(w^{n-1})} + \sqrt{\tol_n}\right) + \sqrt{\tol_{n+1}}\right]\\
 \leq &\frac{\tau}{2}\bar{W}_{n+1}^2\sqrt{\tol_{n+1}}\left[ c_n + \sqrt{\tol_{n+1}}\right]\leq a_{n+1}, \\
  \end{align*}
\endgroup
  where we have used the Cauchy-Schwarz inequality in the first inequality, \eqref{eq:iteratestol1} and \eqref{eq:iteratestol2} in the fifth inequality, \eqref{eq: weight_approx} in the third inequality, \HR{the definition of $c_n$ in \eqref{eq:defcnwn},} and the \HR{Assumption~\eqref{eq:deftol}} on $\textnormal{tol}_{n+1}$ in the last inequality.
  
  Since $1\leq \bar{W}_{n+1}$, we obtain~\eqref{eq: Toleranz2} by
  \begingroup
  \allowdisplaybreaks
  \begin{align*}
 |\sJ \left( \xtil^{n+1}, \wtil^{n}, \etil^{n} \right) - &\sJ \left( \hat{x}^{n+1}, \wtil^{n}, \etil^{n} \right)|  = \left|\frac{\tau}{2} \sum\limits_{j=1}^{N}{\left(|\hat{x}_{j}^{n+1}|^{2} - |\tilde{x}_{j}^{n+1}|^{2}\right)\wtil_{j}^{n}} \right|\\
  & \leq \frac{\tau}{2}\left(\sum\limits_{j=1}^{N}{\left| \hat{x}_{j}^{n+1} - \tilde{x}_{j}^{n+1}\right|^2\wtil_{j}^{n}}\right)^{\frac12}\left(\sum\limits_{j=1}^{N}{\left|Ê|\hat{x}_{j}^{n+1}| + |\tilde{x}_{j}^{n+1}|\right|^2 \wtil_{j}^{n}}\right)^{\frac12}\\
    & \leq \frac{\tau}{2}\bar{W}_{n+1}^2\left(\sum\limits_{j=1}^{N}{\left| \hat{x}_{j}^{n+1} - \tilde{x}_{j}^{n+1}\right|^2\wtil_{j}^{n}}\right)^{\frac12}\left(\sum\limits_{j=1}^{N}{\left|Ê|\hat{x}_{j}^{n+1}| + |\tilde{x}_{j}^{n+1}|\right|^2 \wtil_{j}^{n}}\right)^{\frac12}\\
  & \leq \frac{\tau}{2}\bar{W}_{n+1}^2\sqrt{\tol_{n+1}}\left[ c_n + \sqrt{\tol_{n+1}}\right]\leq a_{n+1},
 \end{align*}
 \endgroup

with the same arguments as above. \HR{Lemma \ref{lem: J_ineq} yields} 
\begin{align*}
\sJ \left( \hat{x}^{n+1}, \wtil^{n+1}, \etil^{n+1} \right) &\le \sJ \left( \tilde{x}^{n+1}, \wtil^{n+1}, \etil^{n+1} \right) + \mathnormal{a}_{n+1} \le \sJ \left( \tilde{x}^{n+1}, \wtil^{n}, \etil^{n+1} \right) + \mathnormal{a}_{n+1} \\
 &\le \sJ \left( \tilde{x}^{n+1}, \wtil^{n}, \etil^{n} \right) + \mathnormal{a}_{n+1} \le \sJ \left( \hat{x}^{n+1}, \wtil^{n}, \etil^{n} \right) + 2\mathnormal{a}_{n+1}, 
 \end{align*}
where the first inequality follows from \eqref{eq: Toleranz}, the second and third by \eqref{sec:eq: J_ineq}, and the last by~\eqref{eq: Toleranz2}. 
\end{proof}

In the above lemma, we showed that the error of the evaluations of the functional $\sJ$ on the approximate solution $\xtil^n$ and the weighted $\ell_2$-minimizer $\hat{x}^n$ can be bounded by choosing an appropriate tolerance in the algorithm. This result will be used to show that the difference between the iterates $\xtil^{n+1}$ and $\xtil^n$ becomes arbitrarily small for $n \rightarrow \infty$, as long as we choose the sequence $(a_{n})_{n\in\N}$ summable. This will be the main result of this section. Before, we
\HR{prove some further auxiliary statements concerning the functional $\sJ(x,\wtil,\etil)$ and the iterates $\xtil^{n}$ and $\wtil^{n}$}. 


\begin{lemma}\label{lem: bounds}
Let $(a_n)_{n\in \HR{\N}}$, $a_n \in \mathbb{R}_+$, be a summable sequence with $A\defleft\sum\limits_{n=1}^{\infty}a_n<\infty$, and define ${C\defleft 3A} + \sJ \left( \tilde{x}^{1}, \wtil^{0}, \etil^{0} \right)$ as in Theorem~\ref{thm: rate of convergence}. For each $n\ge1$ we have 
\begin{align} \label{eq: J_equals_sqrt}
\sJ \left( \tilde{x}^{n+1}, \wtil^{n+1}, \etil^{n+1} \right) &= \sum\limits_{j=1}^{N}{\left(|\tilde{x}_{j}^{n+1}|^{2}+(\etil^{n+1})^{2}\right)^{\frac{\tau}{2}}},\\
\label{eq: 1_norm_estimate}
\left\|\xtil^{n}\right\|_{\ell_{\tau}}^{\tau} &\leq C,\\
\label{eq: weight_bound}
\wtil_{j}^{n} &\ge C^{-\frac{2-\tau}{\tau}}, j = 1, \ldots, N,\text{ and} \\
\label{eq:l2lwNorm}
\vectornorm{x}_{\ltwo} &\leq C^{\frac{2-\tau}{2\tau}}\vectornorm{x}_{\ltwo(w^n)} \text{ for all } x\in\C^N.
\end{align}
\begin{proof}
\HR{Identity} \eqref{eq: J_equals_sqrt} follows by insertion of the definition of $\wtil^{n+1}$ in step 4 of Algorithm~CG-IRLS.

By the minimizing property of $\hat{x}^{n+1}$ and the fact that $\hat{x}^n\in\sF{\Phi}$, we have 
\[\sJ \left( \hat{x}^{n+1}, \wtil^{n}, \etil^{n} \right) \leq \sJ \left( \hat{x}^{n}, \wtil^{n}, \etil^{n} \right), \]
and thus, together with~\eqref{eq: J_xopt_delta}, it follows that
\[\sJ \left( \hat{x}^{n+1}, \wtil^{n+1}, \etil^{n+1} \right) \leq \sJ \left( \hat{x}^{n+1}, \wtil^{n}, \etil^{n} \right)+2a_{n+1} \leq \sJ \left( \hat{x}^{n}, \wtil^{n}, \etil^{n} \right)+2a_{n+1}.\]
\HR{Hence, the telescoping sum} 
\[\sum\limits_{k=1}^{n} \left(\sJ \left( \hat{x}^{k+1}, \wtil^{k+1}, \etil^{k+1} \right) - \sJ \left( \hat{x}^{k}, \wtil^{k}, \etil^{k} \right) \right)\leq 2 \sum\limits_{k=1}^{n} a_{k+1}\]
leads to the estimate
\[\sJ \left( \hat{x}^{n+1}, \wtil^{n+1}, \etil^{n+1} \right) \leq \sJ \left( \hat{x}^{1}, \wtil^{1}, \etil^{1} \right)+2A\leq \sJ \left( \tilde{x}^{1}, \wtil^{0}, \etil^{0} \right)+2A + a_1.\]
\HR{Inequality} \eqref{eq: 1_norm_estimate} then follows from \eqref{eq: Toleranz} and
\begin{align*}
\left\|\xtil^{n+1}\right\|_{\ell_{\tau}}^{\tau} &\le \sum\limits_{j=1}^{N}{\left[|\xtil_{j}^{n+1}|^{2}+(\etil^{n+1})^{2}\right]^{\frac{\tau}{2}}} = \sJ\left(\xtil^{n+1},\wtil^{n+1},\etil^{n+1}\right)\\
&\leq \sJ\left(\hat{x}^{n+1},\wtil^{n+1},\etil^{n+1}\right) + a_{n+1} \leq C, \quad \HR{ \mbox{ for all } n \geq 1}.
\end{align*}
Consequently, the bound \eqref{eq: weight_bound} follows from
\begin{equation*}
(\wtil_{j}^{n})^{-\frac{\tau}{2-\tau}} \leq \frac{2-\tau}{\tau}(\wtil_{j}^{n})^{-\frac{\tau}{2-\tau}} \le \sJ\left(\xtil^{n},\wtil^{n},\etil^{n}\right) \le C.
\end{equation*}
Inequality~\eqref{eq:l2lwNorm} is a direct consequence of~\eqref{eq: weight_bound}.
\end{proof}
\end{lemma}

Notice that \HR{\eqref{eq: 1_norm_estimate} states the boundedness of the iterates.} 
\HR{The lower bound \eqref{eq: weight_bound} on the weights $\wtil^n$ will become useful} in the proof of Lemma \ref{lem: finite}.

By using the estimates collected so far, we can adapt \cite[Lemma 5.1]{dadefogu10} to our situation. First, we shall prove that the differences between the $n$-th $\ell_{2}(\wtil^{n-1})$-minimizer and its successor become arbitrarily small.
\begin{lemma}\label{lem: finite}
\HR{Given} a summable sequence $(a_{n})_{n\in\mathbb{N}}$, $a_n \in \mathbb{R}_+$, the sequence $(\hat{x}^{n})_{n\in\mathbb{N}}$ satisfies
\begin{equation}\label{eq: delta_sum le 2CC}
\sum\limits_{n=1}^{\infty}{\left\|\hat{x}^{n+1} - \hat{x}^{n}\right\|_{\ell_{2}}^{2}} \le \frac{2}{\tau}C^{\frac{2}{\tau}},\end{equation}
where C is the constant of Lemma \ref{lem: bounds} and $\hat{x}^{n} = \argmin\limits_{x\in\sF{\Phi}}\sJ\left(x, \wtil^{n-1}, \etil^{n-1}\right)$. As a consequence we have
\begin{equation}\label{eq: delta_convergence}
\lim\limits_{n\rightarrow\infty}\left\|\hat{x}^{n}-\hat{x}^{n+1}\right\|_{\ell_2} = 0.
\end{equation}
\begin{proof} \HR{We have}
\begin{align*}
&\frac{2}{\tau}\left[\sJ\left(\hat{x}^{n}, \wtil^{n}, \etil^{n}\right) - \sJ\left(\hat{x}^{n+1}, \wtil^{n+1}, \etil^{n+1}\right) + 2\mathnormal{a}_{n+1}\right] \\
\ge\; & \frac{2}{\tau}\left[\sJ\left(\hat{x}^{n}, \wtil^{n}, \etil^{n}\right) - \sJ\left(\hat{x}^{n+1}, \wtil^{n}, \etil^{n}\right)\right] = \left\langle \hat{x}^{n}, \hat{x}^{n}\right\rangle_{\wtil^{n}} - \left\langle \hat{x}^{n+1}, \hat{x}^{n+1}\right\rangle_{\wtil^{n}} = \left\langle \hat{x}^{n}+\hat{x}^{n+1}, \hat{x}^{n}-\hat{x}^{n+1}\right\rangle_{\wtil^{n}} \\
 =\;& \left\langle \hat{x}^{n}-\hat{x}^{n+1}, \hat{x}^{n}-\hat{x}^{n+1}\right\rangle_{\wtil^{n}} = \sum\limits_{i=1}^{N}{\wtil_{j}^{n}|\hat{x}_{j}^{n}-\hat{x}_{j}^{n+1}|^{2}} \ge C^{-\frac{2-\tau}{\tau}}\left\|\hat{x}^{n}-\hat{x}^{n+1}\right\|_{\ell_{2}}^{2}.
\end{align*}
Here we used the fact that $\hat{x}^{n}-\hat{x}^{n+1}\in\mathcal{N}_{\Phi}$ and therefore, $\left\langle \hat{x}^{n+1}, \hat{x}^{n}-\hat{x}^{n+1}\right\rangle = 0$ and in the last step we applied the bound \HR{\eqref{eq:l2lwNorm}}. 
Summing these inequalities over $n\ge1$, we arrive at
\begin{eqnarray*}
\sum\limits_{n=1}^{N}{\left\|\hat{x}^{n}-\hat{x}^{n+1}\right\|_{\ell_{2}}^{2}} &\le& C^{\frac{2-\tau}{\tau}}\sum\limits_{n=1}^{N}{\frac{2}{\tau}\left[\sJ\left(\hat{x}^{n}, \wtil^{n}, \etil^{n}\right) - \sJ\left(\hat{x}^{n+1}, \wtil^{n+1}, \etil^{n+1}\right) + 2\mathnormal{a}_{n+1}\right]} \\
 &\le& \frac{2}{\tau}C^{\frac{2-\tau}{\tau}}\left[\sJ\left(\hat{x}^{1}, \wtil^{1}, \etil^{1}\right) + \sum\limits_{n=1}^{N}{2\mathnormal{a}_{n+1}}\right] \HR{\leq \frac{2}{\tau} C^{\frac{2}{\tau}}}.
\end{eqnarray*}
Letting $N \rightarrow \infty$ yields the desired result. 
\end{proof}
\end{lemma}


The following lemma will play a major role in our proof of convergence since it shows that \HR{not only \eqref{eq: delta_convergence} holds but that also}
the difference between successive iterates becomes arbitrarily small. 

\begin{lemma}\label{lem: delta_xtil}
	Let $\xtil^{n}$ be as described in Algorithm CG-IRLS and $(a_n)_{n\in\mathbb{N}}$ be a summable sequence. Then
	\begin{equation}\label{eq: delta_xtil}
		\lim\limits_{n\rightarrow\infty} \left\|\xtil^n - \xtil^{n+1}\right\|_{\ell_2} = 0.
	\end{equation}
	\end{lemma}
	\begin{proof}
	By \eqref{eq:l2lwNorm} of Lemma~\ref{lem: bounds} and the \HR{condition \eqref{eq:deftol}} on $\tol_n$, we have
	\begin{align*}
	\vectornorm{\hat{x}^n - \tilde{x}^n}_{\ltwo} &\leq C^{\frac{2-\tau}{2\tau}} \vectornorm{\hat{x}^n - \tilde{x}^n}_{\ltwo(w^{n-1})}
		 \leq C^{\frac{2-\tau}{2\tau}} \sqrt{\tol_n}
		 \leq C^{\frac{2-\tau}{2\tau}} \left(-\frac{c_n}{2} +\sqrt{\left(\frac{c_n}{2}\right)^2}+\sqrt{\frac{2a_{n}}{\tau\bar{W}_{n}^2}}\right)\\
		& \leq C^{\frac{2-\tau}{2\tau}}\sqrt{\frac{2}{\tau}} \sqrt{a_n}
	\end{align*}
	since $\bar{W}_{n}\geq 1$ as defined in Lemma~\ref{lem: toleranz}.
	Since $(a_n)_{n\in\N}$ is summable, we conclude that
	\begin{equation}\label{eq: delta_exact_xtil}
		\lim\limits_{n \rightarrow\infty}\left\|\hat{x}^{n} - \xtil^{n}\right\|_{\ell_2} = 0.
	\end{equation}
	Together with Lemma \ref{lem: finite} we can prove our statement:
	\begin{eqnarray*}
		\lim\limits_{n \rightarrow\infty}\left\|\xtil^{n} - \xtil^{n+1}\right\|_{\ell_2}  
		&=& \lim\limits_{n \rightarrow\infty}\left\|\xtil^{n} - \hat{x}^{n} + \hat{x}^{n} - \hat{x}^{n+1} + \hat{x}^{n+1} - \xtil^{n+1}\right\|_{\ell_2}  \\
		&\leq & \lim\limits_{n \rightarrow\infty}\left\|\xtil^{n} - \hat{x}^{n}\right\|_{\ell_2}  + \lim\limits_{n \rightarrow\infty}\left\|\hat{x}^{n} - \hat{x}^{n+1}\right\|_{\ell_2}  + \lim\limits_{n \rightarrow\infty} \left\|\hat{x}^{n+1} - \xtil^{n+1}\right\|_{\ell_2}  \\
		&=& 0,
	\end{eqnarray*}
	where the first and last term vanish because of \eqref{eq: delta_exact_xtil} and the other term due to \eqref{eq: delta_convergence}.
	\end{proof}

\subsubsection{The functional $f_{\etil,\tau}(z)$}
\label{sec:fetil}

In this section, we \HR{introduce an auxiliary} functional which \HR{is useful} for the proof of convergence. 
From the monotonicity of $\etil_{n}$, we know that $\etil = \lim\limits_{n\rightarrow\infty}\etil_{n}$ exists and is nonnegative. We introduce the functional
\begin{equation}\label{eq: f_epsilon}
f_{\etil,\tau}(x) \defleft \sum\limits_{j=1}^{N}|x_{j}^{2} + \etil^{2}|^{\frac{\tau}{2}}.
\end{equation}
Note that if we \HR{would know} that $\xtil^{n}$ converges to $x$, then in view of \eqref{eq: J_equals_sqrt}, $f_{\etil,\tau}(x)$ would be the limit of $\sJ(\xtil^{n}, \wtil^{n}, \etil^{n})$. When $\etil > 0$, \HR{the Hessian is given by} $H(f_{\etil,\tau})(x) = \diag \left[\tau\frac{x_j^2(\tau-1)+\etil^{2}}{|x_{j}^{2}+\etil^{2}|^{\frac{4-\tau}{2}}}\right]_{i=1}^N$. Thus, in particular, $H(f_{\etil,1})(x)$ is strictly positive definite, so that $f_{\etil,1}$ is strictly convex and therefore has a unique minimizer 
\begin{equation}\label{eq: x_epsilon}
x^{\etil,1} \defleft \argmin\limits_{x\in\sF{\Phi}}f_{\etil,1}(x).
\end{equation}
In the case of $0 < \tau < 1$, we denote by $\mathcal{X}_{\etil,\tau}(y)$ the set of global minimizers of $f_{\etil,\tau}$  on $\sF{\Phi}$. For both cases, the minimizers are characterized by the following lemma.


\begin{lemma}\label{lem: tangens}
	Let $\etil > 0$ and $x \in \sF{\Phi}$. If $x=x^{\etil,1}$ or $x\in\mathcal{X}_{\etil,\tau}(y)$, then $\left\langle x,\eta\right\rangle_{\hat{w}(x,\etil,\tau)} = 0$ for all	$\eta \in \mathcal{N}_{\Phi}$, where $\hat{w}(x,\etil,\tau)= \left[\left||x_{i}|^{2} + \etil^{2}\right|^{-\frac{2-\tau}{2}}\right]_{i=1}^N$. In the case of $\tau=1$ also the converse is true.
\end{lemma}
\begin{proof}
The proof is an adaptation of~\cite[Lemma 5.2, Section 7]{dadefogu10} and is presented for \HR{the sake of} completeness in Appendix~\ref{app:lemmafe}.
\end{proof}

\subsubsection{Proof of convergence}
By the results of the previous section, we are able now to prove the convergence of Algorithm~CG-IRLS. The proof is inspired by the ones of \cite[Theorem 5.3, Theorem 7.7]{dadefogu10}, see also \cite[Chapter 15.3]{fora13}, which we adapted to our case. 

	\begin{proof}[{Proof of Theorem~\ref{thm: convergence}}]
	Since $0\leq\etil^{n+1} \le \etil^{n}$ the \HR{sequence $(\etil^{n})_{n\in\N}$} always converges \HR{to some $\varepsilon > 0$}. 
	

	\hspace*{3mm} \textbf{Case} $\etil = 0$: Following the first part of the proof of~\cite[Theorems~5.3 and 7.7]{dadefogu10}, where \HR{the boundedness of the sequence $\xtil^n$ and the definition of $\varepsilon^n$ is used}, we can show that there is a subsequence $(\tilde{x}^{p_j})_{p_j\in\N}$ of $(\tilde{x}^n)_{n\in\N}$ such that $\tilde{x}^{p_j}\rightarrow \bar{x}\in \sF{\Phi}$ and $\bar{x}$ is the unique $\ell_{\tau}$-minimizer. \HR{It remains} to show that also $\tilde{x}^{n}\rightarrow\bar{x}$. \HR{To this end}, 
we \HR{first notice} that $\xtil^{p_j} \rightarrow \bar{x}$ and $\etil^{p_j}\rightarrow 0$ imply $\sJ\left(\xtil^{p_j}, \wtil^{p_j}, \etil^{p_j}\right) \rightarrow \|\bar{x}\|_{\ltau}^{\tau}$. The convergence of $\sJ\left(\xtil^n, \wtil^n, \etil^n\right) \rightarrow \|\bar{x}\|_{\ltau}^{\tau}$ is established by the following argument: For each $n\in\N$ there is exactly one \HR{$i=i(n)$} such that $p_i< n \leq p_{i+1}$. We 
\HR{use \eqref{eq: J_xopt_delta} and \eqref{eq: Toleranz} to estimate} the telescoping sum
\begin{align*}
\left| \sJ  \left( \tilde{x}^{n}, \wtil^{n}, \etil^{n} \right) - \sJ \left( \tilde{x}^{p_{i(n)}}, \wtil^{p_{i(n)}}, \etil^{p_{i(n)}} \right) \right| & \leq \sum\limits_{k=p_i}^{n-1} \left|\sJ \left( \tilde{x}^{k+1}, \wtil^{k+1}, \etil^{k+1} \right) - \sJ \left( \tilde{x}^{k}, \wtil^{k}, \etil^{k} \right) \right| \\
& 
\leq 4 \sum\limits_{k=p_{\HR{i(n)}}}^{n-1} a_{k+1}. 
\end{align*}
\HR{Since $\sum_{k=0}^\infty a_k < \infty$ this implies that 
$\lim_{n \to \infty} \left| \sJ  \left( \tilde{x}^{n}, \wtil^{n}, \etil^{n} \right) - \sJ \left( \tilde{x}^{p_{\HR{i(n)}}}, \wtil^{p_{i(n)}}, \etil^{p_{i(n)}} \right) \right| = 0$
so that} 
\[
\HR{\lim_{n \to \infty} \sJ  \left( \tilde{x}^{n}, \wtil^{n}, \etil^{n} \right) =  \|\bar{x}\|_{\ltau}^{\tau}.}
\]

Moreover \eqref{eq: J_equals_sqrt} implies
	\begin{equation*}
		\sJ\left(\xtil^n, \wtil^n, \etil^n\right) - N(\etil^n)^{\tau} 
		\le 
		\|\xtil^{n}\|_{\ltau}^{\tau}
		\le
		\sJ\left(\xtil^n, \wtil^n, \etil^n\right),
	\end{equation*}
	and thus, $\|\xtil^{n}\|_{\ltau}^{\tau} \rightarrow \|\bar{x}\|_{\ltau}^{\tau}$. Finally we invoke Lemma~\ref{lem:NSP:invtriangle} 
	with $z' = \xtil^n$ and $z=\bar{x}$ to obtain
	\begin{equation*}
		\limsup\limits_{n\rightarrow\infty} \|\xtil^n - \bar{x}\|_{\ltau}^{\tau} 
		\le
		\frac {1+\gamma}{1-\gamma}\left(\lim\limits_{n\rightarrow\infty}\|\xtil^n\|_{\ltau}^{\tau} - \|\bar{x}\|_{\ltau}^{\tau}\right) = 0,
	\end{equation*}
	which completes the proof of $\xtil^n\rightarrow \bar{x}$ in this case. 
	To see \eqref{eq: theorem_statement_1} and establish (i), invoke Lemma \ref{lem: NSP}. 
	
	
	\hspace{3mm}\textbf{Case} $\etil > 0$: By Lemma \ref{lem: bounds}, we know that $(\xtil^n)_{n\in\N}$ is a bounded sequence and hence has accumulation points. Let $(\xtil^{n_i})$ be any convergent subsequence of $(\xtil^{n})_{n\in\N}$ and let $\bar{x}\in\mathcal{Z}_{\tau}(y)$ its limit. By \eqref{eq: delta_exact_xtil}, we know that also $\bar{x}\in\sF{\Phi}$. Following the proof of~\cite[Theorem 5.3 and Theorem 7.7]{dadefogu10}, one shows that $\left\langle \bar{x},\eta\right\rangle_{\hat{w}(\bar{x},\etil,\tau)} = 0$ for all $\eta\in\mathcal{N}_{\Phi}$, \HR{where $\hat{w}(\bar{x},\etil,\tau)$ is defined as in Lemma~\ref{lem: tangens}}.

	In the case of $\tau = 1$, \HR{Lemma \ref{lem: tangens} implies} $\bar{x} = x^{\etil,1}$. Hence, $x^{\etil,1}$ is the unique accumulation point of $(\xtil^{n})_{n \in \mathbb{N}}$. This establishes (ii). 
	
	To prove (iii), assume that $\bar{x} \in \mathcal{Z}_{\tau}(y) \cap \mathcal{X}_{\etil,\tau}(y)$, and follow the proof of~\cite[Theorem 5.3, and 7.7]{dadefogu10} to conclude.	
	\end{proof}

\subsubsection{Proof of rate of convergence}

\HR{The proof follows} similar steps as in \cite[Section 6]{dadefogu10}. We define the auxiliary sequences of error vectors $\tilde{\eta}^n \defleft \tilde{x}^n - x^*$ and $\hat{\eta}^n \defleft \hat{x}^n - x^*$.
\begin{proof}[Proof of Theorem~\ref{thm: rate of convergence}]
	We apply the characterization \eqref{eq: orthogonality} with $w = \wtil^n$, $\hat{x} = \hat{x}^{n+1} = x^* + \hat{\eta}^{n+1}$, and $\eta = \hat{x}^{n+1} - x^* = \hat{\eta}^{n+1}$, which gives
	\begin{equation*}
		\sum\limits_{j=1}^{N}(x_j^* + \hat{\eta}_j^{n+1})\hat{\eta}_j^{n+1}w_j^{n} = 0. 
	\end{equation*}
	Rearranging the terms and using the fact that $x^*$ is supported on $\Lambda$, we obtain
	\begin{equation} \label{eq: sum eta square w}
		\sum\limits_{j=1}^{N}|\hat{\eta}_j^{n+1}|^2w_j^{n} = - \sum\limits_{j=1}^{N}x_j^*\hat{\eta}_j^{n+1}w_j^{n} = -\sum\limits_{j\in\Lambda}\frac{x_j^*}{[|\xtil_j^n|^2 + (\etil^n)^2]^{\frac{2-\tau}{2}}}\hat{\eta}_j^{n+1}.
	\end{equation}
	By assumption there exists $n_0$ such that $E_{n_0}\leq R^*$. We prove~\eqref{eq:errordynamicsltau}, and $E_{n} \leq R^* \Rightarrow E_{n+1} \leq R^*$ to obtain the validity for all $n\geq n_0$.
	\HR{Assuming $E_n \leq R^*$,} we have for all $j\in \Lambda$,
	\begin{equation*}
		|\tilde{\eta}_j^n| \le \|\tilde{\eta}^n\|_{\ltau} = \sqrt[\tau]{E_n} \le \nu |x_j^*|, 
	\end{equation*} 
	and thus
	\begin{equation*}
		|\xtil_j^n| 
		=
		|x_j^* +\tilde{\eta}_j^n|
		\geq 
		|x_j^*| - |\tilde{\eta}_j^n|
		\geq 
		|x_j^*| - \nu |x_j^*|,
	\end{equation*}
	so that
	\begin{equation} \label{eq: 1/1-mu}
		\frac{|x_j^*|}{[|\xtil_j^n|^2 + (\etil^n)^2]^{\frac{2-\tau}{2}}} \le \frac{|x_j^*|}{|\xtil_j^n|^{2-\tau}} \le \frac{1}{(1-\nu)^{2-\tau}|x_j^*|^{1-\tau}}.
	\end{equation}
	Hence, \eqref{eq: sum eta square w} combined with \eqref{eq: 1/1-mu} and the NSP leads to
	\begin{align*}
		& \left(\sum\limits_{j=1}^N |\hat{\eta}_j^{n+1}|^2\wtil_j^n\right)^{\tau} \le \left((1-\nu)^{2-\tau}\left(\min\limits_{j\in\Lambda}|x_j^*|\right)^{1-\tau}\right)^{-\tau} \|\hat{\eta}_\Lambda^{n+1}\|_{\ell_1}^{\tau} \\
		\le\;& \left((1-\nu)^{(2-\tau)}\left(\min\limits_{j\in\Lambda}|x_j^*|\right)^{(1-\tau)}\right)^{-\tau}\|\hat{\eta}_\Lambda^{n+1}\|_{\ltau}^{\tau}
		  \leq  \frac{\gamma }{(1-\nu)^{\tau(2-\tau)}\left(\min\limits_{j\in\Lambda}|x_j^*|\right)^{\tau(1-\tau)}} \|\hat{\eta}_{\Lambda^c}^{n+1}\|_{\ltau}^{\tau}.
	\end{align*}
\begingroup
\allowdisplaybreaks
	Combining \cite[Proposition 7.4]{dadefogu10} with the above estimate yields

	\begin{align} \label{eq: eta_square bound}
		\|\hat{\eta}_{\Lambda^c}^{n+1}\|_{\ltau}^{2\tau} 
		& =  \vectornorm{\left[\hat{\eta}_i^{n+1}(\wtil_i^n)^{-\frac{1}{\tau}}\right]_{i\in\Lambda^c}}_{\ltau(w^n)}^{2\tau} 
		 \leq   \vectornorm{\hat{\eta}^{n+1}_{\Lambda^c}}_{\ltwo(w^n)}^{2\tau} \vectornorm{\left[(\wtil_i^n)^{-\frac{1}{\tau}}\right]_{i\in\Lambda^c}}_{\ell_{\frac{2\tau}{2-\tau}}(w^n)}^{2\tau} \nonumber \\
		&\leq
		\left( \sum\limits_{j = 1}^N |\hat{\eta}_j^{n+1}|^2\wtil_j^n \right)^{\tau} 
		\left( \sum\limits_{j \in \Lambda^c} \left[|\tilde{\eta}_j^n| + \etil^n\right]^{\tau} \right)^{2-\tau}	\nonumber \\
		&\leq	
		\frac{\gamma }{(1-\nu)^{\tau(2-\tau)}\left(\min\limits_{j\in\Lambda}|x_j^*|\right)^{\tau(1-\tau)}}\|\hat{\eta}_{\Lambda^c}^{n+1}\|_{\ltau}^{\tau} \left(\|\tilde{\eta}^n\|_{\ltau}^{\tau} + (N-k)\left(\etil^n\right)^{\tau}\right)^{2-\tau}.
	\end{align}
\endgroup
It follows that
	\begin{equation*}
		\|\hat{\eta}_{\Lambda^c}^{n+1}\|_{\ltau}^{\tau} \le \frac{\gamma }{(1-\nu)^{\tau(2-\tau)}\left(\min\limits_{j\in\Lambda}|x_j^*|\right)^{\tau(1-\tau)}} \left(\|\tilde{\eta}^n\|_{\ltau}^{\tau} + (N-k)\left(\etil^n\right)^{\tau}\right)^{2-\tau}.
	\end{equation*}
	Note that this is also valid \HR{if} $\hat{\eta}_{\Lambda^c}^{n+1} = 0$ since \HR{then} 
	the left-hand side is zero and the right-hand side non-negative. We furthermore obtain
	\begin{align} 
		\nonumber \|\hat{\eta}^{n+1}\|_{\ltau}^{\tau}
		&= 
		\|\hat{\eta}_{\Lambda}^{n+1}\|_{\ltau}^{\tau}+\|\hat{\eta}_{\Lambda^c}^{n+1}\|_{\ltau}^{\tau} 
		\le 
		(1+\gamma)\|\hat{\eta}_{\Lambda^c}^{n+1}\|_{\ltau}^{\tau}\\
		\label{eq: eta_one bound}&\le
		\frac{\gamma (1+\gamma)}{(1-\nu)^{\tau(2-\tau)}\left(\min\limits_{j\in\Lambda}|x_j^*|\right)^{\tau(\tau-1)}} \left(\|\tilde{\eta}^n\|_{\ltau}^{\tau} + (N-k)\left(\etil^n\right)^{\tau}\right)^{2-\tau}.
	\end{align}
	In addition to this, we know by \cite[Lemma 4.1, 7.5]{dadefogu10}, that
\begin{equation}
	(J-j)r(x)_J^{\tau} \le \|x-x'\|_{\ltau}^{\tau} + \sigma_{j}(x')_{\ltau}.
\end{equation}	
for any $J > j$ and $x, x'\in\C^N$. Thus, we have by the definition of $\etil^n$ in step \ref{def: etil} of Algorithm~CG-IRLS that
	\begin{align}
		\nonumber (N-k)(\etil^n)^{\tau} &\leq (N-k)\beta^{\tau}\left(r(\xtil^n)_{K+1}\right)^{\tau} \le \frac{(N-k)\beta^{\tau}}{K+1-k}(\| \xtil^n - x^* \|_{\ltau}^{\tau} + \sigma_k(x^*)_{\ltau}) \\
		\label{eq:Nkcalc}& = \frac{(N-k)\beta^{\tau}}{K+1-k} \| \tilde{\eta}^n \|_{\ltau}^{\tau}
	\end{align}
	since by assumption $\sigma_k(x^*)_{\ltau} = 0$.	
	Together with \eqref{eq: eta_one bound} this yields
	\begin{align*}
	    \|\hat{\eta}^{n+1}\|_{\ltau}^{\tau} & \le \frac{\gamma (1+\gamma) }{(1-\nu)^{\tau(2-\tau)}\left(\min\limits_{j\in\Lambda}|x_j^*|\right)^{\tau(1-\tau)}}  \left(1 +\frac{(N-k)\beta^{\tau}}{K+1-k}\right)^{2-\tau}\|\tilde{\eta}^{n}\|_{\ltau}^{\tau(2-\tau)}\\
		& \le \mu E_n^{2-\tau}. 
	\end{align*}
	Finally, we obtain~\eqref{eq:errordynamicsltau} by
	\begin{align*}
	E_{n+1} & = \vectornorm{\tilde{\eta}^{n+1}}_{\ltau}^{\tau} 
	 \leq \vectornorm{\hat{\eta}^{n+1}}_{\ltau}^{\tau} + \vectornorm{\tilde{x}^{n+1}-\hat{x}^{n+1}}_{\ltau}^{\tau} \leq \vectornorm{\hat{\eta}^{n+1}}_{\ltau}^{\tau} + N^{1-\frac{\tau}{2}}\vectornorm{\tilde{x}^{n+1}-\hat{x}^{n+1}}_{\ltwo}^{\tau}\\
	& \leq \vectornorm{\hat{\eta}^{n+1}}_{\ltau}^{\tau} + (NC)^{1-\frac{\tau}{2}}\vectornorm{\tilde{x}^{n+1}-\hat{x}^{n+1}}_{\ltwo(w^n)}^{\tau} 
	 \leq  \mu E_n^{2-\tau} + (NC)^{1-\frac{\tau}{2}}(\tol_{n+1})^{\frac{\tau}{2}},
	\end{align*}
	where we used \HR{the} triangle inequality in the first inequality, \eqref{eq:l2lwNorm} in the third inequality, and $C$ is the constant from Lemma~\ref{lem: bounds}.

Equation~\eqref{eq:errordynamicsltaugeneral} then follows by condition~\eqref{tolrec}. By means of~\eqref{eq:tildemu}, we obtain 
	\[E_{n+1}\leq\tilde{\mu}E_n^{2-\tau}  \leq \tilde{\mu}\left(R^*\right)^{2-\tau} \leq R^*,\]
	and therefore the linear convergence for $\tau = 1$, and the super-linear convergence for $\tau < 1$ as soon as $n\geq n_0$.
\end{proof}

\section{Conjugate gradient acceleration of IRLS  method for $\ltau$-norm regularization}
\label{sec:IRLSgenltau}
In the previous chapter the solution $x^*$ was intended to solve the linear system $\Phi x = y$ exactly. In most engineering and physical applications such a setting
\HR{may not be required} since the measurements are perturbed by noise. \HR{In this context, it is more appropriate to work with} a functional 
that balances the residual error in the linear system with an $\ltau$-norm penalty, promoting sparsity. We consider the problem
\begin{equation}
\label{eq:prob}
\min\limits_{x}^{} \left ( F_{\tau,\lambda}(x):=||x||_{\ell_{\tau}}^{\tau} + \frac{1}{2 \lambda}||\Phi x-y||_{\ell_2}^{2} \right),
\end{equation}
where $\lambda > 0$, $\Phi  \in \C^{m\times N}$, $y\in\C^m$ is a given measurement vector, and $0 < \tau \leq 1$.
\begin{definition}
Given a real number $\varepsilon > 0$, $x \in \C^N$, and a weight vector $w\in\mathbb{R}^{N}$, $w>0$, we define
\begin{equation}
\label{eq:deffunctional}
J_{\tau,\lambda}(x,w,\varepsilon):=\frac{\tau}{2}\sum\limits_{j=1}^{N}\left[|x_j|^2 w_j + \varepsilon^2 w_j + \frac{2-\tau}{\tau}w_j^{-\frac{\tau}{2-\tau}}\right] + \frac{1}{2\lambda} ||\Phi x-y||_{\ell_2}^2.
\end{equation}
\end{definition}
Lai, Xu, and Yin in~\cite{xulaiyin} and \Rev{Daubechies and Voronin in~\cite{sergei,davo15}} showed independently that \HR{computing the optimizer of the} problem~\eqref{eq:prob} can be approached by an alternating minimization of the functional $J_{\tau,\lambda}$ with respect to $x$, $w$, and $\varepsilon$. The difference between these two works is the definition of the update rule for $\varepsilon$. Here, we chose the rule \HR{in step~\ref{eq:step4} of Algorithm~\ref{alg:IRLS}} proposed by \Rev{Daubechies and Voronin} because it allows us to show that the algorithm \HR{converges} to a minimizer of~\eqref{eq:prob} for $\tau=1$ and to critical points of ~\eqref{eq:prob} for $\tau <1$ (more precise statements will be given below). However, we \HR{were not able to} prove similar statements 
\HR{for} the rule of Lai, Xu, and Yin. \HR{It} only allows to show the convergence of the algorithm to a critical point of the smoothed functional 
\begin{equation*}
\min\limits_{x}^{} ||x||_{\ell_{\tau},\varepsilon}^{\tau} + \frac{1}{2 \lambda}||\Phi x-y||_{\ell_2}^{2},
\end{equation*}
where 
$||x||_{\ell_{\tau},\varepsilon}^{\tau}:=\sum\limits_{j=1}^{N}|x_j^2+\varepsilon^2|^{\frac{\tau}{2}}$ \HR{with $\varepsilon = \lim_{n \to \infty} \varepsilon^n$}. 
\begin{algorithm}[H]
\caption{IRLS-$\lambda$}
\label{alg:IRLS}
\begin{algorithmic}[1]
	\STATE Set $w^0:=(1,\ldots,1)$, $\varepsilon^0:= 1$, $\alpha\in(0,1]$, $\phi\in(0,\frac{1}{4-\tau})$.
	\WHILE{$\varepsilon^n > 0$}
		\STATE $x^{n+1}:=\underset{x}{\argmin} J_{\tau,\lambda}(x, w^n, \varepsilon^n)$
        \STATE $\varepsilon^{n+1} := \min\left\{\varepsilon^n, | J_{\tau,\lambda}(\tilde{x}^{n-1},w^{n-1},\varepsilon^{n-1}) - J_{\tau,\lambda}(\tilde{x}^{n},w^{n},\varepsilon^{n})|^{\phi} +\alpha^{n+1}\right\}$ \label{eq:step4}
		\STATE $w^{n+1} := \underset{w>0}{\argmin}J_{\tau,\lambda}(x^{n+1},w,\varepsilon^{n+1})$
	\ENDWHILE
\end{algorithmic}
\end{algorithm}

We approach the first step of the algorithm by 
\HR{computing} a critical point of $J_{\tau,\lambda}(\cdot,w,\varepsilon)$ \HR{via} the first order optimality condition
\begin{equation}
\label{eq:lspre}
\tau \left[ x_j w_j^n\right]_{j=1,\ldots,N} + \frac{1}{\lambda}\Phi^*(\Phi x-y) = 0,
\end{equation}
or equivalently 
\begin{equation}
\label{eq:ls}
\left(\Phi^* \Phi  + \diag\left[\lambda\tau w_j^n\right]_{j=1}^N\right)x = \Phi^* y.
\end{equation}
We denote the solution of this system by $x^{n+1}$.
The new weight $w^{n+1}$ is obtained in step 3 and can be expressed componentwise by
\begin{equation}
w_j^{n+1} = ((x_j^{n+1})^2+(\varepsilon^{n+1})^2)^{-\frac{2-\tau}{2}}.
\end{equation}
\HR{Similarly to the previous section} we propose the combination of Algorithm~\ref{alg:IRLS} with the CG method. CG is used to calculate an approximation of the solution of the linear system \eqref{eq:ls} in line 3 of the algorithm. After including the CG method, the modified algorithm which we shall consider is Algorithm CG-IRLS-$\lambda$. 

\begin{algorithm}
\caption{CG-IRLS-$\lambda$}
\label{alg:IRLSCG}
\begin{algorithmic}[1]
	\STATE Set $w^0:=(1,\ldots,1)$, $\varepsilon^0:= 1$, $\alpha\in(0,1]$, $\phi\in(0,\frac{1}{4-\tau})$.
	\WHILE{$\varepsilon^n > 0$}
		\STATE Compute $\tilde{x}^{n+1}$ by means of CG, s.t.  $||\tilde{x}^{n+1}-\hat{x}^{n+1}||_{\ell_2(w^n)}\leq \text{tol}_{n+1}$, \\where $\hat{x}^{n+1}:=\underset{x}{\argmin} J_{\tau,\lambda}(x, w^n, \varepsilon^n)$. \HR{Use $\tilde{x}^{n}$ as the initial vector for CG.}
		\STATE $\varepsilon^{n+1} := \min\left\{\varepsilon^n, | J_{\tau,\lambda}(\tilde{x}^{n-1},w^{n-1},\varepsilon^{n-1}) - J_{\tau,\lambda}(\tilde{x}^{n},w^{n},\varepsilon^{n})|^{\phi} +\alpha^{n+1}\right\}$
		\STATE $w^{n+1} := \underset{w>0}{\argmin}J_{\tau,\lambda}(\tilde{x}^{n+1},w,\varepsilon^{n+1})$
	\ENDWHILE
\end{algorithmic}
\end{algorithm}

Notice that $\tilde{x}$ always denotes the approximate solution of the minimization with respect to $x$ in line 3 and $\hat{x}$ the corresponding exact solution. Thus $\hat{x}^{n+1}$ fulfills~\eqref{eq:ls} but not $\tilde{x}^{n+1}$. 

Theorem~\ref{thm:quarteroni} \HR{provides} a stopping condition for the CG method, \HR{but as in the previous section} it is not practical for us, 
since we do not dispose of the minimizer and the computation of the condition number is computationally expensive. Therefore, we provide \HR{an alternative} stopping criterion to make sure that $\|\tilde{x}^{n+1}-\hat{x}^{n+1}\|_{\ell_2(w^n)}\leq \text{tol}_{n+1}$ is fulfilled in line 3 of Algorithm CG-IRLS-$\lambda$. 

 Let $\tilde{x}^{n+1,l}$ be the $l$-th iterate of the CG method and define 
\[A_n\defleft \Phi^* \Phi  + \diag\left[\lambda\tau w_j^n\right]_{j=1}^N.\]
Notice that the matrix $\Phi^* \Phi$ is positive semi-definite and $\lambda\tau D_n^{-1} = \lambda\tau \diag\left[w_j^n\right]_{j=1}^N$ is positive definite. Therefore, $A_n$ is positive definite and invertible, and furthermore
\begin{equation}
\label{eq:evcomp}
\lambda_{\min}(A_n)\geq \lambda_{\min}(\diag\left[\lambda\tau w_j^n\right]_{j=1}^N).
\end{equation}
We obtain
\begin{equation}
\label{eq:stopcrit}
\vectornorm{\hat{x}^{n+1}-\tilde{x}^{n+1,l}}_{\ell_2(w^n)}\leq \vectornorm{A_n^{-1}\left(\Phi^* y-A_n\tilde{x}^{n+1,l}\right)}_{\ell_2(w^n)} \leq \vectornorm{D_n^{-\frac12}}\vectornorm{A_n^{-1}}\vectornorm{r^{n+1,l}}_{\ell_2},
\end{equation}

where $r^{n+1,l}\defleft \Phi^* y-A_n\tilde{x}^{n+1,l}$ is the residual as it appears in line 5 of Algorithm~\ref{CG}. The first factor on the right-hand side of \eqref{eq:stopcrit} can be estimated by
\[\vectornorm{D_n^{-\frac12}}=\lambda_{\max}\left(D_n^{-\frac12}\right) = \sqrt{\max\limits_{j}w_j^n} = \sqrt{\max\limits_{j}\left(\left(\tilde{x}_j^n\right)^2 + \left(\varepsilon^n\right)^2\right)^{-\frac{2-\tau}{2}}}\leq \left(\varepsilon^n\right)^{-\frac{2-\tau}{2}}.\]
The second factor of \eqref{eq:stopcrit} is estimated by
\begin{align*}
\vectornorm{A_n^{-1}} &= \left(\lambda_{\min}(A_n)\right)^{-1} \leq \left(\lambda_{\min}(\diag\left[\lambda\tau w_j^n\right]_{j=1}^N)\right)^{-1} = \left(\lambda\tau \left(\left(\max\limits_{j}|\tilde{x}_j^n|\right)^2 + \left(\varepsilon^n\right)^2\right)^{-\frac{2-\tau}{2}}\right)^{-1},
\end{align*}

where we used \eqref{eq:evcomp} in the inequality. Thus, we obtain
\[ \vectornorm{\tilde{x}^{n+1}-\hat{x}^{n+1,l}}_{\ell_2(w^n)}\leq \frac{ \left(\left(\max\limits_{j}|\tilde{x}_j^n|\right)^2 + \left(\varepsilon^n\right)^2\right)^{\frac{2-\tau}{2}}}{\left(\varepsilon^n\right)^{\frac{2-\tau}{2}} \lambda\tau} \vectornorm{r^{n+1,l}}_{\ell_2},\]
and the suitable stopping condition
\begin{equation}
\label{eq:stopcritR}
 \vectornorm{r^{n+1,l}}_{\ell_2}\leq \frac{\left(\varepsilon^n\right)^{\frac{2-\tau}{2}} \lambda\tau}{ \left(\left(\max\limits_{j}|\tilde{x}_j^n|\right)^2 + \left(\varepsilon^n\right)^2\right)^{\frac{2-\tau}{2}}}\;\tol_{n+1}.
 \end{equation}

In the remainder of this section, we clarify how to choose the tolerance $\tol_{n+1}$, and establish a convergence result of the algorithm. In the case of $\tau=1$, the problem~\eqref{eq:prob} is the minimization of the well-known LASSO functional. It is convex, and the optimality conditions can be stated in terms of subdifferential inclusions. We are able to show that at least a subsequence of the algorithm is converging to a solution of~\eqref{eq:prob}. If $0<\tau<1$, the problem is non-convex and non-smooth. Necessary first order optimality conditions for a global minimizer of this functional were derived in~\cite[Proposition 3.14]{BrediesLorenz14}, and~\cite[Theorem 2.2]{ItoKunisch14}. In our case, we are able to show that the non-zero components of the limits of the algorithm fulfill the respective conditions. However, as soon as the algorithm is producing zeros in some components of the limit, so far, we were not able to verify the conditions mentioned above. On this account, we pursue a different strategy, which originates from~\cite{Zarzer09}. We do not directly show that the algorithm computes a solution of problem~\eqref{eq:prob}. Instead we show that a subsequence of the algorithm is at least computing a point $x^{\dagger}$, whose transformation $\breve{x}^{\dagger}= \mathcal{N}_{\upsilon/\tau}^{-1}(x^{\dagger})$ is a {\it critical point} of the new functional 
\begin{equation}
\label{eq:transprob}
 \breve{F}_{\upsilon,\lambda}(x):=\vectornorm{x}_{\ell_{\upsilon}}^{\upsilon} +\frac{1}{2\lambda}\vectornorm{\Phi\mathcal{N}_{\upsilon / \tau}(x)-y}_{\ell_2}^2,
 \end{equation}
 where 
 \begin{equation}
 \label{eq:bijmap}
 \mathcal{N}_{\zeta}\colon \HR{\C^N}\rightarrow\HR{\C^N},\quad\left(\mathcal{N}_{\zeta}(x)\right)_j := \sign(x_j)|x_j|^{\zeta},\quad j=1,\ldots,N,
 \end{equation}
  is a continuous bijective mapping and $1 < \upsilon\leq 2$.
It was shown in~\cite{Zarzer09,RamlauZarzer12} that assuming $\breve{x}^{\dagger}$ is a global minimizer of $\breve{F}_{\upsilon,\lambda}(x)$ implies that $x^{\dagger}$ is a global minimizer of $F_{\tau,\lambda}$, i.e., a solution of problem~\eqref{eq:prob}. Furthermore, it was also shown that this result can be partially extended to local minimizers. We comment \HR{on} this issue in Remark~\ref{rem:locminramlau}. These considerations allow us to state the main convergence result.
\begin{theorem}
\label{thm:conv}
Let $0<\tau\leq 1$, $\lambda > 0$, $\Phi  \in \C^{m\times N}$, and $y\in\C^m$. Define the sequences $(\tilde{x}^n)_{n\in\N}$, $(\varepsilon^n)_{n\in\N}$ and $(w^n)_{n\in\N}$ as \HR{the ones} generated by Algorithm CG-IRLS-$\lambda$. 
Choose  the accuracy $\tol_n$ of the CG-method, such that 
\begin{eqnarray}
\label{eq:tol}
\tol_{n} &\leq & \min\left\{a_{n}\left(\sqrt{2\bar{J}\tau} C_{w^{n-1}}+ 2\sqrt{\frac{2\bar{J}}{\lambda}} \sqrt{\left(\frac{2-\tau}{\tau\bar{J}}\right)^{-\frac{2-\tau}{\tau}}} ||\Phi||\right)^{-1},\right.\nonumber \\
& &\phantom{xxxx} \left.\sqrt{a_{n}}\left(\frac{\tau}{2} +\frac{||\Phi||^2}{2\lambda}\left(\frac{2-\tau}{\tau\bar{J}}\right)^{-\frac{2-\tau}{\tau}}\right)^{-\frac{1}{2}}\right\},\\ 
\label{eq:tolcwn} \text{ with }C_{w^{n-1}} &\defleft &\left(\frac{\max\limits_{j}(\tilde{x}_j^{n-1})^2 +(\varepsilon^{n-1})^2}{(\varepsilon^{n})^2}\right)^{1-\frac{\tau}{2}},
\end{eqnarray}
where \HR{$(a_n)_{n\in\N}$ is a positive sequence satisfying $\sum_{n=0}^\infty a_n < \infty$} and $\bar{J}:=  J_{\tau, \lambda}(\tilde{x}^{1},w^{0},\varepsilon^{0})$.

Then the sequence $(\tilde{x}^n)_{n\in\N}$ has at least one convergent subsequence $(\tilde{x}^{n_k})_{n_k\in\N}$. In the case that $\tau = 1$ and $x^\lambda \neq 0$, any convergent subsequence is such that its limit $x^{\lambda}$ is a minimizer of $F_{1,\lambda}(x)$. 
In the case that $0<\tau <1$, the subsequence $(\tilde{x}^{n_k})_{n_k\in\N}$ can be chosen such that the transformation of its limit $\breve{x}^{\lambda}:=\mathcal{N}_{\upsilon / \tau}^{-1}(x^{\lambda})$, $1<\upsilon\leq 2$, as defined in~\eqref{eq:bijmap}, is a critical point of~\eqref{eq:transprob}. If $\breve{x}^{\lambda}$ is  a global minimizer of~\eqref{eq:transprob}, then $x^{\lambda}$ is also a global minimizer of $F_{\tau,\lambda}(x)$.
\end{theorem}
\Rev{
\begin{remark}
Note that the bound~\eqref{eq:tol} on $\tol_n$ is---in contrast to the one in Theorem~\ref{thm: convergence}---not implicit. Although $\tol_n$ depends on $\varepsilon^{n}$, the latter only depends on $\tilde{x}^{n-1}$, $\varepsilon^{n-1}$, $w^{n-1}$, and $\tilde{x}^{n-2}$, $\varepsilon^{n-2}$, $w^{n-2}$. Since in particular $\varepsilon^n$ does not depend on $\tilde{x}^n$, we are able to exchange the steps 3 and 4 in Algorithm~\ref{alg:IRLSCG}.

As we argued in Remark~\ref{rem:loose_tol}, a possible relaxation of the tolerance bound~\eqref{eq:deftol} is allowed to further boost the convergence, the same applies to the bound~\eqref{eq:tol}. 
\end{remark}} 
\begin{remark}
In the case $0 < \tau < 1$, the theorem includes the possibility that there may exist several converging subsequences with different limits. Potentially only one of these limits may have the nice property that its transformation is a critical point. In the proof of the theorem, which follows further below, an appropriate subsequence is constructed. Actually this construction leads to the following hint, how to practically choose the subsequence: Take a converging subsequence $x_{n_l}$ for which the $n_l$ satisfy equation~\eqref{eq:lowerboundepsequality}.
\end{remark}
\HR{It will be important below that a minimizer $x^\sharp$ of $F_{1,\lambda}(x)$ is characterized by the conditions 
\begin{align}
\label{eq:lassominnonzero}-(\Phi^\HR{*}(y - \Phi x^{\sharp}))_j &= \lambda\sign(x^\sharp_j) \quad \mbox{if } x^{\sharp}_j\neq 0, \\
\label{eq:lassominzero}|(\Phi^\HR{*}(y - \Phi x^{\sharp}))_j| &\leq \lambda \quad \mbox{if } x^{\sharp}_j = 0.
\end{align}
Note that in the (less important) case $x^\lambda = 0$, our theorem does not give a conclusion about $x^\lambda$ being a minimizer of $F_{1,\lambda}(x)$.
}

\begin{remark}
\label{rem:locminramlau}
The result of Theorem~\ref{thm:conv} for $0<\tau<1$ can be partially extended towards local minimizers. For \HR{the sake of completeness} 
we sketch the argument from~\cite{RamlauZarzer12}. Assume that $\breve{x}^{\lambda}$ is a local minimizer. Then there is a neighborhood $U_{\epsilon}(\breve{x}^{\lambda})$ with $\epsilon>0$ such that for all $x'\in U_{\epsilon}(\breve{x}^{\lambda})$: 
$$ \breve{F}_{\upsilon,\lambda}(x') \geq  \breve{F}_{\upsilon,\lambda}(\breve{x}^{\lambda}).$$
By continuity of $\mathcal{N}_{\upsilon / \tau}$ there \HR{exists} an $\hat{\epsilon}>0$ such that the neighborhood $U_{\hat{\epsilon}}(x^{\lambda})\subset \mathcal{N}_{\upsilon / \tau}(U_{\epsilon}(\breve{x}^{\lambda}))$. Thus, for all $x\in U_{\hat{\epsilon}}(x^{\lambda})$, we have $x'=\mathcal{N}^{-1}_{\upsilon / \tau}(x)\in U_{\epsilon}(\breve{x}^{\lambda})$, and obtain
\begin{align*}
F_{\tau,\lambda}(x) & = ||x||_{\ell_{\tau}}^{\tau} + \frac{1}{2 \lambda}||\Phi x-y||_{\ell_2}^{2} 
 = ||\mathcal{N}_{\upsilon / \tau}(x')||_{\ell_{\tau}}^{\tau} + \frac{1}{2 \lambda}||\Phi \mathcal{N}_{\upsilon / \tau}(x')-y||_{\ell_2}^{2} \\
& = ||x'||_{\ell_{\upsilon}}^{\upsilon} + \frac{1}{2 \lambda}||\Phi \mathcal{N}_{\upsilon / \tau}(x')-y||_{\ell_2}^{2} = \breve{F}_{\upsilon,\lambda}(x') \\
&\geq \breve{F}_{\upsilon,\lambda}(\breve{x}^{\lambda}) = ||\breve{x}^{\lambda}||_{\ell_{\upsilon}}^{\upsilon} + \frac{1}{2 \lambda}||\Phi \mathcal{N}_{\upsilon / \tau}(\breve{x}^{\lambda})-y||_{\ell_2}^{2} \\
& = ||x^{\lambda}||_{\ell_{\tau}}^{\tau} + \frac{1}{2 \lambda}||\Phi x^{\lambda}-y||_{\ell_2}^{2}  = F_{\tau,\lambda}(x^{\lambda}) .
\end{align*}
\end{remark}

\HR{For the proof of Theorem~\ref{thm:conv}}, we proceed 
\HR{similarly to} Section~\ref{sec:IRLSltau}, by \HR{first} presenting a sequence of auxiliary lemmas on properties of the functional $J_{\tau,\lambda}$ and the dynamics of Algorithm CG-IRLS-$\lambda$. 

\subsection{Properties of the functional $J_{\tau, \lambda}$}
\begin{lemma} 
\label{lemma:ineqfunctional}
For the functional $J_{\tau, \lambda}$ 
\HR{defined} in \eqref{eq:deffunctional}, and the iterates $\tilde{x}^{n}$, $w^{n}$, and $\varepsilon^n$ produced by Algorithm CG-IRLS-$\lambda$, the following \HR{inequalities hold true}:
\begin{eqnarray}
\label{eq:chain1} J_{\tau, \lambda}(\tilde{x}^{n+1},w^{n+1},\varepsilon^{n+1})& \leq & J_{\tau, \lambda}(\tilde{x}^{n+1},w^{n},\varepsilon^{n+1}) \\
\label{eq:chain2} & \leq & J_{\tau, \lambda}(\tilde{x}^{n+1},w^{n},\varepsilon^{n}) \\
\label{eq:chain3} & \leq & J_{\tau, \lambda}(\tilde{x}^{n},w^{n},\varepsilon^{n}).
\end{eqnarray}
\end{lemma}
\begin{proof}The first inequality holds because $w^{n+1}$ is the minimizer and the second inequality holds since $\varepsilon^{n+1} \leq \varepsilon^{n}$. In the third inequality we use the fact that the CG-method is a descent method, decreasing the functional in each iteration. Since we take $\tilde{x}^{n}$ as the initial estimate in the first iteration of CG, the output $\tilde{x}^{n+1}$ of CG must have a value of the functional that is less or equal to the one of the initial estimate.
\end{proof}

The iterative application of Lemma~\ref{lemma:ineqfunctional} leads to the fact that for each $n\in \N^+$ the functional $J_{\tau, \lambda}$ is bounded:
\begin{equation}
\label{eq:boundedness}
0 \leq J_{\tau, \lambda}(\tilde{x}^{n},w^{n},\varepsilon^{n})\leq J_{\tau, \lambda}(\tilde{x}^{1},w^{0},\varepsilon^{0})= \bar{J}.
\end{equation}
Since the functional is composed of positive summands, \HR{its definition and \eqref{eq:boundedness} imply} 
\begin{align}
 \nonumber||\Phi\tilde{x}^n -y ||_{\ell_2} &\leq \sqrt{2\lambda\bar{J}}, \\
 \label{eq:restriction} ||\tilde{x}^{n}||_{\ell_2(w^n)}=\sqrt{\sum\limits_{j=1}^N \left(\tilde{x}_j^n\right)^2 w_j^n} &\leq \sqrt{\frac{2\bar{J}}{\tau}}, \quad \text{ and} \\
\nonumber w_j^n& \geq \left(\frac{2-\tau}{\tau\bar{J}}\right)^{\frac{2-\tau}{\tau}}, \quad j = 1,\ldots,N.
\end{align}
\HR{The} last inequality leads to a general relationship between the $\ell_2$-norm and $\ell_2(w^n)$-norm for arbitrary $x\in\R^N$:
\begin{equation}
\label{eq:normw2}
||x||_{\ell_2(w^n)} \geq \sqrt{\left(\frac{2-\tau}{\tau\bar{J}}\right)^{\frac{2-\tau}{\tau}}}||x||_{\ell_2}.
\end{equation}

\HR{In order to} show convergence to a critical point or minimizer of the functional $F_{\tau,\lambda}$, we will use \HR{the first order condition}~\eqref{eq:lspre}. Since this property is only valid for the exact solution $\hat{x}^{n+1}$, we need a connection between $\hat{x}^{n+1}$ and $\tilde{x}^{n+1}$. Observe that
\begin{equation}
\label{eq:funcexact}
J_{\tau, \lambda}(\hat{x}^{n+1},w^{n},\varepsilon^{n})\leq J_{\tau, \lambda}(\tilde{x}^{n+1},w^{n},\varepsilon^{n}) 
\end{equation}
since $\hat{x}^{n+1}$ is the exact minimizer. From \eqref{eq:funcexact} we obtain
\begin{equation*}
 \frac{\tau}{2}\sum\limits_{j=1}^{N}\left(\hat{x}_j^{n+1}\right)^2 w_j^n + \frac{1}{2\lambda} ||\Phi\hat{x}^{n+1}-y||_{\ell_2}^2  \leq \frac{\tau}{2}\sum\limits_{j=1}^{N}\left(\tilde{x}_j^{n+1}\right)^2 w_j^n + \frac{1}{2\lambda} ||\Phi\tilde{x}^{n+1}-y||_{\ell_2}^2
\end{equation*}
which leads to
\begin{equation}
\label{eq:wnormexactapprox}
\frac{\tau}{2}||\hat{x}^{n+1}||_{\ell_2(w^n)}^2 \leq \frac{\tau}{2}||\tilde{x}^{n+1}||_{\ell_2(w^n)}^2 + \frac{1}{2\lambda} \left(||\Phi\tilde{x}^{n+1}-y||_{\ell_2}^2 -  ||\Phi\hat{x}^{n+1}-y||_{\ell_2}^2\right).
\end{equation}
Since~\eqref{eq:funcexact} holds in addition to~\eqref{eq:chain3} and~\eqref{eq:boundedness}, we conclude, also for the exact solution $\hat{x}^{n+1}$, the bound
\begin{equation}
\label{eq:restrictionexact}
||\Phi\hat{x}^n -y ||_{\ell_2}\leq \sqrt{2\lambda J_{\tau,\lambda}(\hat{x}^n, w^{n-1},\varepsilon^{n-1})} \leq \sqrt{2\lambda\bar{J}},
\end{equation}
for all $n\in\HR{\N}$, and
\begin{equation}
\label{eq:roughbound}
||\hat{x}^{n+1}||_{\ell_2(w^n)} \leq \sqrt{\frac{2 J_{\tau,\lambda}(\hat{x}^{n+1}, w^{n},\varepsilon^{n})}{\tau}} \leq \sqrt{\frac{2\bar{J}}{\tau}}.
\end{equation}
\HR{Additionally using}~\eqref{eq:restrictionexact}, we are able \HR{to 
estimate} the second summand of~\eqref{eq:wnormexactapprox} by
\begin{align}
\label{eq:est2ndsum}
\begin{aligned}
 &\left(||\Phi\tilde{x}^{n+1}-y||_{\ell_2}^2 \right. -  \left.||\Phi\hat{x}^{n+1}-y||_{\ell_2}^2\right)  \leq \left|  \left(||\Phi\tilde{x}^{n+1}-y||_{\ell_2}^2 -  ||\Phi\hat{x}^{n+1}-y||_{\ell_2}^2\right) \right| \\
 =&  \left| ||\Phi\tilde{x}^{n+1}-\Phi\hat{x}^{n+1}||_{\ell_2}^2 + 2 \left\langle \Phi\tilde{x}^{n+1}-\Phi\hat{x}^{n+1},\Phi\hat{x}^{n+1} - y \right\rangle_{\ell_2}\right| \\
  \leq & ||\Phi\tilde{x}^{n+1}-\Phi\hat{x}^{n+1}||_{\ell_2} \left( ||\Phi\tilde{x}^{n+1}-\Phi\hat{x}^{n+1}||_{\ell_2} + 2 ||\Phi\hat{x}^{n+1} - y ||_{\ell_2}\right) \\
  \leq & ||\Phi\tilde{x}^{n+1}-\Phi\hat{x}^{n+1}||_{\ell_2} \left( ||\Phi\tilde{x}^{n+1}- y||_{\ell_2} + 3 ||\Phi\hat{x}^{n+1} - y ||_{\ell_2}\right)  \leq 4\sqrt{2\lambda\bar{J}} ||\Phi||\, ||\tilde{x}^{n+1}-\hat{x}^{n+1}||_{\ell_2},
 \end{aligned}
\end{align}
where we used the Cauchy-Schwarz inequality in the second inequality, the triangle inequality in the third inequality, and the bounds in \eqref{eq:restriction} and \eqref{eq:restrictionexact} in the last inequality. 

The following pivotal result of this section allows us to control the difference between the exact and approximate solution of the linear system in line 3 of Algorithm CG-IRLS-$\lambda$.
\begin{lemma}
\label{lemma:propertiesfunctional}
For a given positive number $a_{n+1}$ and a choice of the accuracy $\tol_{n+1}$ \HR{satisfying}~\eqref{eq:tol}, the functional $J_{\tau, \lambda}$ fulfills the two monotonicity properties
\begin{equation}
\label{eq:an1}
J_{\tau, \lambda}(\hat{x}^{n+1},w^{n+1},\varepsilon^{n+1}) - J_{\tau, \lambda}(\tilde{x}^{n+1},w^{n+1},\varepsilon^{n+1}) \leq a_{n+1}
\end{equation}
and
\begin{equation}
\label{eq:an2}
J_{\tau, \lambda}(\tilde{x}^{n+1},w^{n},\varepsilon^{n}) - J_{\tau, \lambda}(\hat{x}^{n+1},w^{n},\varepsilon^{n}) \leq a_{n+1}.
\end{equation}
\end{lemma}
\begin{proof}
By means of the relation
\begin{equation*}
w_j^{n+1} = w_j^n \frac{w_j^{n+1}}{w_j^n} \leq w_j^n	\left(\frac{(\tilde{x}_j^n)^2 +(\varepsilon^n)^2}{(\tilde{x}_j^{n+1})^2 +(\varepsilon^{n+1})^2}\right)^{1-\frac{\tau}{2}}\leq w_j^n \left(\frac{\max\limits_{j}(\tilde{x}_j^n)^2 +(\varepsilon^n)^2}{(\varepsilon^{n+1})^2}\right)^{1-\frac{\tau}{2}} \HR{= w_j^n C_{w^n}},
\end{equation*}
\HR{where $C_{w^n}$ was defined in \eqref{eq:tolcwn}}, we can estimate
\begin{align}
\nonumber &J_{\tau, \lambda}(\hat{x}^{n+1},w^{n+1},\varepsilon^{n+1}) - J_{\tau, \lambda}(\tilde{x}^{n+1},w^{n+1},\varepsilon^{n+1}) \\
\nonumber & \leq \frac{\tau}{2} \sum\limits_{j=1}^{N} \left(\hat{x}_j^{n+1} - \tilde{x}_j^{n+1}\right)\left(\hat{x}_j^{n+1} + \tilde{x}_j^{n+1}\right)w_j^{n+1} + \left|\frac{1}{2\lambda}||\Phi\hat{x}^{n+1} - y||_{\ell_2}^2 - ||\Phi\tilde{x}^{n+1} - y||_{\ell_2}^2\right|\\
\nonumber & \leq \frac{\tau}{2} \left|\left\langle\hat{x}^{n+1} - \tilde{x}^{n+1},\hat{x}^{n+1} + \tilde{x}^{n+1}\right\rangle_{\ell_2(w^{n+1})}\right| + \frac{4\sqrt{2\lambda\bar{J}}}{2\lambda} ||\Phi||||\tilde{x}^{n+1}-\hat{x}^{n+1}||_{\ell_2}\\
\nonumber & \leq \frac{\tau}{2} \sqrt{\sum\limits_{j=1}^N(\hat{x}_j^{n+1} - \tilde{x}_j^{n+1})^2w_j^{n+1}} \sqrt{\sum\limits_{j=1}^N(\hat{x}_j^{n+1} + \tilde{x}_j^{n+1})^2w_j^{n+1}} + \frac{4\sqrt{2\lambda\bar{J}}}{2\lambda} ||\Phi|| ||\tilde{x}^{n+1}-\hat{x}^{n+1}||_{\ell_2} \\
\nonumber & \HR{ \leq } \frac{\tau}{2} C_{w^n}||\hat{x}^{n+1} - \tilde{x}^{n+1}||_{\ell_2(w^{n})}||\hat{x}^{n+1} + \tilde{x}^{n+1}||_{\ell_2(w^{n})} + \frac{4\sqrt{2\lambda\bar{J}}}{2\lambda} ||\Phi|| ||\tilde{x}^{n+1}-\hat{x}^{n+1}||_{\ell_2} \\
\nonumber & \leq  C_{w^n}||\hat{x}^{n+1} - \tilde{x}^{n+1}||_{\ell_2(w^{n})}2\max\left\{\frac{\tau}{2}||\hat{x}^{n+1}||_{\ell_2(w^{n})}, \frac{\tau}{2}||\tilde{x}^{n+1}||_{\ell_2(w^{n})}\right\} + \frac{4\sqrt{2\lambda\bar{J}}}{2\lambda} ||\Phi|| ||\tilde{x}^{n+1}-\hat{x}^{n+1}||_{\ell_2} \\
\nonumber & \leq \sqrt{2\bar{J}\tau} C_{w^n}||\hat{x}^{n+1} - \tilde{x}^{n+1}||_{\ell_2(w^{n})}+ \frac{4\sqrt{2\lambda\bar{J}}}{2\lambda} \sqrt{\left(\frac{2-\tau}{\tau\bar{J}}\right)^{-\frac{2-\tau}{\tau}}} ||\Phi|| ||\tilde{x}^{n+1}-\hat{x}^{n+1}||_{\ell_2(w^n)} \\
\nonumber & \leq  \left(\sqrt{2\bar{J}\tau} C_{w^n}+ \frac{4\sqrt{2\lambda\bar{J}}}{2\lambda} \sqrt{\left(\frac{2-\tau}{\tau\bar{J}}\right)^{-\frac{2-\tau}{\tau}}} ||\Phi|| \right)||\tilde{x}^{n+1}-\hat{x}^{n+1}||_{\ell_2(w^n)} \leq a_{n+1},
\end{align}

where we used \eqref{eq:est2ndsum} in the second inequality, Cauchy-Schwarz in the third inequality, and~\eqref{eq:normw2},~\eqref{eq:restriction}, and~\eqref{eq:roughbound} in the \HR{sixth} inequality. Thus we obtain~\eqref{eq:an1}. To show~\eqref{eq:an2}, we use \eqref{eq:normw2} in the \HR{second to last} inequality, \HR{condition~\eqref{eq:tol} in the last inequality} and the fact that $\hat{x}^{n+1}=\underset{x}{\argmin} J_{\tau,\lambda}(x, w^n, \varepsilon^n)$ (and thus fulfilling~\eqref{eq:lspre}) in the second identity below:
\begingroup
\begin{align}
\label{eq:xtilde}
 &J_{\tau, \lambda}(\tilde{x}^{n+1},w^{n},\varepsilon^{n}) - J_{\tau, \lambda}(\hat{x}^{n+1},w^{n},\varepsilon^{n}) \\
&= \frac{\tau}{2} \sum\limits_{j=1}^{N} \left((\tilde{x}_j^{n+1})^2 - (\hat{x}_j^{n+1})^2\right)w_j^n + \frac{1}{2\lambda}\left( ||\Phi\tilde{x}^{n+1} -\Phi\hat{x}^{n+1}||_{\ell_2}^2+ 2 \left\langle \Phi(\tilde{x}^{n+1}-\hat{x}^{n+1}),\Phi\hat{x}^{n+1} - y \right\rangle_{\ell_2}\right)\\
  & = \frac{\tau}{2} \sum\limits_{j=1}^{N} \left((\tilde{x}_j^{n+1})^2 - (\hat{x}_j^{n+1})^2 - 2 \hat{x}_j^{n+1}\tilde{x}_j^{n+1} + 2 \left(\hat{x}_j^{n+1}\right)^2\right)w_j^n + \frac{1}{2\lambda} ||\Phi\tilde{x}^{n+1} -\Phi\hat{x}^{n+1}||_{\ell_2}^2 \\
  & \leq \frac{\tau}{2} \sum\limits_{j=1}^{N} \left((\tilde{x}_j^{n+1})^2 + (\hat{x}_j^{n+1})^2 - 2 \hat{x}_j^{n+1}\tilde{x}_j^{n+1} \right)w_j^n + \frac{1}{2\lambda} ||\Phi|| ||\tilde{x}^{n+1} -\hat{x}^{n+1}||_{\ell_2}^2 \\
 & \leq \left(\frac{\tau}{2} +\frac{||\Phi||^2}{2\lambda}\left(\frac{2-\tau}{\tau\bar{J}}\right)^{-\frac{2-\tau}{\tau}}\right)||\tilde{x}^{n+1} - \hat{x}^{n+1}||_{\ell_2(w^n)}^2 \leq a_{n+1}.
\end{align}
\endgroup
\end{proof}

\HR{Besides} Lemma~\ref{lemma:propertiesfunctional} there are two more helpful properties of the functional. First, the identity
\begin{equation*}
J_{\tau, \lambda}(\hat{x}^{n},w^{n},\varepsilon^{n}) - J_{\tau, \lambda}(\hat{x}^{n+1},w^{n},\varepsilon^{n})  = \frac{\tau}{2} ||\hat{x}^n-\hat{x}^{n+1}||_{\ell_2(w^n)}^2 + \frac{1}{2\lambda} ||\Phi\hat{x}^{n} -\Phi\hat{x}^{n+1}||_{\ell_2}^2
\end{equation*}
can be shown by the same calculation as in \eqref{eq:xtilde}, by means of replacing $\tilde{x}^{n+1}$ by $\hat{x}^n$.
Second, it follows in particular that 
\begin{eqnarray}
\nonumber\frac{\tau}{2}\sqrt{\left(\frac{2-\tau}{\tau\bar{J}}\right)^{\frac{2-\tau}{\tau}}}||\hat{x}^{n+1}-\hat{x}^n||_{\ell_2}^2 &\leq& \frac{\tau}{2} ||\hat{x}^{n+1}-\hat{x}^n||_{\ell_2(w^n)}^2  \\
\label{eq:estxhat} &\leq& J_{\tau, \lambda}(\hat{x}^{n},w^{n},\varepsilon^{n}) - J_{\tau, \lambda}(\hat{x}^{n+1},w^{n},\varepsilon^{n}).
\end{eqnarray} 
where the estimate \eqref{eq:normw2} is used in the first inequality.\\

\subsection{Proof of convergence}
We need to show that the difference \HR{$\hat{x}^{n+1} - \hat{x}^n$} between two successive {\it exact} iterates and the one between the exact and approximated iterates,
\HR{$\hat{x}^n - \tilde{x}^n$, become} arbitrarily small. This result is used in the proof of Theorem \ref{thm:conv} to show that \HR{both $(\hat{x}^n)_{n\in\N}$ and $(\tilde{x}^n)_{n\in\N}$} 
converge to the same limit.
\begin{lemma}
\label{lemma:asymptotic}
Consider a summable sequence $(a_{n})_{n\in\N}$ and choose the accuracy of the CG solution $\tol_{n}$ 
\HR{satisfying}~\eqref{eq:tol} for the $n$-th iteration step. Then the sequences $(\hat{x}^n)_{n\in\N}$ and $(\tilde{x}^n)_{n\in\N}$ have the properties
\begin{equation}
\label{eq:asymptoticregularity}
\lim\limits_{n\rightarrow\infty}||\hat{x}^n - \hat{x}^{n+1}||_{\ell_2} = 0
\end{equation}
and
\begin{equation}
\label{eq:asymptotic}
\lim\limits_{n\rightarrow\infty}||\tilde{x}^{n+1} - \hat{x}^{n+1}||_{\ell_2} = 0.
\end{equation}
\end{lemma}
\begin{proof}
We use the properties of $J$, which we derived in the previous subsection. First, we show \eqref{eq:asymptoticregularity}:
\begingroup
\allowdisplaybreaks
\begin{align*}
\frac{\tau}{2}\sqrt{\left(\frac{2-\tau}{\tau\bar{J}}\right)^{\frac{2-\tau}{\tau}}}\sum\limits_{n = 1}^{M}||\hat{x}^{n+1}-\hat{x}^n||_{\ell_2}^2 &\leq \sum\limits_{n = 1}^{M}J_{\tau, \lambda}(\hat{x}^{n},w^{n},\varepsilon^{n}) - J_{\tau, \lambda}(\hat{x}^{n+1},w^{n},\varepsilon^{n}) \\
&\leq \sum\limits_{n = 1}^{M}J_{\tau, \lambda}(\hat{x}^{n},w^{n},\varepsilon^{n}) - J_{\tau, \lambda}(\tilde{x}^{n+1},w^{n},\varepsilon^{n}) + a_{n+1} \\
&\leq \sum\limits_{n = 1}^{M}J_{\tau, \lambda}(\hat{x}^{n},w^{n},\varepsilon^{n}) - J_{\tau, \lambda}(\tilde{x}^{n+1},w^{n+1},\varepsilon^{n+1}) + a_{n+1} \\
&\leq \sum\limits_{n = 1}^{M}J_{\tau, \lambda}(\hat{x}^{n},w^{n},\varepsilon^{n}) - J_{\tau, \lambda}(\hat{x}^{n+1},w^{n+1},\varepsilon^{n+1}) + 2a_{n+1} \\
&= J_{\tau, \lambda}(\hat{x}^{1},w^{1},\varepsilon^{1}) - J_{\tau, \lambda}(\tilde{x}^{M+1},w^{M+1},\varepsilon^{M+1}) + 2 \sum\limits_{n = 1}^{M}a_{n+1} \\
&\leq \bar{J} + 2\sum\limits_{n = 1}^{M}a_{n+1}.
\end{align*}
\endgroup
We used~\eqref{eq:estxhat} in the first inequality, ~\eqref{eq:an2} in the second inequality,~\eqref{eq:chain1} and~\eqref{eq:chain2} in the third inequality,~\eqref{eq:an1} in the fourth inequality and a telescoping sum in the identity. Letting $M\rightarrow\infty$ we obtain
\begin{equation*}
\frac{\tau}{2}\left(\frac{2-\tau}{\tau\bar{J}}\right)^{\frac{2-\tau}{\tau}}\sum\limits_{n = 1}^{\infty}||\hat{x}^{n+1}-\hat{x}^n||_{\ell_2}^2 \leq \bar{J} + 2\sum\limits_{n = 1}^{\infty}a_{n+1} < \infty
\end{equation*}
and thus \eqref{eq:asymptoticregularity}.

Second, we show \eqref{eq:asymptotic}. From line 1 and 3 of~\eqref{eq:xtilde} we know that
\begin{align*}
J_{\tau, \lambda}(&\tilde{x}^{n+1},w^{n},\varepsilon^{n}) - J_{\tau, \lambda}(\hat{x}^{n+1},w^{n},\varepsilon^{n}) \\
&= \frac{\tau}{2} \sum\limits_{j=1}^{N} \left((\tilde{x}_j^{n+1})^2 - (\hat{x}_j^{n+1})^2 - 2 \hat{x}_j^{n+1}\tilde{x}_j^{n+1} + 2 \left(\hat{x}_j^{n+1}\right)^2\right)w_j^n + \frac{1}{2\lambda} ||\Phi\tilde{x}^{n+1} -\Phi\hat{x}^{n+1}||_{\ell_2}^2 \\
&= \frac{\tau}{2} ||\tilde{x}_j^{n+1} - \hat{x}_j^{n+1}||_{\ell_2(w^n)}^2 +  \frac{1}{2\lambda} ||\Phi\tilde{x}^{n+1} -\Phi\hat{x}^{n+1}||_{\ell_2}^2.
\end{align*}
Since the second summand is positive, we conclude
\begin{align*}
J_{\tau, \lambda}(\tilde{x}^{n+1},w^{n},\varepsilon^{n}) - J_{\tau, \lambda}(\hat{x}^{n+1},w^{n},\varepsilon^{n}) \geq \frac{\tau}{2} ||\tilde{x}_j^{n+1} - \hat{x}_j^{n+1}||_{\ell_2(w^n)}^2.
\end{align*}
Together with~\eqref{eq:an2} we find that
\begin{align}
\nonumber \frac{\tau}{2}\left(\frac{2-\tau}{\tau\bar{J}}\right)^{\frac{2-\tau}{\tau}}||\tilde{x}^{n+1} - \hat{x}^{n+1}||_{\ell_2}^2 &\leq \frac{\tau}{2}||\tilde{x}^{n+1} - \hat{x}^{n+1}||_{\ell_2(w^n)}^2 \\
\nonumber &\leq J_{\tau, \lambda}(\tilde{x}^{n+1},w^{n},\varepsilon^{n}) - J_{\tau, \lambda}(\hat{x}^{n+1},w^{n},\varepsilon^{n}) \leq a_{n+1},
\end{align}
and thus taking limits on both sides we get
\[\frac{\tau}{2}\left(\frac{2-\tau}{\tau\bar{J}}\right)^{\frac{2-\tau}{\tau}}\lim \sup_{n\rightarrow\infty}||\tilde{x}^{n+1} - \hat{x}^{n+1}||_{\ell_2}^2 \leq \lim\limits_{n\rightarrow\infty} a_{n+1} = 0, \]
which implies \eqref{eq:asymptotic}.
\end{proof}
\begin{remark}
\label{rem:asymptotic}
By means of Lemma \ref{lemma:asymptotic} we obtain
\begin{equation}
\label{eq:diff}
\lim\limits_{n\rightarrow\infty}||\tilde{x}^{n} - \tilde{x}^{n+1}||_{\ell_2} \leq \lim\limits_{n\rightarrow\infty}||\tilde{x}^{n} - \hat{x}^{n}||_{\ell_2} + \lim\limits_{n\rightarrow\infty}||\hat{x}^{n} - \hat{x}^{n+1}||_{\ell_2} + \lim\limits_{n\rightarrow\infty}||\hat{x}^{n+1} - \tilde{x}^{n+1}||_{\ell_2} = 0.
\end{equation}
\end{remark}

The following lemma provides a lower bound for the $\varepsilon^n$, which is used to show a contradiction in the proof of Theorem~\ref{thm:conv}.
\HR{Recall that $\phi \in \big(0,\frac{1}{4- \tau}\big)$ is the parameter appearing in the update rule for $\varepsilon$ in step \ref{eq:step4} of both the algorithms CG-IRLS-$\lambda$ and IRLS-$\lambda$.}
\begin{lemma}[{\cite[Lemma 4.5.4, Lemma 4.5.6]{sergei}}]
\label{lem:lowerboundeps}
Let $\tau =1$ and thus $w_j^{n}=\left((\tilde{x}^{n}_j)^2+(\varepsilon^{n})^2\right)^{-\frac{1}{2}}$, $j \in\{1,\ldots,N\}$. There exists a strictly increasing subsequence $(n_l)_{l\in\N}$ and some constant $C>0$ such that 
$$(\varepsilon^{n_l+1})^2 \geq C((w_j^{n_l})^{-1})^{2\tau\phi}|(w_j^{n_l-1})^{-1}-(w_j^{n_l})^{-1}|^{4\phi}.$$
\end{lemma}
\begin{proof}
Since $J_{\tau,\lambda}(\tilde{x}^n,w^n,\varepsilon^n)$ is decreasing with $n$ due to Lemma~\ref{lemma:ineqfunctional} and bounded below by $0$, the difference $|J_{\tau,\lambda}(\tilde{x}^{n-1},w^{n-1},\varepsilon^{n-1}) - J_{\tau,\lambda}(\tilde{x}^{n},w^{n},\varepsilon^{n})|$ is converging to $0$ for $n\rightarrow\infty$. In addition $\alpha^{n+1}\rightarrow 0$ for $n\rightarrow\infty$, and thus by definition also $\varepsilon^n\rightarrow 0$. Consequently there exists a subsequence $(n_l)_{l\in\N}$ such that 
\begin{equation}
\label{eq:lowerboundepsequality}
\varepsilon^{n_l+1} = |J_{\tau,\lambda}(\tilde{x}^{n_l-1},w^{n_l-1},\varepsilon^{n_l-1}) - J_{\tau,\lambda}(\tilde{x}^{n_l},w^{n_l},\varepsilon^{n_l})|^{\phi}+\alpha^{n_l+1}.
\end{equation}
Following exactly the steps of the proof of~\cite[Lemma 4.5.6.]{sergei} yields the assertion. Observe that all of these steps are also valid for $0<\tau < 1$, although in~\cite[Lemma 4.5.6]{sergei} the author restricted it to the case $\tau \geq 1$.
\end{proof}
\HR{
\begin{remark}\label{rem:eps-zero} The observation in the previous proof that $(\varepsilon^n)$ converges to $0$ will be again important below. 
\end{remark}
}
We \HR{are now prepared for the proof of Theorem~\eqref{thm:conv}}. 

\begin{proof}[Proof of Theorem~\ref{thm:conv}]
Consider the subsequence $(\tilde{x}^{n_l})_{l\in\N}$ of Lemma~\ref{lem:lowerboundeps}. Since $\|\tilde{x}^{n_l}\|_{\ell_2}$ is bounded 
\HR{by~\eqref{eq:restriction}}, there exists a converging subsequence $(\tilde{x}^{n_k})_{k\in\N}$, which has limit $x^{\lambda}$.

Consider the case $\tau = 1$ \HR{and $x^\lambda \neq 0$}. We first show that 
\begin{equation}
\label{eq:limitwx}
-\infty < \lim\limits_{n\rightarrow\infty}\tilde{x}^{n_k+1}_j w_j^{n_k} = \lim\limits_{n\rightarrow\infty}\hat{x}^{n_k+1}_j w_j^{n_k} <\infty\text{, for all }j=1,\ldots,N.
\end{equation}
\HR{It follows from equation~\eqref{eq:lspre} and the boundedness of the residual~\eqref{eq:restrictionexact} that the sequence $(\hat{x}^{n_k+1} w_j^{n_k})_{n_k}$ is bounded, i.e.,
\[
\left\| \left[\hat{x}^{n_k+1}_j w_j^{n_k} \right]_j \right\|_2 = \frac{1}{\lambda} \|\Phi ^*(\Phi \hat{x}^{n_k+1}-y) \| \leq C. 
\]
Therefore, there exists a converging subsequence, for simplicity again denoted by $(\hat{x}^{n_k+1} w_j^{n_k})_{n_k}$.}
To show the identity in~\eqref{eq:limitwx}, we estimate
\begin{align*}
|\hat{x}^{n_k+1}_j w_j^{n_k}-\tilde{x}^{n_k+1}_j w_j^{n_k}|&\leq \frac{\tol_{n_k+1}}{\sqrt{(\tilde{x}_j^{n_k})^2+(\varepsilon^{n_k})^2}}
 \leq\frac{a_{n_k+1}}{\sqrt{2\bar{J}}C_{w^{n_k}}\sqrt{(\tilde{x}_j^{n_k})^2+(\varepsilon^{n_k})^2}}\\
& =\frac{a_{n_k+1}\varepsilon^{n_k+1}}{\sqrt{2\bar{J}}\sqrt{\max\limits_\HR{\ell}(\tilde{x}_\HR{\ell}^{n_k})^2+(\varepsilon^{n_k})^2}\sqrt{(\tilde{x}_j^{n_k})^2+(\varepsilon^{n_k})^2}} \leq\frac{a_{n_k+1}\varepsilon^{n_k+1}}{\sqrt{2\bar{J}}(\max\limits_\HR{\ell}|\tilde{x}_\HR{\ell}^{n_k}|)(\varepsilon^{n_k})}\\
& \leq\frac{a_{n_k+1}}{\sqrt{2\bar{J}}(\max\limits_\HR{\ell}|\tilde{x}_\HR{\ell}^{n_k}|)},
\end{align*}
for all $j=1,\ldots,N$, where the second inequality follows by the upper bound of $\tol_{n}$ in~\eqref{eq:tol}, and the last inequality is due to the definition of $\varepsilon^{n+1}$ which yields $\frac{\varepsilon^{n+1}}{\varepsilon^n} \leq 1$. Since we assumed $\lim\limits_{k\rightarrow\infty}\tilde{x}^{n_k} = x^{\lambda}\neq 0$, there is a $k_0$ such that for all $k \geq k_0$, we have that $\max\limits_j|\tilde{x}_j^{n_k}| \geq \HR{c} > 0$. Since \HR{$(a_{n_k})$ tends} to $0$, we conclude that $\lim\limits_{n\rightarrow\infty}|\hat{x}^{n_k+1}_j w_j^{n_k}-\tilde{x}^{n_k+1}_j w_j^{n_k}| = 0$, and therefore we have~\eqref{eq:limitwx}. Note that we will use the notation $k_0$ several times in the presentation of this proof, but for different arguments. We do not mention it explicitly, but we assume a newly defined $k_0$ to be always larger or equal \HR{to} the previously defined one.

\HR{Next we show that $x^\lambda$ is a minimizer of $F_{1,\lambda}$ by verifying} conditions~\eqref{eq:lassominnonzero} and~\eqref{eq:lassominzero}. \HR{To this end} we notice that by Lemma~\ref{lemma:asymptotic} and Remark~\ref{rem:asymptotic} it follows that $\lim\limits_{k\rightarrow\infty}\hat{x}^{n_k}_j = \lim\limits_{k\rightarrow\infty}\tilde{x}^{n_k}_j = \lim\limits_{k\rightarrow\infty}\tilde{x}^{n_k-1}_j  =x^{\lambda}_j$. By means of this result, in the case of $x^{\lambda}_j\neq 0$, we have, due to continuity arguments, \eqref{eq:lspre} \HR{and Remark~\ref{rem:eps-zero}},
\begin{align*}
-(\Phi^* (y - \Phi x^{\lambda}))_j &= \lim\limits_{k\rightarrow\infty}-(\Phi^*(y - \Phi \hat{x}^{n_k}))_j = \lim\limits_{k\rightarrow\infty}\lambda \hat{x}^{n_k}_jw^{n_k-1}_j = \lambda\lim\limits_{k\rightarrow\infty} \hat{x}^{n_k}_j((\tilde{x}^{n_k-1}_j)^2+(\varepsilon^{n_k-1})^2)^{-\frac{1}{2}} \\
&=\lambda x^{\lambda}_j((x^{\lambda}_j)^2+(0)^2)^{-\frac{1}{2}} = \lambda\sign(x^\lambda_j),
\end{align*}
and thus \eqref{eq:lassominnonzero}.

\HR{In order to show} condition~\eqref{eq:lassominzero} for \HR{$j$} 
such that $x^{\lambda}_j=0$, we follow the main idea in the proof of Lemma 4.5.9.~in~\cite{sergei}. Assume
\begin{equation}
\label{eq:wrongassumption}
\lim\limits_{k\rightarrow\infty} \hat{x}^{n_k}_jw^{n_k-1}_j > 1.
\end{equation} 
Then there exists an $\epsilon > 0 $ and a $k_0\in\N$, such that for all $k\geq k_0$ the inequality $(\hat{x}^{n_k}_jw^{n_k-1}_j)^2 > 1+\epsilon$ holds. \HR{Due to} \eqref{eq:limitwx}, we can furthermore choose $k_0$ large enough such that also $(\tilde{x}^{n_k}_j w^{n_k-1}_j)^2 > 1+\epsilon$ for all $k\geq k_0$. 
\HR{Recalling the identity for} $w_j^{n}$ from Lemma~\ref{lem:lowerboundeps}, we obtain 
\begin{align}
\label{eq:estimxj} (\tilde{x}^{n_k}_j)^2 & > (1+\epsilon)((w_j^{n_k-1})^{-1})^2 \\
\nonumber &= (1+\epsilon)((\tilde{x}_j^{n_k-1})^2 + (\varepsilon^{n_k-1})^2) \geq  (1+\epsilon)(\varepsilon^{n_k+1})^2 \\
\nonumber &\geq  (1+\epsilon) C |(w_j^{n_k})^{-1}|^{2\phi}|(w_j^{n_k-1})^{-1} - (w_j^{n_k})^{-1}|^{4\phi}  \geq (1+\epsilon) C |\tilde x_j^{n_k}|^{2\phi}|(w_j^{n_k-1})^{-1} - (w_j^{n_k})^{-1}|^{4\phi},
\end{align}
where the second inequality follows by the definition of the $\varepsilon^n$, and the third inequality follows from Lemma~\ref{lem:lowerboundeps}. Furthermore, in the last inequality we used that $w_j^{n}\leq |\tilde{x}_j^n|^{-1}$ which follows directly from the definition of $w_j^n$. By means of this estimate, we conclude
\begin{equation}
\label{eq:estimsj}
(w_j^{n_k-1})^{-1}\geq (w_j^{n_k})^{-1} - |(w_j^{n_k-1})^{-1}-(w_j^{n_k})^{-1}| > |\tilde x_j^{n_k}| - ((1+\epsilon)C)^{-\frac{1}{4\phi}}|\tilde{x}_j^{n_k}|^{\frac{2-2\phi}{4\phi}}.
\end{equation}
Since $0<\HR{\phi}<\frac{1}{3}$, the exponent $\frac{2-2\phi}{4\phi} > 1$. In combination with the fact that $\tilde x_j^{n_k}$ is vanishing for $k\to \infty$, we are able to choose $k_0$ large enough to have $((1+\epsilon)C)^{-\frac{1}{4\phi}}|\tilde x_j^{n_k}|^{\frac{2-2\phi}{4\phi}-1}<\bar{\epsilon}:=1-(1+\epsilon)^{-\frac{1}{2}}$ for all $k\geq k_0$ and therefore
\begin{equation}
\label{eq:estimsj2}
(w_j^{n_k-1})^{-1}\geq |\tilde x_j^{n_k}|(1-\bar{\epsilon}).
\end{equation}
The combination of~\eqref{eq:estimxj} and~\eqref{eq:estimsj2} yields
\begin{equation}\label{eq:lalala}
|\tilde x_j^{n_k}|^2 > (1+\epsilon) \left(w_j^{n_k-1}\right)^{-2} \geq (1+\epsilon)
|\tilde x_j^{n_k}|^2(1-\bar{\epsilon})^2.
\end{equation}
Since we have $|\tilde{x}_j^{n_k} w_j^{n_k-1}| > 1+\epsilon$ for all $k\geq k_0$, we also have $\tilde{x}_j^{n_k}\neq 0 $, and thus, we can divide in \eqref{eq:lalala} by $|\tilde{x}_j^{n_k}|$ and insert the definition of $\bar{\epsilon}$ to obtain
$$1 > (1+\epsilon) (1-\bar{\epsilon})^2 =1, $$
which is a contradiction, and thus the assumption~\eqref{eq:wrongassumption} is false. By means of this result and again a continuity argument, we show condition \eqref{eq:lassominzero} by
\begin{equation*}
(\Phi^T(y - \Phi x^{\lambda}))_j = \lim\limits_{k\rightarrow\infty}(\Phi^T(y - \Phi \hat{x}^{n_k}))_j = \lambda \lim\limits_{k\rightarrow\infty}\hat{x}^{n_k}_jw^{n_k-1}_j \leq \lambda.
\end{equation*}
At this point, we have shown that at least the convergent subsequence $(\tilde{x}^{n_k})_{{n_k}\in\N}$ is such that its limit $x^{\lambda}$ is a minimizer of $F_{1,\lambda}(x)$. To show that this is valid for any convergent subsequence of  $(\tilde{x}^{n})_{{n}\in\N}$, we remind that the subsequence $(\tilde{x}^{n_k})_{{n_k}\in\N}$ is the one of Lemma~\ref{lem:lowerboundeps}, and therefore fulfills~\eqref{eq:lowerboundepsequality}. Thus, we can adapt~\cite[Lemma 4.6.1]{sergei} to our case, following the arguments in the proof. These arguments only require the monotonicity of the functional $J_{\tau,\lambda}$, which we show in Lemma~\ref{lemma:ineqfunctional}. Consequently the limit $x^{\lambda}$ of any convergent subsequence of $(\tilde{x}^{n})_{{n}\in\N}$ is a minimizer of $F_{1,\lambda}(x)$.

Consider the case $0<\tau < 1$. 
\HR{The transformation $\mathcal{N}_{\zeta}(x)$ defined in~\eqref{eq:bijmap}} is continuous and bijective. Thus, $\breve{x}^{\lambda}:=\mathcal{N}_{\upsilon / \tau}^{-1}(x^{\lambda})$ is well-defined, and $x^{\lambda}_j = 0$ if and only if $\breve{x}^{\lambda}_j=0$. At a critical point of the {\it differentiable} functional $\breve{F}_{\tau,\lambda}$, its first derivative has to vanish \HR{which is equivalent to} 
the conditions
\begin{equation}
\label{eq:crittauless1}
\frac{\upsilon}{\tau}|x_j|^{\frac{\upsilon-\tau}{\tau}}\left(\Phi^*y-\Phi^*\Phi \mathcal{N}_{\upsilon / \tau}(x)\right)_j +\lambda\upsilon\sign(x_j)|x_j|^{\upsilon-1} = 0,\quad     j=1,\ldots,N.
\end{equation}
We 
\HR{show now} that 
$\breve{x}^{\lambda}$ fulfills this first order optimality condition. It is obvious that for all $j$ such that $\breve{x}_j^{\lambda} = 0$ the condition is trivially fulfilled. Thus, it remains to consider all $j$ where $\breve{x}_j^{\lambda}\neq 0$. As in the case of $\tau = 1$, we conclude by Lemma~\ref{lemma:asymptotic} and Remark~\ref{rem:asymptotic} that $\lim\limits_{k\rightarrow\infty}\hat{x}^{n_k}_j = \lim\limits_{k\rightarrow\infty}\tilde{x}^{n_k}_j = \lim\limits_{k\rightarrow\infty}\tilde{x}^{n_k-1}_j  =x^{\lambda}_j$. Therefore continuity arguments as well as~\eqref{eq:lspre} yield
\begin{align*}
-(\Phi^* (y - \Phi x^{\lambda}))_j &= \lim\limits_{k\rightarrow\infty}-(\Phi^*(y - \Phi \hat{x}^{n_k}))_j = \lim\limits_{k\rightarrow\infty}\lambda \HR{\tau} \hat{x}^{n_k}_jw^{n_k-1}_j = \lambda\tau\lim\limits_{k\rightarrow\infty} \hat{x}^{n_k}_j((\tilde{x}^{n_k-1}_j)^2+(\varepsilon^{n_k-1})^2)^{-\frac{2-\tau}{2}} \\
&=\lambda\tau x^{\lambda}_j((x^{\lambda}_j)^2+(0)^2)^{-\frac{2-\tau}{2}} = \lambda\tau\sign(x^\lambda_j)|x_j|^{\tau-1}.
\end{align*}
We replace $x^{\lambda}=\mathcal{N}_{\upsilon / \tau}(\breve{x}^{\lambda})$ and obtain
\begin{align*}
-(\Phi^* (y - \Phi \mathcal{N}_{\upsilon / \tau}(\breve{x}^{\lambda}))_j  &= \lambda\tau\sign((\mathcal{N}_{\upsilon / \tau}(\breve{x}^{\lambda}))_j)|(\mathcal{N}_{\upsilon / \tau}(\breve{x}^{\lambda}))_j|^{\tau-1}\\
& =\lambda\tau\sign(\breve{x}^{\lambda}_j)|\breve{x}^{\lambda}_j|^{\upsilon-\frac{\upsilon}{\tau}}.
\end{align*}
\HR{We multiply} this identity by $\frac{\upsilon}{\tau}|x_j|^{\frac{\upsilon-\tau}{\tau}}$ and obtain~\eqref{eq:crittauless1}. 

If $\breve{x}^{\lambda}$ is also a global minimizer of $\breve{F}_{\upsilon,\lambda}$, then $x^{\lambda}$ is a global minimizer of $F_{\tau,\lambda}$. This is due the equivalence of the two problems \HR{which was} shown in~\cite[Proposition 2.4]{RamlauZarzer12}  \HR{based} on the continuity and bijectivity of the mapping $\mathcal{N}_{\upsilon / \tau}$~\cite[Proposition 3.4]{Zarzer09}.
\end{proof}

\section{Numerical Results}
\label{sec:numerics}
We illustrate  the theoretical results of this paper by several numerical experiments. We first show that \HR{our modified versions of IRLS}
yield significant  improvements in terms
of computational time \HR{and often outperform} the state of the art methods Iterative Hard Thresholding (IHT)~\cite{Blumensath09} and Fast Iterative Soft-Thresholding Algorithm (FISTA)~\cite{beck09}. 

Before going into the detailed presentation of the numerical tests, we raise two plain numerical disclaimers concerning the numerical stability of CG-IRLS and CG-IRLS-$\lambda$:
\begin{itemize}
\item The first issue concerns IRLS methods in general: The case where $\varepsilon^{n}\rightarrow 0$, \HR{i.e.,} $x_j^{n}\rightarrow 0$, for some $j\in\{1,\ldots,N\}$ and $n\rightarrow\infty$, is very likely since our goal is the computation of sparse vectors. In this case \HR{$w^{n}_j$} will \HR{for} some $n$ become too large to be properly represented by a computer. Thus, in practice, we have to provide a lower bound for $\varepsilon$ by some $\varepsilon^{\textrm{min}}>0$. Imposing such a limit has the theoretical disadvantage that in general the algorithms are only calculating an approximation of the respective problems~\eqref{eq:ltauproblem} and~\eqref{eq:prob}. Therefore, to obtain a \enquote{sufficiently good} approximation, one has to choose $\varepsilon^{\textrm{min}}$ sufficiently small. This raises yet another numerical issue: If we choose, e.g., $\varepsilon^{\textrm{min}}=1\text{\sc{e}-}8$ and assume that also $x_j^{n}  \ll 1$, then $w^{n}_j$ is of the order $1\text{\sc{e}+}8$. Compared to the entries of the matrix $\Phi$, which are of the order $1$, any multiplication or addition by such a value will cause serious numerical errors. In this context  we cannot expect that the IRLS method reaches high accuracy, and saturation effects of the error are likely to occur before machine precision.
\item The second issue concerns 
the CG method: In Algorithm~\ref{CG} and Algorithm~\ref{MCG} we have to divide at some point by $\vectornorm{T^*p^i}^2_{\ell_2}$ or $\langle Ap^i,p^i\rangle_{\ell_2}$ respectively. As soon as the residual decreases, also $p^i$ decreases with the same order of magnitude. If the above vector products are at the level of machine precision, e.g.~$1\text{\sc{e}-}16$, this would mean that the norm of the residual is of the order of its square-root, here $1\text{\sc{e}-}8$. But this is the measure of the stopping criterion. Thus, if we ask for a high precision of the CG method, the algorithm might become numerically unstable, \Rev{depending on the machine precision. Such saturation of the error is an intrinsic property of the CG method, and   here we want to mention it just as a disclaimer. As described further below, we set the lower bound of the CG tolerance to the value $1\text{\sc{e}-}12$, i.e., as soon as this accuracy is reached, we consider the result as \enquote{numerically exact}. For this particular bound the method works stably on the machine that we used.} 
\end{itemize} 

In the following, we start with a description of the general test settings, which  will be common for both Algorithms~CG-IRLS and CG-IRLS-$\lambda$. Afterwards we independently analyze the speed of both methods and compare them with state of the art algorithms, namely IHT and FISTA. We respectively start with a single trial, followed by a speed-test on a variety of problems. We will also compare the performance of both CG-IRLS and CG-IRLS-$\lambda$ for the noiseless case
\HR{which leads to surprising results}.

\subsection{Test settings}
All tests are performed with MATLAB version R2014a. For the sake of faster tests (in some cases \HR{experiments run for several} days) and simplicity, we restrict ourselves to experiments with models defined by real numbers although everything can be similarly done 
\HR{over} the complex field. To exploit the advantage of fast matrix-vector multiplications and to allow high dimensional tests, we use randomly sampled partial discrete cosine transformation matrices $\Phi$. We perform tests in three different dimensional settings (later we will extend them to higher dimension) and choose different values $N$ of the dimension of the signal, the amount $m$ of measurements, the respective sparsity $k$ of the synthesized solutions, and the index $K$ in Algorithm (CG-)IRLS:
\begin{center} \begin{tabular}{|c|c|c|c|}
\hline
  & \textbf{Setting A} & \textbf{Setting B} & \textbf{Setting C} \\
  \hline
N & 2000 & 4000 & 8000\\
m & 800 & 1600 & 3200\\
k & 30 & 60 & 120\\
K & 50  & 100 & 200\\
\hline
\end{tabular}
\end{center}

For each of these settings, we draw at random a set of 100 synthetic problems on which a speed-test is performed. 
\HR{For each synthetic problem 
the support $\Lambda$ is determined by the first $k$ entries of a random permutation of the numbers $1,\ldots,N$.} Then we draw the sparse vector $x^*$ at random with entries $x^*_i \sim \mathcal{N}(0,1)$ for $i \in \Lambda$ and $x^*_{\Lambda^c}=0$, and a randomly row sampled \HR{normalized} discrete cosine matrix $\Phi$, where the full non-normalized discrete cosine matrix is given by
\[\Phi^{\text{full}}_{i,j}=\begin{cases}1,&i=1, j = 1,\ldots,N, \\ 
\sqrt{2}\cos \left(\frac{\pi(2j-1)(i-1)}{2N}\right),&2\leq i\leq N,1\leq j\leq N.\end{cases} \]
For a given noise vector $e$ of entries $e_i\sim\mathcal{N}(0,\sigma^2)$, we eventually obtain the measurements $y=\Phi x^* + e$. 
Later we need to specify the noise level and we will do so by fixing a signal to noise ratio. By assuming that $\Phi$ has the \emph{Restricted Isometry Property} of order $k$ (compare, e.g.,~\cite{fora13}), i.e., $\|\Phi z\|_{\ell_2} \sim \|z\|_{\ell_2}$, for all $z \in \mathbb R^N$ with $\# \supp(z) \leq k$, we can estimate the measurement signal to noise ratio by
\[\text{MSNR} \defleft \frac{\mathbb{E}(\|\Phi x^*\|_{\ell_2})}{\mathbb{E}(\|e\|_{\ell_2})} \sim \frac{\sqrt{k}}{\sqrt{m}\sigma}.\]
In practice, we set the MSNR first and choose the noise level $\sigma=\frac{\sqrt{k}}{\text{MSNR}\sqrt{m}}$. If $\text{MSNR} = \infty$, the problem is noiseless, i.e., $e=0$.
 \subsection{Algorithm~CG-IRLS}
 \label{sec:numAlgCGIRLS}
\paragraph{Specific settings.}
We restrict the maximal number of outer iterations to 30. Furthermore, we modify \eqref{eq:deftol}, so that the CG-algorithm also stops as soon as $\vectornorm{\rho^{n+1,i}}_{\ell_2} \leq 1\text{\sc{e}-}12$. As soon as the residual undergoes this particular threshold, we call the CG solution (numerically) \enquote{exact}. The $\varepsilon$-update rule is extended by imposing the lower bound $\varepsilon^n=\varepsilon^n \vee \varepsilon^{\min}$ where $ \varepsilon^{\min}=1\text{\sc{e}-}9/N$. The summable sequence $(a_n)_{n \in \mathbb N}$ in Theorem \ref{thm: convergence} is defined by  $a_n=100\cdot (1/2)^n$. 

As we define the synthetic tests by choosing the solution  $x^*$ of the linear system $\Phi x^* =y$ (here we assume $e=0$), we can use it to determine the error of the iterations $\|\tilde{x}^n - x^*\|_{\ell_2}$.

\paragraph{IRLS vs.~CG-IRLS} 
To get an immediate impression about the general behavior of CG-IRLS, we compare its performance in terms of accuracy and speed to IRLS, where the intermediate linear systems are solved exactly \HR{via Gaussian elimination (i.e., by the standard MATLAB backslash operator)}.  We choose IHT as a first order state of the art benchmark, to get a fair comparison with another method which can exploit fast matrix-vector multiplications. 

In this first single trial experiment, we choose an instance of setting B, and set $\tau=1$ for CG-IRLS and compare it to IRLS with different values of $\tau$. The result is presented in the left plot of Figure~\ref{fig:singletrialAlgB}.  
We show the decrease of the relative error in $\ell_2$-norm as a function of the computational time. 
One sees that the computational time of IRLS is significantly outperformed by CG-IRLS and by the exploitation of fast matrix-vector multiplications. The standard IRLS is not competitive in terms of computational time, even if we choose $\tau < 1$, which is known to yield super-linear convergence \cite{dadefogu10}. With increasing dimension of the problem, in general the advantage of using the CG method becomes even more significant. However CG-IRLS does not outperform yet IHT in terms of computational time. We also observe the expected numerical error saturation (as mentioned at the beginning of this section), which appears as soon as the accuracy falls below $1\text{\sc{e}-}13$.
\begin{figure}[ht!]
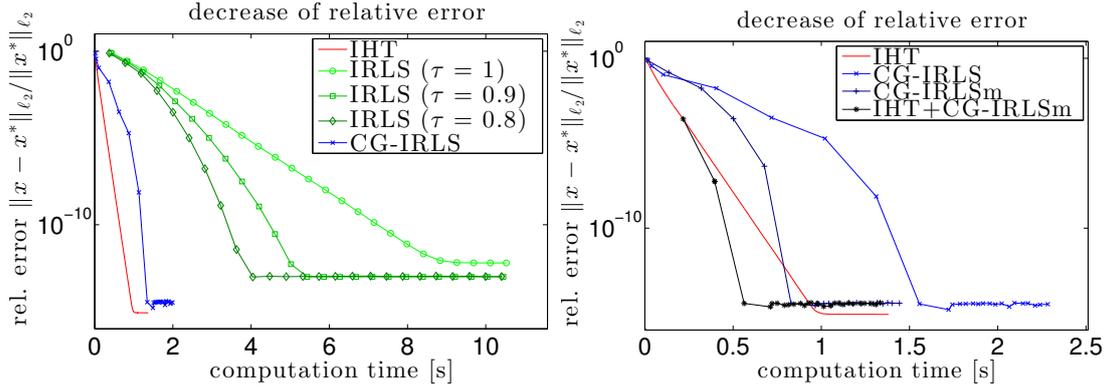

\begin{center}
\includegraphics[width=0.45\textwidth]{\figurespath{singletrialIRLS}}
\includegraphics[width=0.45\textwidth]{\figurespath{singletrialCGIRLS}}
\end{center}
\caption{Single trial of Setting B. Left: Relative error plotted against the computational time for IRLS[$\tau =1$] (light green, $\circ$),  IRLS[$\tau =0.9$] (green, $\square$),  IRLS[$\tau =0.8$] (dark green, $\Diamond$), CG-IRLS (blue, $\times$), and IHT (red, $-$). Right: Relative error plotted against computational time for CG-IRLS (blue, $\times$), CG-IRLSm (dark blue, $+$), IHT+CG-IRLSm (black, $*$), and IHT (red, $-$).}
\label{fig:singletrialAlgB}
\end{figure}

For this test, we set the parameter $\beta$ in the $\varepsilon$-update rule to 2. We comment on the choice of this particular parameter in a dedicated paragraph below.
\paragraph{Modifications to CG-IRLS}
As we have shown by a single trial in the previous paragraph,  \mbox{CG-IRLS} as it is presented in Section~\ref{sec:IRLSCGalg} is not able to outperform IHT. Therefore, we introduce the following practical modifications to the algorithm:
\begin{enumerate} 
\item We introduce the parameter \texttt{maxiter\_cg}, which defines the maximal number of inner CG iterations. Thus, the inner loop of the algorithm stops as soon as \texttt{maxiter\_cg} iterations were performed, even if the theoretical tolerance $\tol_n$ is not reached yet.
\item CG-IRLS includes a stopping criterion depending on $\tol_{n+1}$, which is only {\it implicitly} given as a function of $\varepsilon^{n+1}$ (compare Section~\ref{sec:resfunc}, and in particular formulas~\eqref{eq:deftol} and~\eqref{eq:defcnwn}), which in turn depends on the current $\tilde x^{n+1}$ by means of sorting
and a matrix-vector multiplication. To further reduce the computational cost of each iteration, we avoid the aforementioned operations by only updating $\tol_{n+1}$ outside the MCG loop, i.e., after the computation of $\tilde x^{n+1}$ with fixed $\tol_{n+1}$ we update $\varepsilon^{n+1}$ as in step 3 of Algorithm CG-IRLS and subsequently update $\tol_{n+2}$ which again is fixed for the computation of $\tilde{x}^{n+2}$.
\item The left plot of Figure~\ref{fig:singletrialAlgB} reveals that  in the beginning CG-IRLS reduces the error more slowly than IHT, and it gets faster after it reached a certain ball around the solution. Therefore, we use IHT as a warm up for CG-IRLS, in the sense that we apply a number \texttt{start\_iht} of IHT iterations to compute a proper starting vector for CG-IRLS.
\end{enumerate}
We call \emph{CG-IRLSm} the algorithm with modifications (i) and (ii), and \emph{IHT+CG-IRLSm} the algorithm with modifications (i), (ii), and (iii). We set $\texttt{maxiter\_cg}=\lfloor m/12\rfloor$, $\texttt{start\_iht}=150$, and we set $\beta$ to 0.5. If these algorithms are executed on the same trial as in the previous paragraph, we obtain the result which is shown on the right plot in Figure~\ref{fig:singletrialAlgB}. For this trial, the modified algorithms show a significantly reduced computational time with respect to the unmodified version and they \HR{now} converge faster than IHT. However, the introduction of the practical modifications (i)--(iii) does not necessarily comply anymore with the  assumptions of Theorem~\ref{thm: convergence}.  Therefore, we do not have rigorous convergence and recovery guarantees anymore 
and recovery might potentially fail more often.
In the next paragraph, we empirically investigate the failure rate and explore the performance of the different methods on a sufficiently large test set. \\

\Rev{In order to investigate the influence of the tolerance $\tol_n$ and 
the number of (inner) iterations of the MCG procedure performed along the IRLS iterations, we plot both quantities in Figure~\ref{fig:singletrialAlgB_itertol} for CG-IRLS, CG-IRLSm, and IHT+CG-IRLSm. Obviously the tolerance quickly decreases in any method. For IHT+CG-IRLSm also the number of MCG iterations decreases (until the method becomes unstable), while for the other two methods the number of MCG iterations first increases and then decreases again until the methods also become unstable. In CG-IRLSm, the number of MCG iterations is bounded by $\texttt{maxiter\_cg}=\lfloor m/12\rfloor$. 
This more economical behavior 
only slightly influences the approximation of the MCG solutions and leads to reduced computational time.
\begin{figure}[ht!]
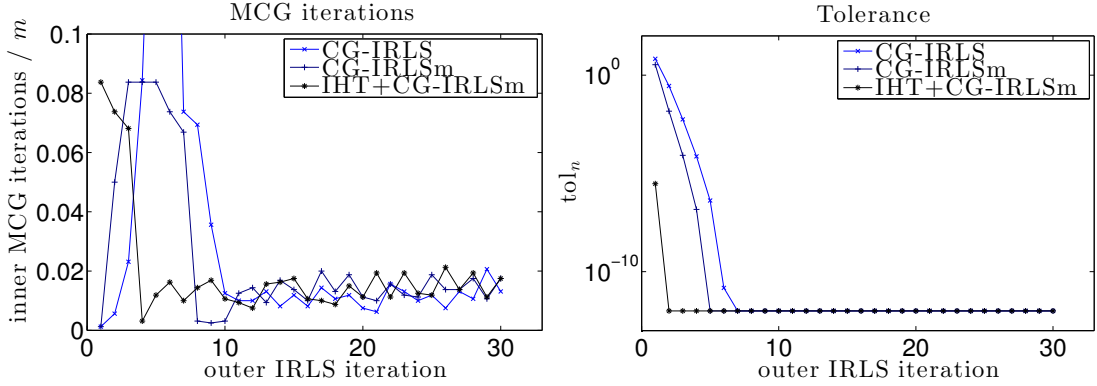

\begin{center}
\includegraphics[width=0.45\textwidth]{\figurespath{singletrialCGIRLS_iter}}
\includegraphics[width=0.45\textwidth]{\figurespath{singletrialCGIRLS_tol}}
\end{center}
\caption{Single trial of Setting B. Left: Number of MCG iterations (divided by $m$) plotted against the outer iteration number for CG-IRLS (blue, $\times$), CG-IRLSm (dark blue, $+$), IHT+CG-IRLSm (black, $*$). Right: Tolerance $\tol_n$ plotted against the outer iteration number for the same algorithms.}
\label{fig:singletrialAlgB_itertol}
\end{figure}}

Another natural modification to CG-IRLS \HR{consists in} the introduction of a preconditioner to compensate for the deterioriation of the condition number of $\Phi D_n \Phi^*$ as soon as $\varepsilon^n$ \HR{becomes} too small (when $w^n$ \HR{becomes very large}).
The matrix $\Phi \Phi^*$ is very well conditioned, while the matrix $\Phi D_n \Phi^*$ ``sandwiching'' $D_n$ 
\HR{becomes more ill-conditioned as $n$ gets larger,} and, unfortunately, it is hard to identify additional ``sandwiching'' preconditioners $P_n$ \HR{such that the matrix $P_n \Phi D_n \Phi^* P_n^*$ is suitably well-conditioned}. 
In the numerical experiments standard preconditioners failed to yield any significant improvement in terms of convergence speed. Hence, we 
\HR{refrained from introducing further preconditioners}. 
Instead, as we will show at the end of Subsection \ref{CGIRLSl}, a standard \HR{(Jacobi)} preconditioning 
of the matrix $$\left(\Phi^* \Phi  + \diag\left[\lambda\tau w_j^n\right]_{j=1}^N\right),$$ where the source of singularity is added to the product $\Phi^* \Phi$, leads to a dramatic improvement of computational speed.

\paragraph{Empirical test on computational time and failure rate}
In the following, we define a method to be ``successful'' if it is computing a solution $x$ for which the relative error $\|x - x^*\|_{\ell_2} / \|x^*\|_{\ell_2} \leq 1\text{\sc{e}-}13$. The computational time of a method is measured by the time it needs to produce the first iterate which reaches this accuracy. In the following, we present the results of a test which runs the methods CG-IRLS, CG-IRLSm, IHT+CG-IRLSm, and IHT on 100 trials of Setting A, B, and C respectively and $\tau\in\{1,0.9,0.8\}$. For values of $\tau<0.8$ the methods become unstable, due to the severe nonconvexity of the problem and 
\HR{it seems that good performance cannot} be reached \HR{below} this level. Therefore we do not investigate further these cases. 
Let us stress that IHT does not  depend on $\tau$. 

In each setting we check for each trial which methods \HR{succeeds or fails}. If all methods \HR{succeed}, we \HR{compare} the computational time, determine the fastest method, and \HR{count} the computational time of each method for the respective mean computational time. The results are shown in Figure~\ref{fig:EmpTestIRLSCG}. By analyzing the diagrams, we are able to distill the following observations:
\begin{itemize}
\item \HR{Especially} in Setting A and B,  CG-IRLSm and IHT+CG-IRLSm are better or comparable to IHT in terms of mean computational time 
and provide in most cases the fastest method. CG-IRLS performs much worse. The failure rate of all the methods is negligible here.
\item The gap in the computational time between all methods becomes larger when $N$ is larger. 
\item With increasing dimension of the problem, the advantage of using the modified CG-IRLS methods subsides, in particular in Setting C.

\item In the literature \cite{Chartrand07,Chartrand08,ChartrandYin08,dadefogu10} superlinear convergence is reported for $\tau<1$, and perhaps one of the most surprising outcomes is that the best results for all CG-IRLS methods are instead obtained for $\tau = 1$. This \HR{can probably be} explained by observing that superlinear convergence kicks in only in a rather small ball around the solution and hence does not necessarily improve the actual \HR{computation} time! 
\item 
\HR{Not only the computational performance,} but also the failure rate of the CG-IRLS based methods increases with decreasing $\tau$. However, as expected, CG-IRLS \HR{succeeds} in the convex case of $\tau = 1$. The failure of CG-IRLS for $\tau < 1$ can \HR{probably} be attributed to non-convexity. 
\end{itemize} 

\begin{figure}[ht!]
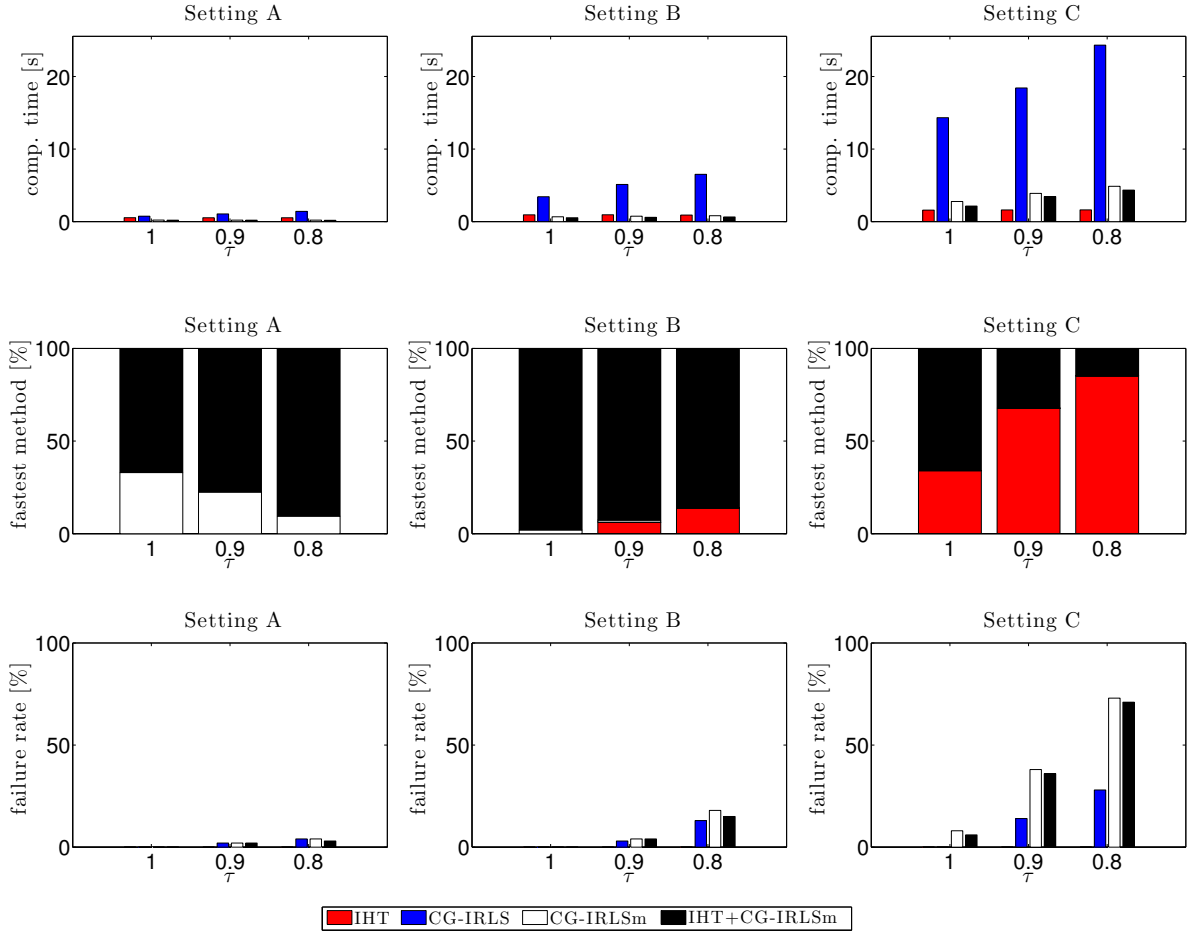

\begin{center}
\includegraphics[width=\textwidth]{\figurespath{speedtest_comptime}}
\includegraphics[width=\textwidth]{\figurespath{speedtest_fastest}}
\includegraphics[width=\textwidth]{\figurespath{speedtest_failure}}
\end{center}
\caption{Empirical test on Setting A, B, and C for the methods CG-IRLS (blue), CG-IRLSm (white), IHT+CG-IRLSm (black), and IHT (red). Upper: Mean computational time. Center: Fastest method (in \%). Lower: Failure rate (in \%).}
\label{fig:EmpTestIRLSCG}
\end{figure}

We conclude that CG-IRLSm and IHT+CG-IRLSm perform well for $\tau = 1$ and for the problem dimension $N$ within the range of 1000 -- 10000. They are even able to outperform IHT. 
However, by extrapolation of the numerical results IHT is expected to be faster for $N>10000$. (This is in compliance with the general folklore that first order methods should be preferred for higher dimension. However, as we \HR{will see} in Subsection \ref{CGIRLSl}, a proper preconditioning of CG-IRLS-$\lambda$ will win over IHT for dimensions $N \geq 10^5$!)
As soon as $N<1000$, direct methods \HR{such} as Gaussian elimination are faster than CG, and thus, one should use standard IRLS with $\tau<1$. 
\paragraph{Choice of $\beta$, $\texttt{maxiter\_cg}$, and $\texttt{start\_iht}$}
The numerical tests in the previous paragraph were preceded by a careful and systematic investigation of the tuning of the parameters $\beta$, $\texttt{maxiter\_cg}$, and $\texttt{start\_iht}$. While we fixed $\texttt{start\_iht}$ to 100, 150, and 200 for Setting A, B, and C respectively to produce a good starting value, we tried $\beta \in \{1/N, 0.01, 0.1, 0.5, 0.75, 1, 1.5, 2, 5, 10\}$, and $\texttt{maxiter\_cg}\in\{\lfloor m/8 \rfloor,\lfloor m/12 \rfloor,\lfloor m/16 \rfloor\}$ for each setting. The results of this parameter sensitivity study can be summarized as follows:
\begin{itemize}
\item The best computational time is obtained for $\beta \sim 1$. In particular the computational time is not depending substantially on $\beta$ in this order of magnitude. More precisely, for CG-IRLS the choice of $\beta =0.5$ and for (IHT+)CG-IRLSm the choice of $\beta = 2$ works best.
\item The choice of $\texttt{maxiter\_cg}$ very much determines the tradeoff between failure and speed of the method. The value $\lfloor m/12 \rfloor$ seems to be the best compromise. For a smaller value the failure rate becomes too high, for a larger value the method is too slow.
\end{itemize}
\paragraph{Phase transition diagrams.}
\HR{Besides} the empirical analysis of the speed of convergence, we \HR{also 
investigate} the robustness of CG-IRLS with respect to the achievable 
sparsity level for exact recovery of $x^*$. Therefore, we fix $N = 2000$ and we compute a phase transition diagram for IHT and CG-IRLS on a regular Cartesian $50\times 40$ grid, where one axis represents $m/N$ and the other represents $k/m$. For each grid point we plot the empirical success recovery rate, which is numerically realized by running both algorithms on 20 random trials. CG-IRLS or IHT is successful if it is able to compute a solution with a relative error of less than $1\text{\sc{e}-}4$ within 20 or 500 (outer) iterations respectively. Since we \HR{aim at simulating} a setting in which the sparsity $k$ is not known exactly, we set the parameter $K = 1.1\cdot k$ for both IHT and CG-IRLS. The interpolated plot is shown in Figure~\ref{fig:PT_diagram}. It turns out that CG-IRLS has a significantly higher success recovery rate than IHT for less sparse solutions. 
\begin{figure}[h!]
\begin{center}
\includegraphics[width=\textwidth]{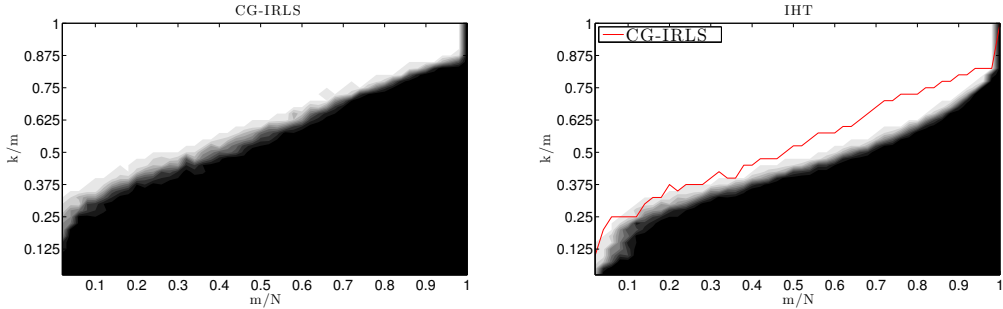}
\end{center}
\caption{Phase transition diagrams of IHT and CG-IRLS for $N=2000$. The recovery rate is presented in grayscale values from 0\% (white) up to 100\% (black). As a reference, in the right subfigure, the 90\% recovery rate level line of the CG-IRLS phase transition diagram is plotted to show more evidently the improved success rate of the latter algorithm.}
\label{fig:PT_diagram}
\end{figure}
\subsection{Algorithm CG-IRLS-$\lambda$}\label{CGIRLSl}
\paragraph{Specific settings}
We restrict the maximal number of outer iterations to 25. Furthermore, we modify \eqref{eq:stopcritR}, so that the CG-algorithm also stops as soon as $\vectornorm{\rho^{n+1,i}}_{\ell_2} \leq 1\text{\sc{e}-}16\cdot N^{3/2}m$. As soon as the residual undergoes this particular threshold, we call the CG solution (numerically) \enquote{exact}. The $\varepsilon$-update rule is extended by imposing the lower bound $\varepsilon^n=\varepsilon^n \vee \varepsilon^{\min}$ where $ \varepsilon^{\min}=1\text{\sc{e}-}9$. Additionally we propose to choose $\varepsilon^{n+1}\leq 0.8^n\varepsilon^n$, which practically turns out to increase dramatically the speed of convergence. The summable sequence $(a_n)_{n\in\N}$ in Theorem~\ref{thm:conv} is defined by setting $a_n=\sqrt{Nm}\cdot 
\HR{10^4} \cdot (1/2)^n$. We split our investigation into a noisy and a noiseless setting.

For the noisy setting we set $\text{MSNR} = 10$. According to~\HR{\cite{birits09,candes2009}}, 
we choose $\lambda = c \sigma\sqrt{m \log N}$ as a near-optimal regularization parameter, where we empirically determine $c = 0.48$. Since we work with relatively large values of $\lambda$ in the regularized problem~\eqref{eq:prob}, we cannot use the synthesized sparse solution $x^*$ as a reference for the convergence analysis. Instead, we need another reliable method to compute the minimizer of the functional. In the convex case of $\tau=1$, this is performed by the well-known and fast algorithm FISTA~\cite{beck09}, which shall also serve as a benchmark for the speed analysis. In the non-convex case of $\tau < 1$, there is no method which guarantees the computation of the global minimizer, thus, we have to omit a detailed speed-test in this case. However, we describe the behavior of Algorithm CG-IRLS-$\lambda$ for $\tau$ changing.  

If the problem is noiseless, i.e., $e=0$, the solution $x^{\lambda}$ of~\eqref{eq:prob} converges to the solution of~\eqref{eq:ltauproblem} for $\lambda\rightarrow 0$. Thus, we choose $\lambda=m\cdot 1\text{\sc{e}-}8$, and assume the synthesized sparse solution $x^*$ as a good proxy for the minimizer and a reference for the convergence analysis. (Actually, this can also be seen the other way around, i.e., we use the minimizer $x^\lambda$ of the regularized functional to compute a good approximation to $x^*$.) 
\HR{It turns out that} for $\lambda \approx 0$, as we comment below in more detail, FISTA is basically of no use. 
\paragraph{CG-IRLS-$\lambda$ vs.~IRLS-$\lambda$}
As in the previous subsection, we 
\HR{first show} that the CG-method within IRLS-$\lambda$ \HR{leads to} significant improvements in terms of the computational speed. Therefore we choose a noisy trial of Setting B, and compare the computational time of the methods IRLS-$\lambda$, CG-IRLS-$\lambda$, and FISTA. The result is presented on the left plot of Figure~\ref{fig:singletriallambda}. We observe, that CG-IRLS-$\lambda$ computes the first iterations in much less time than IRLS-$\lambda$, but due to bad conditioning of the inner CG problems it performs much worse afterwards. Furthermore, 
\HR{as may be expected}, the algorithm is not suitable to compute a highly accurate solution. For the computation of a solution with a relative error in the order of $1\text{\sc{e}-}3$, CG-IRLS-$\lambda$ outperforms FISTA. 
FISTA is able to compute \HR{highly accurate solutions}, 
\HR{but} a solution with a relative error of $1\text{\sc{e}-}3$ should be sufficient in most applications because the goal in general is not to compute the minimizer of the Lagrangian functional but an approximation of the sparse signal.  
\begin{figure}[ht]
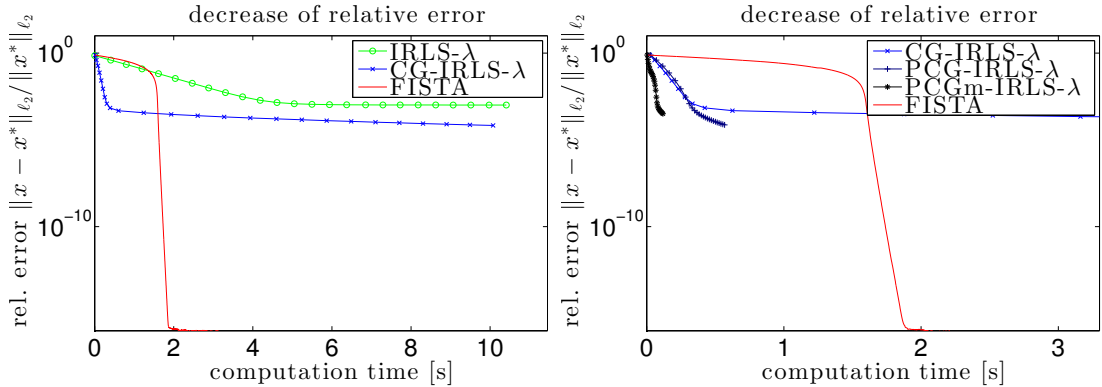

\begin{center}
\includegraphics[width=0.45\textwidth]{\figurespath{singletrialIRLSlambda}}
\includegraphics[width=0.45\textwidth]{\figurespath{singletrialCGIRLSlambda}}
\end{center}
\caption{Single trial of Setting B. Left: Relative error plotted against the computational time for IRLS-$\lambda$ (light green, $\circ$),  CG-IRLS-$\lambda$ (blue, $\times$), and FISTA (red, $-$). Right: Relative error plotted against computational time for CG-IRLS-$\lambda$ (blue, $\times$), PCG-IRLS-$\lambda$ (dark blue, $+$), PCGm-IRLS-$\lambda$ (black, $*$), and FISTA (red, $-$).}
\label{fig:singletriallambda}
\end{figure}
\paragraph{Modifications to CG-IRLS-$\lambda$}
To further decrease the computational time of CG-IRLS-$\lambda$, we propose the following modifications:
\begin{enumerate}
\item To overcome the bad conditioning in the CG loop, we precondition the matrix $A_n =\Phi^* \Phi  + \diag\left[\lambda\tau w_j^n\right]_{j=1}^N$ by means of the Jacobi preconditioner, i.e., we pre-multiply the linear system by the inverse of its diagonal, $\left(\diag A_n\right)^{-1}$, which is a very efficient operation in practice.
\item We introduce the parameter $\texttt{maxiter\_cg}$ which defines the maximal number of inner CG iterations \HR{and is set to the value $\texttt{maxiter\_cg}= 4$
in the following.}
\end{enumerate}
\HR{The algorithm with modification (i) is called PCG-IRLS-$\lambda$, and the one with modification (i) and (ii) PCGm-IRLS-$\lambda$.} 
\HR{We} run these algorithms on the same trial of Setting B as in the previous paragraph. The respective result is shown on the right plot of Figure~\ref{fig:singletriallambda}. This time, preconditioning effectively yields a strong decrease of computational time, especially in the final iterations where $A_n$ is badly conditioned. Furthermore, modification (ii) importantly increases the performance of the proposed algorithm also in the initial iterations. However, again we have to take into consideration that we may violate the assumptions of Theorem~\ref{thm:conv} \HR{so that convergence is not guaranteed anymore and failure rates might potentially increase}.
\Rev{In the two paragraphs below that are entitled \emph{Empirical test on computational time and failure rate with noisy/noiseless data}}, we present simulations on noisy and noiseless data, which give a more precise picture of the speed and failure rate of the previously introduced methods in comparison to FISTA and IHT.

\Rev{We investigate the influence of the tolerance $\tol_n$ and the number of (inner) iterations of the CG procedure performed along the IRLS iterations
in Figure~\ref{fig:singletrialAlgB_itertol_lambda} for CG-IRLS-$\lambda$, PCG-IRLS-$\lambda$, and PCGm-IRLS-$\lambda$. We see that the methods do not differ much in terms of the tolerance. In particular CG-IRLS-$\lambda$ and PCG-IRLS-$\lambda$ have nearly the same sequence of $\tol_n$, however, due to the bad conditioning, the number of inner CG iterations tremendously increases with growing number of outer IRLS iterations in CG-IRLS-$\lambda$. In contrast, the number of inner CG iterations in PCG-IRLS-$\lambda$ stays very low. A bound on the CG iterations of $\texttt{maxiter\_cg} = 4$ does only very slightly change the behavior of $\tol_n$ in PCGm-IRLS-$\lambda$ and leads to a further advantage in the computational time, as can be seen in Figure~\ref{fig:singletriallambda}.
\begin{figure}[ht!]
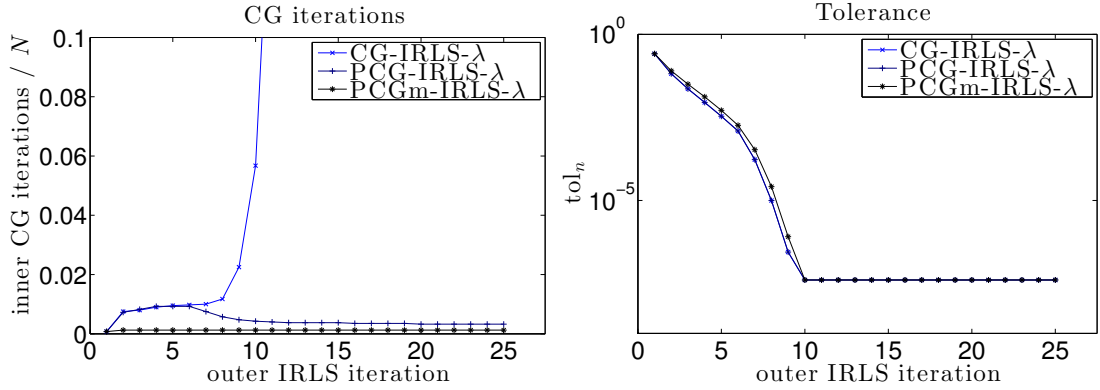

\begin{center}
\includegraphics[width=0.45\textwidth]{\figurespath{singletrialCGIRLSlambda_iter}}
\includegraphics[width=0.45\textwidth]{\figurespath{singletrialCGIRLSlambda_tol}}
\end{center}
\caption{Single trial of Setting B. Left: Number of CG iterations (divided by $N$) plotted against the outer iteration number for CG-IRLS-$\lambda$ (blue, $\times$), PCG-IRLS-$\lambda$ (dark blue, $+$), PCGm-IRLS-$\lambda$ (black, $*$). Right: Tolerance $\tol_n$ plotted against the outer iteration number for the same algorithms.}
\label{fig:singletrialAlgB_itertol_lambda}
\end{figure}}

\paragraph{Empirical test on computational time and failure rate with noisy data}
In the previous paragraph, we observed that the CG-IRLS-$\lambda$ methods are only computing efficiently solutions with a low relative error. Thus we now focus on this setting and compare the three methods PCG-IRLS-$\lambda$, PCGm-IRLS-$\lambda$, and FISTA with respect to their computational time and failure rate in recovering solutions with a relative error of $1\text{\sc{e}-}1$, $1\text{\sc{e}-}2$, and $1\text{\sc{e}-}3$. We only consider the convex case $\tau = 1$. Similarly to the procedure in Section~\ref{sec:numAlgCGIRLS}, we run these algorithms on 100 trials for each setting with the respectively chosen values of $\lambda$. In Figure~\ref{fig:EmpTestIRLSCGlambda} the upper bar plot shows the result for the mean computational time and the lower stacked bar plot shows how often a method was the fastest one. We do not present a plot of the failure rate since none of the methods failed at all. By means of the plots, we demonstrate that both PCG-IRLS-$\lambda$, and PCGm-IRLS-$\lambda$ are faster than FISTA, while PCGm-IRLS-$\lambda$ always performs best.
\begin{figure}[ht!]
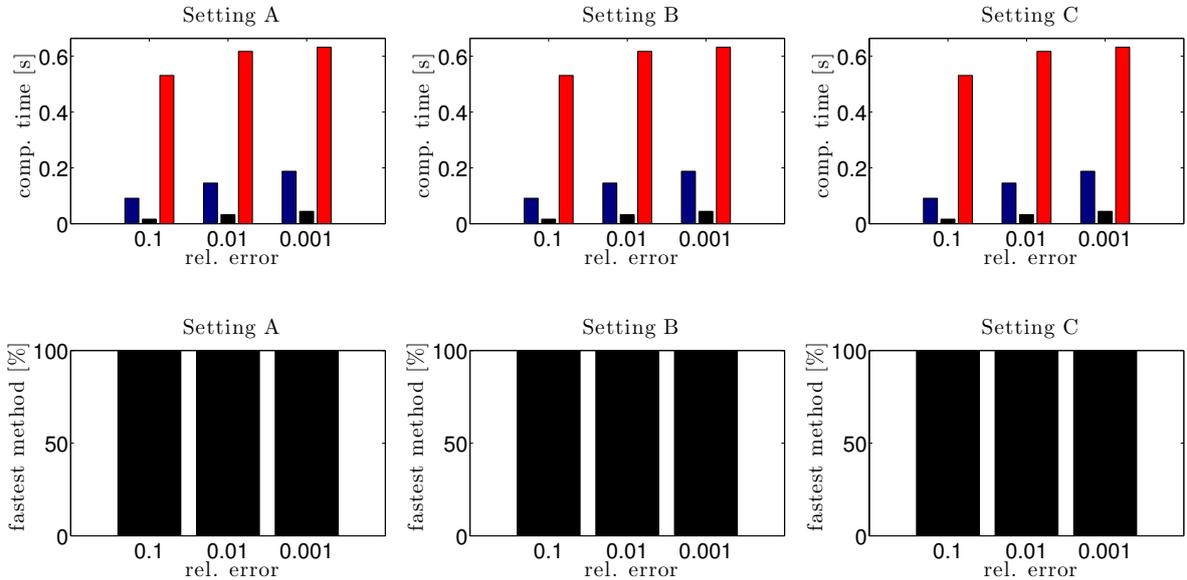

\begin{center}
\includegraphics[width=\textwidth]{\figurespath{speedtest_comptime_lambda}}
\includegraphics[width=\textwidth]{\figurespath{speedtest_fastest_lambda}}
\end{center}
\caption{Empirical test on Setting A, B, and C for the methods PCG-IRLS-$\lambda$ (blue), PCGm-IRLS-$\lambda$ (black), and FISTA (red). Upper: Mean computational time. Lower: Fastest method (in \%).}
\label{fig:EmpTestIRLSCGlambda}
\end{figure}

\paragraph{Empirical test on computational time and failure rate with noiseless data}
In the noiseless case, we compare the computational time of FISTA and PCGm-IRLS-$\lambda$ to IHT and IHT+CG-IRLSm. We set $\texttt{maxiter\_cg}=40$ for PCGm-IRLS-$\lambda$. In a first test, we run these algorithms on one trial of Setting A, and C respectively, and plot the results in Figure~\ref{fig:singletrialall}. \\
\begin{figure}[ht!]
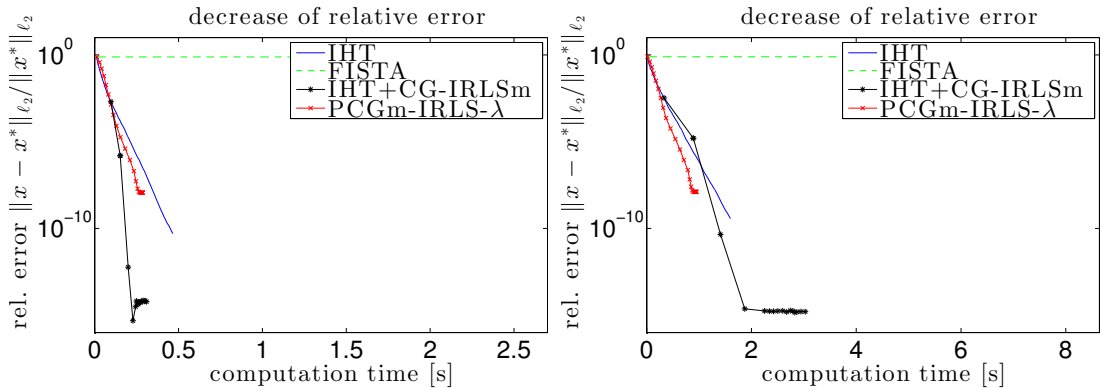

\begin{center}
\includegraphics[width=0.45\textwidth]{\figurespath{singletrialallcompare_SettingA}}
\includegraphics[width=0.45\textwidth]{\figurespath{singletrialallcompare_SettingC}}
\end{center}
\caption{Left: Setting A. Right: Setting C. Comparison of IHT (blue, $-$),  FISTA (green, $--$), IHT+CG-IRLSm (black, $*$), and PCGm-IRLS-$\lambda$ (red, $\times$).}
\label{fig:singletrialall}
\end{figure}
As already mentioned, FISTA is not suitable  \HR{for small values of $\lambda$ on the order of $m \cdot 1\text{\sc{e}-}8$} and converges then extremely slowly, but PCGm-IRLS-$\lambda$ can compete with the remaining methods. IHT+CG-IRLSm is in some settings able to outperform IHT, at least when a high accuracy is needed. \HR{PCGm-IRLS-$\lambda$ is always at least as fast as IHT with increasing relative performance gain for increasing dimensions}. 
\HR{This observation suggests the conjecture that PCGm-IRLS-$\lambda$ provides the fastest method also in rather high dimensional problems}. To validate this hypothesis numerically, we introduce two new high dimensional settings (to reach higher dimensionalities and retaining \HR{low computation times} for the extensive tests 
\HR{it is again very beneficial to use} the real cosine transform as a model for $\Phi$):
\begin{center} \begin{tabular}{|c|c|c|}
\hline
  & \textbf{Setting D} & \textbf{Setting E} \\
  \hline
N & 100000 & 1000000 \\
m & 40000  & 400000 \\
k & 1500   & 15000 \\
K & 2500   & 25000 \\
\hline
\end{tabular}
\end{center}
We run \HR{the most promising algorithms} IHT and PCGm-IRLS-$\lambda$ on a trial of the large scale settings D and E. The result, which is plotted in Figure~\ref{fig:singletrialallDE}, shows that PCGm-IRLS-$\lambda$ is able to outperform IHT \HR{in these settings} 
\HR{unless one requires} an extremely low relative error ($\leq1\text{\sc{e}-}8$), because of the error saturation effect. We confirm this outcome in a test on 100 trials for Setting D and E and present the result in Figure~\ref{fig:EmpTestallcompare}.
\begin{figure}[ht!]
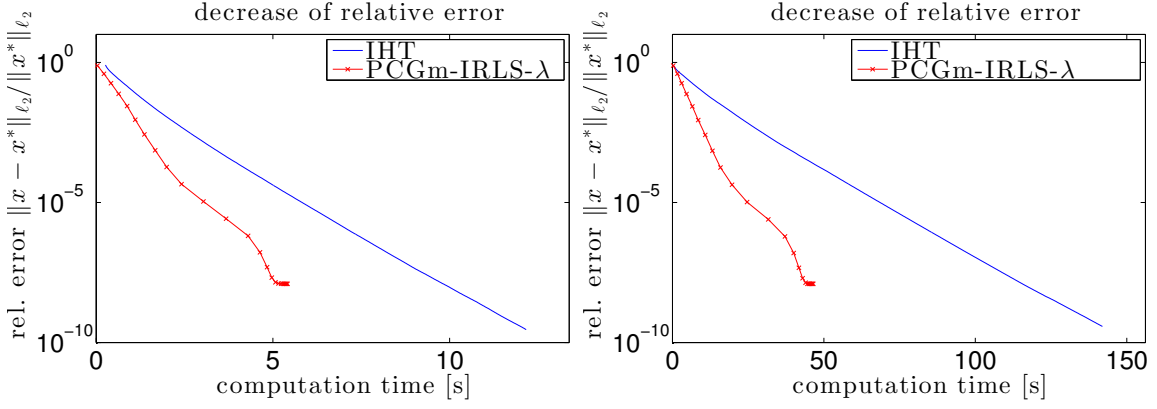

\begin{center}
\includegraphics[width=0.47\textwidth]{\figurespath{singletrialallcompare2_100k}}
\includegraphics[width=0.47\textwidth]{\figurespath{singletrialallcompare2_1m}}
\end{center}
\caption{Left: Setting D. Right: Setting E. Comparison of IHT (blue, $-$), and PCGm-IRLS-$\lambda$ (red, $\times$).}
\label{fig:singletrialallDE}
\end{figure}

\begin{figure}[ht!]
\begin{center}
\includegraphics[width=\textwidth]{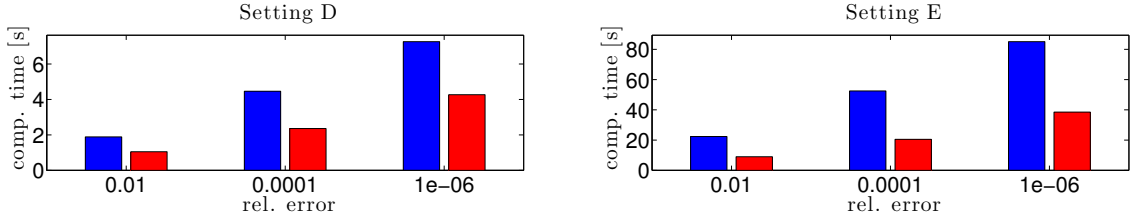}
\end{center}
\caption{Empirical test on the mean computational time of Setting D and E for the methods IHT (blue), and PCGm-IRLS-$\lambda$ (red).}
\label{fig:EmpTestallcompare}
\end{figure}

\paragraph{Dependence on $\tau$.}
In the last experiment of this paper, we are interested in the influence of the parameter $\tau$. Of course, changing $\tau$ also means modifying the problem \HR{resulting
in a different minimizer}. \HR{Due to non-convexity also spurious local minimizers may appear.} 
Therefore, we do not compare the speed of the method to FISTA. In Figure~\ref{fig:difftau}, we show the performance of Algorithm PCGm-IRLS-$\lambda$ for a single trial of Setting C and the parameters $\tau\in\{1,0.9,0.8,0.7\}$ for the noisy and noiseless setting. As reference for the error analysis, we \HR{choose} the sparse synthetic solution $x^*$, which is actually not the minimizer here. 

In both the noisy and noiseless setting, using a parameter $\tau < 1 $ improves the computational time of the algorithm. In the noiseless case, $\tau = 0.9$ seems to be a good choice, smaller values do not improve the performance. \HR{In contrast}, in the noisy setting the computational time decreases with decreasing $\tau$.
\begin{figure}[ht]
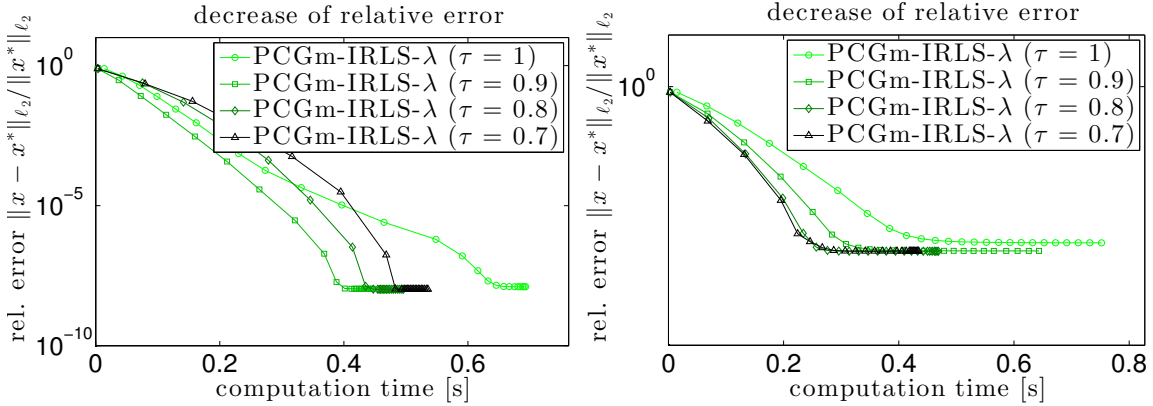

\includegraphics[width=0.47\textwidth]{\figurespath{singletrialcomp_tau}}
\includegraphics[width=0.47\textwidth]{\figurespath{singletrialcomp_tau_noisy}}
\caption{Results of Algorithm PCGm-IRLS-$\lambda$ for a single trial of Setting C for different values of $\tau$ with noise (right) and without noise (left).}
\label{fig:difftau}
\end{figure}

\begin{appendix}

\section{Proof of Lemma~\ref{lem: tangens}}
\label{app:lemmafe}
	\textquotedblleft$\Rightarrow$\textquotedblright (in the case $0 < \tau \leq 1 $) \\ 
	Let $x = x^{\etil,1}$ or $x\in\mathcal{X}_{\etil,\tau}(y)$, and $\eta\in\mathcal{N}_{\Phi}$ arbitrary. Consider the function
	\begin{equation*}
		G_{\etil,\tau}(t) \defleft f_{\etil,\tau}\left( x + t\eta \right) - f_{\etil,\tau}\left( x  \right)
	\end{equation*}
	with its first derivative 
	\begin{equation*}
		G^{\prime}_{\etil,\tau}(t) = \tau\sum\limits_{i=1}^{N}\frac{x_{i}\eta_{i} +t\eta_{i}^2}{\left[|x_{i} + t\eta_{i}|^{2} + \etil^{2}\right]^{\frac{2-\tau}{2}}}.
	\end{equation*}
	Now $G_{\etil,\tau}(0) = 0$ and from the minimization property of $f_{\etil,\tau}(x)$, $G_{\etil,\tau}(t) \ge 0$. Therefore,
	\begin{equation*}
		0 = G^{\prime}_{\etil,\tau}(0) = \sum\limits_{i=1}^{N}\frac{x_{i}\eta_{i}}{\left[x_{i}^{2} + \etil^{2}\right]^{\frac{2-\tau}{2}}} = \left\langle x,\eta\right\rangle_{\hat{w}(x,\etil,\tau)}.
	\end{equation*}
	\textquotedblleft$\Leftarrow$\textquotedblright (only in the case $\tau=1$)\\ 
	Now let $x\in\sF{\Phi}$ and $\left\langle x,\eta\right\rangle_{\hat{w}(x,\etil,1)} = 0$ for all $\eta\in\mathcal{N}_{\Phi}$. We want to show that $x$ is the minimizer of $f_{\etil,1}$ in $\sF{\Phi}$. 
	Consider the convex univariate function $g(u)\defleft [u^{2} + \etil^{2}]^{1/2}$. For any point $u_{0}$ we have from convexity that
	\begin{equation*}
		[u^{2} + \etil^{2}]^{1/2} \geq [u_{0}^{2} + \etil^{2}]^{1/2} + [u_{0}^{2} + \etil^{2}]^{-1/2}u_{0}(u-u_{0})
	\end{equation*} 
	because the right-hand-side is the linear function which is tangent to $g$ at $u_{0}$. It follows, that for every point $v\in\sF{\Phi}$ we have
	\begin{equation*}
		f_{\etil,1}(v) \geq f_{\etil,1}(x) + \sum\limits_{i=1}^{N}{[x_{i}^{2} + \etil^{2}]^{-1/2}x_{i}(v_{i} - x_{i})} = f_{\etil,1}(x) + \left\langle x, v-x\right\rangle_{\hat{w}(x,\etil,1)} = f_{\etil,1}(x),
	\end{equation*}
	where we have used the orthogonality condition and the fact that $(v - x) \in \mathcal{N}_{\Phi}$. Since $v$ was chosen arbitrarily, $x = x^{\etil,1}$ as claimed.
\end{appendix}

\begin{acknowledgements}
 Massimo Fornasier  acknowledges the support of the ERC-Starting Grant HDSPCONTR ``High-Dimensional Sparse Optimal Control'' and the DFG Project ``Optimal Adaptive Numerical Methods for p-Poisson Elliptic equations''. Steffen Peter acknowledges the support of the Project ``SparsEO: Exploiting the Sparsity in Remote Sensing for Earth Observation'' funded by Munich Aerospace. Holger Rauhut would like to thank the
 European Research Council (ERC) for support through the Starting Grant StG 258926 SPALORA (Sparse and Low Rank Recovery) and the Hausdorff Center for Mathematics
 at the University of Bonn where this project has started.
 \end{acknowledgements}


\bibliography{IRLSCG.bbl}

\end{document}